\documentclass[11pt]{article}

\usepackage[english]{babel}
\usepackage{amssymb}
\usepackage{amsmath}
\usepackage{hyperref}
\usepackage{enumerate}
\usepackage{amsthm}
\usepackage{amsfonts}
\usepackage{listings}
\usepackage[curve]{xypic}

\usepackage{MnSymbol}

\DeclareMathSymbol{\mlq}{\mathord}{operators}{``}
\DeclareMathSymbol{\mrq}{\mathord}{operators}{`'}

\title{Fibered Multiderivators and (co)homological descent}
\date{October 20, 2015}
\author{Fritz H\"ormann\\ Mathematisches Institut, Albert-Ludwigs-Universit\"at Freiburg}

\usepackage[numbers,sort&compress]{natbib}

\usepackage{color}
\usepackage[headsepline,footsepline]{scrpage2}

\usepackage{tikz}
\newcommand*\numcirc[1]{\tikz[baseline=(char.base)]{
            \node[shape=circle,draw,inner sep=2pt] (char) {#1};}}

\newtheorem{SATZ}{Theorem}[subsection]
\newtheorem{HAUPTSATZ}[SATZ]{Main Theorem}
\newtheorem{LEMMA}[SATZ]{Lemma}
\newtheorem{DEF}[SATZ]{Definition}
\newtheorem{PROP}[SATZ]{Proposition}

\newtheorem{BEISPIEL}[SATZ]{Example}
\newtheorem{FRAGE}[SATZ]{Question}

\newtheorem{KOR}[SATZ]{Corollary}

\newtheorem{BEM}[SATZ]{Remark}

\newtheoremstyle{bare}        % name
  {}            % Space above, empty = `usual value'
  {}            % Space below
  {\normalfont}                 % Body font (\normalfont)
  {}                            % Indent amount (empty = no indent, \parindent = para indent)
  {\bfseries}                   % Thm head font
  {}                            % Punctuation after thm head
  {.0em}                           % Space after thm head: " " = normal interword space;
  {\thmnumber{#2}#1{\thmnote{ \normalfont(#3)}}. } % Thm head spec \normalfont\textsc{}}

\theoremstyle{bare}
\newtheorem{PAR}[SATZ]{}

\newcommand{\iso}{\stackrel{\sim}{\longrightarrow}}

\newcommand{\Z}{ \mathbb{Z} }

\newcommand{\N}{ \mathbb{N} }

\newcommand{\DD}{ \mathbb{D} }
\newcommand{\EE}{ \mathbb{E} }

\newcommand{\SSS}{ \mathbb{S} }

\DeclareMathOperator{\colim}{colim}

\DeclareMathOperator{\Cone}{Cone}
\DeclareMathOperator{\id}{id}

\DeclareMathOperator{\op}{op}
\DeclareMathOperator{\cart}{cart}
\DeclareMathOperator{\cocart}{cocart}

\DeclareMathOperator{\Iso}{Iso}

\DeclareMathOperator{\Hom}{Hom}
\DeclareMathOperator{\Ho}{Ho}

\DeclareMathOperator{\Ob}{Ob}
\DeclareMathOperator{\Mor}{Mor}

\DeclareMathOperator{\can}{can.}

\DeclareMathOperator{\pr}{pr}

\DeclareMathOperator{\Cat}{Cat}
\DeclareMathOperator{\Dia}{Dia}
\DeclareMathOperator{\Dir}{Dir}
\DeclareMathOperator{\Inv}{Inv}
\DeclareMathOperator{\Catf}{Catf}
\DeclareMathOperator{\Catlf}{Catlf}
\DeclareMathOperator{\Dirf}{Dirf}
\DeclareMathOperator{\Invf}{Invf}
\DeclareMathOperator{\Pos}{Pos}
\DeclareMathOperator{\Posf}{Posf}
\DeclareMathOperator{\Dirpos}{Dirpos}
\DeclareMathOperator{\Invpos}{Invpos}
\DeclareMathOperator{\Dirlf}{Dirlf}
\DeclareMathOperator{\Invlf}{Invlf}
\DeclareMathOperator{\Cof}{Cof}
\DeclareMathOperator{\Fib}{Fib}

\DeclareMathOperator{\dia}{dia}
\DeclareMathOperator{\cosk}{cosk}

\begin{document}

\maketitle



\section*{Abstract}

The theory of derivators enhances and simplifies the theory of triangulated categories. In
this article a notion of fibered (multi-)derivator is developed, which similarly enhances fibrations of (monoidal) triangulated categories. 
We present a theory of cohomological as well as homological descent in this language. The main motivation is
a descent theory for Grothendieck's six operations. 

\section*{Introduction}

\subsection*{Grothendieck's six functors and descent}

Let $\mathcal{S}$ be a category, for instance a suitable category of schemes, topological spaces, analytic manifolds, etc.
A Grothendieck six functor formalism on $\mathcal{S}$ consists of a collection of (derived) categories $\mathcal{D}_S$, one for each ``base space'' $S$ in $\mathcal{S}$
with the following six types of operations:
\vspace{0.5cm}
\[
\begin{array}{lcrp{1cm}l}
f^* & &f_* &  & \text{\em for each $f$ in $\Mor(\mathcal{S})$}  \\
\\
f_! & &f^! & & \text{\em for each $f$ in $\Mor(\mathcal{S})$} \\
\\
\otimes & & \mathcal{HOM} & & \text{\em in each fiber $\mathcal{D}_S$}
\end{array}
\]
\vspace{0.2cm}

The fiber $\mathcal{D}_S$ is, in general, a {\em derived} category of ``sheaves'' over $S$, for example coherent sheaves, $l$-adic sheaves, abelian sheaves, $D$-modules, motives, etc. 
The functors on the left hand side are left adjoints of the functors on the right hand side. 
The functor $f_!$ and its right adjoint $f^!$ are called ``push-forward with proper support'', and ``exceptional pull-back'', respectively.
The six functors come along with a bunch of compatibility isomorphisms between them (cf.\@ \ref{LEMMA6FU}) and it is not easy to make their axioms really precise. 
In an appendix to this article, we explain that one quite simple precise definition is the following:

\vspace{0.2cm}

{\bf Definition \ref{DEF6FU}. }{\em Let $\mathcal{S}$ be a category with fiber products. 
A (symmetric) Grothendieck six-functor-formalism on $\mathcal{S}$ is a bifibration of (symmetric) 2-multicategories with 1-categorical fibers
\[ p: \mathcal{D} \rightarrow \mathcal{S}^{\mathrm{cor}} \]
where $\mathcal{S}^{\mathrm{cor}}$ is the symmetric 2-multicategory of correspondences in $\mathcal{S}$ (cf.\@ Definition \ref{DEFMULCOR}).}

\vspace{0.3cm}
 
From such a bifibration we obtain the operations $f_*$, $f^*$ (resp.\@ $f^!$, $f_!$) as pull-back and push-forward along the correspondences
\[ \vcenter{\xymatrix{ 
 & \ar[ld]_f X \ar@{=}[rd] &\\
 Y & ; & X   } }\quad \text{and} \quad
 \vcenter{\xymatrix{ 
 & \ar@{=}[ld] X \ar[rd]^f &\\
 X & ; & Y ,  }}
 \]
respectively. We get $\mathcal{E} \otimes \mathcal{F}$ for objects $\mathcal{E}, \mathcal{F}$ above $X$ as the target of a Cartesian $2$-ary multimorphism 
from the pair $\mathcal{E}, \mathcal{F}$ over the correspondence
\[ \vcenter{\xymatrix{
& & X \ar@{=}[lld] \ar@{=}[ld] \ar@{=}[rd] &  \\
X & X & ; & X. }}
\]
Given such a six-functor-formalism and a simplicial resolution $\pi: U_\bullet \rightarrow S$ of a space $S \in \mathcal{S}$ (for example arising from a \v{C}ech cover w.r.t.\@ a suitable Grothendieck topology)
\[ \xymatrix{ \cdots \ar@<-1.5ex>[r]\ar@<-0.5ex>[r]\ar@<0.5ex>[r]\ar@<1.5ex>[r]  &  U_2 \ar@<1ex>[r]\ar@<0ex>[r]\ar@<-1ex>[r] & U_1  \ar@<0.5ex>[r] \ar@<-0.5ex>[r]  & U_0,  } \]
and given an object $\mathcal{E}$ in $\mathcal{D}_S$, one can construct complexes in the category $\mathcal{D}_S$:
\[ \xymatrix{ \cdots \ar[r] & \pi_{2,!} \pi_2^{!} \mathcal{E} \ar[r] &\pi_{1,!} \pi_1^{!} \mathcal{E} \ar[r] &\pi_{0,!} \pi_0^{!} \mathcal{E} \ar[r] & 0   } \]
\[ \xymatrix{ \cdots & \ar[l] \pi_{2,*} \pi_2^{*} \mathcal{E} & \ar[l] \pi_{1,*} \pi_1^{*} \mathcal{E} & \ar[l] \pi_{0,*} \pi_0^{*} \mathcal{E} & \ar[l]  0  } \]
The first question of homological (resp. cohomological) descent is whether the hyper(co)homology of these complexes recovers the homology (resp.\@ cohomology) of
$\mathcal{E}$.
Without a suitable enhancement of the situation, this question, however, does not make sense because a double complex, once considered as a complex in the derived category, looses the information of the homology of its total complex.  There are several remedies for this problem. Classically, if at least the $\pi_i^*$ are derived functors and $\mathcal{E}$ is acyclic w.r.t.\@ them, one can derive the whole construction to get a coherent double complex. This does not work, however, for the functors $f_!$, $f^!$ which are often only constructed on the derived category.
One possibility is to consider enhancements of the triangulated categories in question as DG-categories or $\infty$-categories. 
In this article, we have worked out a different approach based on Grothendieck's idea of derivators which is, perhaps, conceptually even simpler. 
It is sufficiently powerful to glue the six functors and define them for morphisms between stacks, or even higher stacks. 
Like the $\infty$-categorical approach, it is also very general, not being restricted to the stable case.
\subsection*{Fibered multiderivators}

The notion of triangulated category developed by Grothendieck and Verdier in the 1960's, as successful as it has been, is not sufficient for many purposes, 
 for both practical reasons (certain natural constructions cannot be performed) as well as
for theoretical reasons (the axioms are rather involved and lack conceptual clarity). Grothendieck much later \cite{Gro91} and Heller independently, with the notion of {\bf derivator}, proposed a marvelously simple remedy
to both deficiencies. The basic observation is that all problems mentioned above are based on the following fact: Consider
a category $\mathcal{C}$ and a class of morphisms $\mathcal{W}$ (quasi-isomorphisms, weak equivalences, etc.) which one would like to become isomorphisms.
Then {\em homotopy limits and colimits} w.r.t.\@ $(\mathcal{C}, \mathcal{W})$ cannot be reconstructed once passed to the homotopy category $\mathcal{C}[\mathcal{W}^{-1}]$ (for example a derived category, or the homotopy category of a model category). Examples of homotopy (co)limits are the cone (required to exist in a triangulated category in a brute-force way, but not functorially!) and
the total complex of a complex of complexes (totally lost in the derived category). Furthermore, very basic
and intuitive properties of homotopy limits and colimits, and more general Kan extensions, not only determine the additional structure (triangles, shift functors) on a triangulated category but also {\em imply} all of its rather involved axioms. This idea has been successfully worked out by Cisinski, Groth, Grothendieck, Heller, Maltsiniotis, and others. We refer to the introductory article \cite{Gro13} for an overview.

The purpose of this article is to propose a notion of fibered (multi-)derivator which enhances the notion of a fibration of (monoidal) triangulated categories 
in the same way as the notion of usual derivator enhances the notion of triangulated category. We emphasize that this new context 
is very well suited to reformulate (and reprove the theorems of) the classical theory of cohomological descent
and to establish a completely dual 
theory of homological descent which should be satisfied by the $f_!, f^!$-functors. 

\subsection*{(Co)homological descent with fibered derivators}

Pursuing the idea of derivators, there is a neat conceptual solution to the problem of (co)homological descent: Analogously to a derivator which associates a (derived) category with each diagram
shape $I$, we should consider a (derived) category $\DD(I, F)$ for each diagram of correspondences $F: I \rightarrow \mathcal{S}^{\mathrm{cor}}$. Then, given a simplicial resolution $\pi: U_\bullet \rightarrow S$ as before, considered as a morphism 
$p:  (\Delta^{\op}, U_\bullet) \rightarrow (\cdot, S)$ of diagrams in $\mathcal{S}^{cor}$, resp.\@ $i:  (\Delta, (U_\bullet)^{\op}) \rightarrow (\cdot, S)$ in a dual diagram category (cf.\@~\ref{DEFDIAOP}), the question becomes:

\begin{enumerate}
\item[Q1:] Does the corresponding pull-back $i^*$ have a right adjoint $i_*$, 
respectively does $p^*$ have a left adjoint $p_!$ (a straightforward generalization of the question of existence of homotopy (co)limits in usual derivators!)
and is the corresponding unit $\id \rightarrow i_*i^*$ (resp.\@ counit $p_! p^* \rightarrow \id$) an isomorphism? 
\end{enumerate}

Instead of, however, taking an association $(I, F) \mapsto \DD(I, F)$ as the fundamental datum, we propose to take
a morphism of pre-derivators $p: \DD \rightarrow \SSS$ (or even pre-multiderivators) as the fundamental datum, the $\DD(I, F)$ being reconstructed as its fibers $\DD(I)_F$, if $\SSS$ is
the pre-derivator associated with a category. This allows more general situations, where $\SSS$ is not associated with an ordinary category. 
In many situations, in particular for a six-functor-formalism, it will be necessary to consider $\SSS$ which are pre-2-multiderivators instead, a notion which will be introduced and investigated in a forthcoming article~\cite{Hor15}. There we will define (and give examples of) a derivator version of a (symmetric) Grothendieck six-functor formalism, that is, a (symmetric) fibered multiderivator 
\[ p: \DD \rightarrow \SSS^{\mathrm{cor}}, \]
where $\SSS^{\mathrm{cor}}$ is the symmetric pre-2-multiderivator of correspondences in $\mathcal{S}$. 

\begin{enumerate}
\item[Q2:] More generally, we may consider Cartesian (resp.\@ coCartesian) objects in the fiber over a diagram $(\Delta^{\op}, U_\bullet)$ (resp. $(\Delta, (U_\bullet)^{\op})$), and
ask whether these categories depend only on $U_\bullet$ up to taking (finite) hypercovers w.r.t.\@ a fixed Grothendieck topology on $\mathcal{S}$.  
\end{enumerate}

This will allow to define the six operations, for example, if
the simplicial objects $U_\bullet$ are presentations of stacks.
The categories of coCartesian objects are a generalization of the equivariant derived categories of Bernstein and Lunts (cf.\@ \ref{BL}).

\subsection*{Overview}

In section \ref{SECTIONFIBDER} we give the general definition of a left (resp.\@ right) fibered multiderivator $p: \DD \rightarrow \SSS$. The axioms are basically a
straight-forward generalization of those of a left, resp. right derivator. To give a priori some conceptual evidence that these
axioms are indeed reasonable, we prove that the notion of fibered multiderivator is transitive (\ref{SECTTRANS}), and that it gives rise to a pseudo-functor from `diagrams in 
$\SSS$' to categories, for which a neat base-change formula holds (\ref{SECTPSEUDOFUNCT}).  

In section \ref{SECTCOHOMDESCENT}, a theory of (co)homological descent for fibered derivators is developed (the monoidal, i.e.\@ multi-, aspect does not play any role here).
We propose a definition of localizer (resp.\@ of system of relative localizers) in the category of diagrams in $\SSS$ which is a generalization of Grothendieck's notion
of fundamental localizer in categories. The latter gives a nice combinatorial description of weak equivalences of categories in terms of the condition of Quillen's theorem A. 
In our more general setting the notion of fundamental localizer depends on the choice of a Grothendieck (pre-)topology on $\SSS$.
In section \ref{SECTSIMP} we show purely abstractly that a finite hypercover, considered as a morphism of simplicial diagrams,  lies in
any localizer or system of relative localizers. Thus this more general notion of localizer has a similar relation to weak equivalences of simplicial pre-sheaves, although we
will not yet give any precise statement in this direction. 

In sections \ref{SECTCART} and \ref{SECTEQUIV} these notions of localizers are tied to the theory of fibered derivators. We introduce two notions of (co)homological descent for a fibered derivator $p: \DD \rightarrow \SSS$, namely
 weak and strong $\DD$-equivalences. The notion of weak $\DD$-equivalences (related to Q1 above) is a straight-forward generalization of Cisinski's  notion of $\DD$-equivalence for
usual derivators. In our relative context, both notions of $\DD$-equivalence come in a cohomological as well as in a homological flavour (a phenomenon which does not occur for usual derivators). 

Whenever the fibered derivator is (co)local w.r.t.\@ to the Grothendieck pre-(co)topology --- a rather weak and obviously necessary condition (see section \ref{SECTLOCAL}) --- then 
the Main Theorem \ref{MAINTHEOREMCOHOMDESCENT} (resp.\@ \ref{MAINTHEOREMHOMDESCENT}) of this article 
states that weak $\DD$-equivalences form a system of relative localizers under very general conditions (the easier case) 
and that strong $\DD$-equivalences (related to Q2 above) form an absolute localizer, at least in the case of fibered derivators with stable, compactly generated fibers.  
The proof uses results from the theory of triangulated categories due to Neeman and Krause (centering around Brown representability type theorems). The link to our theory of fibered (multi\nobreakdash-)derivators is explained in section \ref{SECTBROWN}.

In section \ref{SECTCONSTR} we introduce the notion of fibration of multi-model-categories. This
is the most favorable standard context in which a fibered multi-derivator (whose base is associated with a usual multicategory) can be constructed. 
We will present more general methods of constructing fibered multiderivators in a forthcoming article, in particular those encoding a full six-functor-formalism.

\section*{Notation}

We denote by $\mathcal{CAT}$ the 2-``category''\footnote{where ``category'' has classes replaced by 2-classes (or, if the reader prefers, is constructed w.r.t.\@ a larger universe). } of categories, by $(\mathcal{S})\mathcal{MCAT}$ the 2-``category'' of (symmetric) multicategories, and by $\Cat$ the 2-category of small categories.
We consider a partially ordered set (poset) $X$ as a small category by considering the relation $x \le y$ to be equivalent to the existence of a unique
morphism $x \rightarrow y$.
We denote the positive integers (resp.\@ non-negative integers) by $\N$ (resp.\@ $\N_0$). The ordered sets $\{0, \dots, n\} \subset \N_0$ considered as a small category are denoted by $\Delta_n$. 
We denote by $\Mor(\mathcal{D})$ (resp.\@ $\Iso(\mathcal{D})$) the class of morphisms (resp.\@ isomorphisms) in a category $\mathcal{D}$.
The final category (which consists of only one object and its identity) is denoted by $\cdot$ or $\Delta_0$. The same notation is also used for the final multi-category, i.e.
that with one object and precisely one $n$-ary morphism for any $n$. 
Our conventions about multicategories and fibered (multi-)categories are summarized in appendix \ref{APP}.

\tableofcontents

\section{Fibered derivators}\label{SECTIONFIBDER}

\subsection{Categories of diagrams}

\begin{DEF}\label{DEFDIAGRAMCAT}
A {\bf diagram category} is a full sub-2-category $\Dia \subset \Cat$,
satisfying the following axioms:
\begin{itemize}
\item[(Dia1)] The empty category $\emptyset$, the final category $\cdot$ (or $\Delta_0$), and $\Delta_1$ are objects of $\Dia$.
\item[(Dia2)] $\Dia$ is stable under taking finite coproducts and fibered products.
\item[(Dia3)] For each functor $\alpha: I \rightarrow J$ in $\Dia$ and object $j \in J$ the slice categories
$I \times_{/J} j$ and $j \times_{/J} I$ are in $\Dia$.
\end{itemize}
A diagram category $\Dia$ is called {\bf self-dual}, if it satisfies in addition:
\begin{itemize}
\item[(Dia4)] If $I \in \Dia$ then $I^{\op} \in \Dia$.
\end{itemize}
A diagram category $\Dia$ is called {\bf infinite}, if it satisfies in addition:
\begin{itemize}
\item[(Dia5)] $\Dia$ is stable under taking arbitrary coproducts.
\end{itemize}
\end{DEF}

In the following we mean by a {\bf diagram} a small category.

\begin{BEISPIEL}\label{EXDIAGRAMCAT}
We have the following diagram categories:
\begin{itemize}
\item[$\Cat$] the category of all {\bf diagrams}. It is self-dual.
\item[$\Inv$] the category of {\bf inverse} diagrams $C$, i.e.\@ small categories $C$ such that there exists a functor $C \rightarrow \N_0$ with the property that
the preimage of an identity consists of identities\footnote{In many sources $\N_0$ is replaced by any ordinal.}. An example is the injective simplex category $\Delta^\circ$:
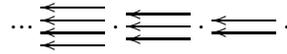
\[ \xymatrix{ \cdots & \ar@<-1.5ex>[l]\ar@<-0.5ex>[l]\ar@<0.5ex>[l]\ar@<1.5ex>[l]    \cdot & \ar@<1ex>[l]\ar@<0ex>[l]\ar@<-1ex>[l] \cdot & \ar@<0.5ex>[l] \ar@<-0.5ex>[l]  \cdot  } \]
\item[$\Dir$] the category of {\bf directed} diagrams $D$, i.e.\@ small categories such that $D^{\op}$ is inverse. An example is the opposite of the injective simplex category $(\Delta^\circ)^{\op}$:
\[ \xymatrix{ \cdots \ar@<-1.5ex>[r]\ar@<-0.5ex>[r]\ar@<0.5ex>[r]\ar@<1.5ex>[r]  &   \cdot \ar@<1ex>[r]\ar@<0ex>[r]\ar@<-1ex>[r] & \cdot \ar@<0.5ex>[r] \ar@<-0.5ex>[r] & \cdot  } \]
\item[$\Catf$,] $\Dirf$, and $\Invf$ are defined as before but consisting of {\bf finite diagrams}.  Those are self-dual and $\Dirf = \Invf$.
\item[$\Catlf$,] $\Dirlf$, and $\Invlf$ are defined as before but consisting of {\bf locally finite diagrams}, i.e.\@ those which have the property that a morphism $\gamma$ factors as $\gamma = \alpha \circ \beta$ only in a finite number of ways.
\item[$\Pos$,] $\Posf$, $\Dirpos$, and $\Invpos$: the categories of {\bf posets, finite posets, directed posets,} and {\bf inverse posets}.
\end{itemize}
\end{BEISPIEL}

\subsection{Pre-(multi-)derivators}

\begin{DEF}\label{DEFPREDER}
A {\bf pre-derivator} of domain $\Dia$ is a contravariant (strict) 2-functor 
\[ \DD: \Dia^{1-\op} \rightarrow \mathcal{CAT}. \]
A {\bf pre-multiderivator} of domain $\Dia$ is a contravariant (strict) 2-functor
\[ \DD: \Dia^{1-\op} \rightarrow \mathcal{MCAT} \]
into the 2-``category'' of multicategories.
A morphism of pre-derivators is a natural transformation.

For a morphism $\alpha: I \rightarrow J$ in $\Dia$ the corresponding functor
\[ \DD(\alpha): \DD(J) \rightarrow \DD(I)  \]
will be denoted by $\alpha^*$.

We call a pre-multiderivator {\bf  symmetric} (resp.\@ {\bf  braided}), if its images are symmetric (resp.\@ braided), and the morphisms $\alpha^*$ are compatible with the actions of the symmetric (resp.\@ braid) groups.
\end{DEF}

\begin{PAR}\label{PARPRECATEGORY}
The pre-derivator associated with a category:
Let $\mathcal{S}$ be a category.
We associate with it the pre-derivator
\[ \SSS: I \mapsto \Hom(I, \mathcal{S}). \]
The pull-back $\alpha^*$ is defined as composition with $\alpha$.
A 2-morphism $\kappa: \alpha \rightarrow \beta$ induces a natural 2-morphism $\SSS(\kappa): \alpha^* \rightarrow \beta^*$.
\end{PAR}

\begin{PAR}\label{PARPRESIMPL}
The pre-derivator associated with a simplicial class (in particular with an $\infty$-category):
Let $\mathcal{S}$ be a simplicial class, i.e.\@ a functor
\[ \mathcal{S}: \Delta \rightarrow \mathcal{CLASS} \]
into the ``category'' of classes.
We associate with it the pre-derivator
\[ \SSS: I \mapsto \Ho(\Hom(N(I), \mathcal{S})), \]
where $N(I)$ is the nerve of $I$ and $\Ho$ is the left adjoint of $N$. 
In detail this means the objects of the category $\SSS(I)$ are 
morphisms $\alpha: N(I) \rightarrow \mathcal{S}$, the class of morphisms in $\SSS(I)$ is
freely generated by morphisms $\mu: N(I \times \Delta_1) \rightarrow \mathcal{S}$ considered to be a morphism
from its restriction to $N(I \times \{0\})$ to its restriction to $N(I \times \{1\})$ modulo the relations given by morphisms
$\nu: N(I \times \Delta_2) \rightarrow \mathcal{S}$, i.e.\@ if $\nu_1, \nu_2$ and $\nu_3$ are the restrictions of $\nu$ to the 3 faces of $\Delta_2$
then we have $\mu_3 = \mu_2 \circ \mu_1$.
The pull-back $\alpha^*$ is defined as composition with the morphism $N(\alpha): N(I) \rightarrow N(J)$.
A 2-morphism $\kappa: \alpha \rightarrow \beta$ can be given as a functor $I \times \Delta_1 \rightarrow J$ which yields (applying $N$ and composing)
a natural transformation which we call $\SSS(\kappa)$.
\end{PAR}
\begin{PAR}\label{PARPREDEND}
More generally, consider the full subcategory 
${}^M\Delta \subset \mathcal{MCAT}$ of all {\em finite connected} multicategories $M$ that are freely generated by a finite set of multimorphisms $f_1, \dots, f_n$ 
such that each object of $M$ occurs at most once as a source and at most once as the target of one of the $f_i$. Similarly
consider the full subcategory $T \subset \mathcal{SMCAT}$ which is obtained from ${}^M\Delta$ adding images under the operations of the symmetric groups. 
This category is usually called the symmetric tree category.
With a functor 
\[ \mathcal{S}: {}^M\Delta \rightarrow \mathcal{CLASS} \quad \text{resp. } \quad \mathcal{S}: T \rightarrow  \mathcal{CLASS} \]
we associate the pre-multiderivator (resp.\@ symmetric pre-multiderivator):
\[ \SSS: I \mapsto \Ho(\Hom(N(I), \mathcal{S})), \]
where $N: \mathcal{MCAT} \rightarrow \mathcal{CLASS}^{{}^M\Delta}$ (resp. $N: \mathcal{SMCAT} \rightarrow \mathcal{CLASS}^{T}$ ) is the nerve, $I$ is considered to be a multicategory without any $n$-ary morphisms for $n\ge 2$, 
and $\Ho$ is the left adjoint of $N$. 
Objects in $\mathcal{SET}^{T}$ are called dendroidal sets in \cite{MW07}. 
\end{PAR}

\subsection{Fibered (multi-)derivators}

\begin{PAR}\label{PARKAN}
Let $p: \DD \rightarrow \SSS$ be a strict morphism of pre-derivators with domain $\Dia$, and let $\alpha: I \rightarrow J$ be a functor in $\Dia$.
Consider an object $S \in \SSS(J)$.
The functor $\alpha^*$ induces a morphism between fibers (denoted the same way)
\[ \alpha^*: \DD(J)_S \rightarrow \DD(I)_{\alpha^*S}.  \]
We are interested in the case that the latter has a left adjoint $\alpha_!^S$, resp.\@ a right adjoint $\alpha_*^S$. These will be
called {\bf relative left/right homotopy Kan extension} functors with {\bf base} $S$. For better readability we often omit the base from the notation. Though the
base is not determined by the argument of $\alpha_!$, it will often be understood from the context, cf.\@ also (\ref{PAREXTENSIONKAN}).

\end{PAR}

\begin{PAR}
We are interested in the case in which all morphisms
\[ p(I): \DD(I) \rightarrow \SSS(I) \]
are Grothendieck fibrations, resp.\@ opfibrations (\ref{APPGROTH}) or, more generally, (op)fibrations of multicategories (\ref{APPMULTI}).
We always assume that the functors $\alpha^*:=\DD(\alpha)$ map coCartesian morphisms to coCartesian morphisms
but map Cartesian morphisms to Cartesian morphisms (for arity $n\ge 2$) only if $\alpha$ itself is a Grothendieck opfibration. 

Then we will choose an associated pseudo-functor, i.e.\@  for each 
$f: S \rightarrow T$ in $\SSS(I)$ a pair of adjoints functors
\[ f_\bullet: \DD(I)_S \rightarrow \DD(I)_T, \]
resp.\@ 
\[ f^\bullet: \DD(I)_T \rightarrow \DD(I)_S, \]
characterized by functorial isomorphisms: 
\[ \Hom_f(\mathcal{E}, \mathcal{F}) \cong \Hom_{\id_T}(\mathcal{E}, f^{\bullet} \mathcal{F})\cong \Hom_{\id_S}(f_{\bullet} \mathcal{E},  \mathcal{F}).  \]

More generally, in the multicategorical setting, if $f$ is a multimorphism $f \in \Hom(S_1, \dots, S_n; T)$ for some $n \ge 1$, we get an {\em adjunction of $n$ variables}
\[ f_\bullet: \DD(I)_{S_1} \times \cdots \times  \DD(I)_{S_n} \rightarrow \DD(I)_T, \]
and
\[ f^{i,\bullet}:  \DD(I)_{S_1}^{\op} \times \mathop{\cdots}\limits^{\widehat{i}} \times  \DD(I)_{S_n}^{\op} \times  \DD(I)_T \rightarrow \DD(I)_{S_i}. \]
\end{PAR}

\begin{PAR}
For a diagram of categories
\[ \xymatrix{  & I \ar[d]^\alpha \\
K \ar[r]_\beta & J \\
} \]
the {\bf slice category} $K \times_{/J} I$ is the category of tripels $(k, i, \mu)$, where $k \in K$, $i \in I$ and $\mu: \alpha(i) \rightarrow \beta(k)$.
It sits in a corresponding 2-commutative square:
\[ \xymatrix{ K \times_{/J} I \ar[r]^-B \ar[d]_{A} \ar@{}[dr]|{\Nearrow^\mu} & I \ar[d]^\alpha \\
K \ar[r]_\beta & J \\
} \]
which is universal w.r.t.\@ such squares. This construction is associative, but of course not commutative unless $J$ is a groupoid. 
The projection $K \times_{/J} I \rightarrow K$ is a Grothendieck fibration and the projection $K \times_{/J} I \rightarrow I$ is a Grothendieck opfibration (see \ref{APPGROTH}). There is an adjunction
\[ \xymatrix{ I \times_{/J} J \ar@<4pt>[rr]^-{} && \ar@<4pt>[ll]^-{} I.  } \]
\end{PAR}

\begin{PAR}\label{PARDERBASECHANGE}
Consider an arbitrary 2-commutative square
\begin{equation}\label{eqcommsquare}
\xymatrix{ L \ar[r]^-B \ar[d]_{A} \ar@{}[dr]|{\Nearrow^\mu} & I \ar[d]^\alpha \\
K \ar[r]_\beta & J \\
} 
\end{equation}

and let $S \in \SSS(J)$ be an object and $\mathcal{E}$ a preimage in $\DD(J)$ w.r.t.\@ $p$.
The 2-morphism (natural transformation) $\mu$  induces a functorial morphism 
\[ \SSS(\mu): A^*\beta^*S \rightarrow B^*\alpha^*S \]
 and therefore a functorial morphism
\[ \DD(\mu): A^*\beta^* \mathcal{E} \rightarrow B^*\alpha^* \mathcal{E} \]
over $\SSS(\mu)$, or ---
if we are in the (op)fibered situation --- equivalently
\[ A^*\beta^* \mathcal{E} \rightarrow (\SSS(\mu))^\bullet B^*\alpha^* \mathcal{E} \]
respectively
\[ (\SSS(\mu))_\bullet A^*\beta^* \mathcal{E} \rightarrow  B^*\alpha^* \mathcal{E} \]
in the fiber above $A^*\beta^* S$, resp.\@ $B^*\alpha^*S$,

Let now $\mathcal{F}$ be an object over $\alpha^*S$.
If relative right homotopy Kan extensions exist, we may form the following composition which will be called the (right) {\bf base-change morphism}:
\begin{equation}\label{eqbasechangeleft}
 \beta^* \alpha_* \mathcal{F} \rightarrow A_* A^* \beta^* \alpha_* \mathcal{F} \rightarrow A_* (\SSS(\mu))^\bullet B^* \alpha^* \alpha_* \mathcal{F} \rightarrow A_* (\SSS(\mu))^\bullet B^* \mathcal{F}.
\end{equation}
(We again omit the base $S$ from the notation for better readability --- it is always determined by the argument.)

Let now $\mathcal{F}$ be an object over $\beta^*S$.
If relative left homotopy Kan extensions exist, we may form the composition, the (left) {\bf base-change morphism}:
\begin{equation}\label{eqbasechangeright} 
 B_! (\SSS(\mu))_\bullet A^* \mathcal{F}  \rightarrow B_! (\SSS(\mu))_\bullet A^* \beta^* \beta_! \mathcal{F}  \rightarrow  B_! B^* \alpha^* \beta_! \mathcal{F}  \rightarrow \alpha^* \beta_!  \mathcal{F}. 
\end{equation}

We will later say that the square (\ref{eqcommsquare}) is {\bf homotopy exact} if (\ref{eqbasechangeleft}) is an isomorphism for all right fibered derivators (see Definition~\ref{DEFFIBDER} below) and (\ref{eqbasechangeright}) is an isomorphism for all left fibered derivators.
It is obvious a priori that  for a left {\em and} right fibered derivator 
 (\ref{eqbasechangeleft}) is an isomorphism if and only if  (\ref{eqbasechangeright}) is, one being the adjoint of the other (see \cite[\S 1.2]{Gro13} for analogous reasoning in the case of usual derivators).
\end{PAR}

\begin{DEF}\label{DEFDER}
We consider the following axioms\footnote{The numbering is compatible with that of \cite{Gro13} in the case of non-fibered derivators.} on a pre-(multi-)derivator $\DD$:
\begin{itemize}
\item[(Der1)] For $I, J$ in $\Dia$, the natural functor $\DD(I \coprod J) \rightarrow \DD(I) \times \DD(J)$ is an equivalence of (multi\mbox{-})categories. Moreover $\DD(\emptyset)$ is not empty.
\item[(Der2)]
For $I$ in $\Dia$ the `underlying diagram' functor
\[ \dia: \DD(I) \rightarrow \Hom(I, \DD(\cdot)) \]
is conservative.
\end{itemize}

In addition, we consider the following axioms for a strict morphism of pre-(multi-)derivators \[ p: \DD \rightarrow \SSS:\]

\begin{enumerate}
\item[(FDer0 left)]
For each $I$ in $\Dia$ the morphism $p$ specializes to an opfibered (multi-)category and
any functor $\alpha: I \rightarrow J$ in $\Dia$ induces a  diagram
\[ \xymatrix{
\DD(J) \ar[r]^{\alpha^*} \ar[d] & \DD(I) \ar[d]\\
\SSS(J) \ar[r]^{\alpha^*} & \SSS(I) 
}\]
of opfibered (multi-)categories, i.e.\@ the top horizontal functor maps coCartesian arrows to coCartesian arrows.

\item[(FDer3 left)]
For each functor $\alpha: I \rightarrow J$ in $\Dia$ and $S \in \SSS(J)$ the functor
$\alpha^*$ between fibers
\[ \DD(J)_{S} \rightarrow \DD(I)_{\alpha^*S} \]
has a left-adjoint $\alpha_!^S$.

\item[(FDer4 left)]
For each functor $\alpha: I \rightarrow J$ in $\Dia$, and for any object $j \in J$, and the 2-cell
\[ \xymatrix{  I \times_{/J} j \ar[r]^-\iota \ar[d]_{\alpha_j} \ar@{}[dr]|{\Swarrow^\mu} & I \ar[d]^\alpha \\
\{j\} \ar@{^{(}->}[r]^j & J \\
} \]
we get that the induced natural transformation of functors ${\alpha_j}_! (\SSS(\mu))_\bullet \iota^* \rightarrow j^* {\alpha}_! $ is an isomorphism\footnote{This is meant to hold w.r.t.\@ all bases $S \in \SSS(J)$.}.
\item[(FDer5 left)] (only needed for the multiderivator case). For any Grothendieck opfibration $\alpha: I \rightarrow J $ in $\Dia$, and for any morphism $\xi \in \Hom(S_1, \dots, S_n; T)$ in $\SSS(\cdot)$ for some $n\ge 1$, the natural transformations of functors
\[ \alpha_! (\alpha^*\xi)_\bullet (\alpha^*-, \cdots, \alpha^*-,\ -\ , \alpha^*-, \cdots, \alpha^*-) \cong  \xi_\bullet (-, \cdots, -,\ \alpha_!-\ , -, \cdots, -) \]
are isomorphisms.
\end{enumerate}

and their dual variants:

\begin{enumerate}
\item[(FDer0 right)]
For each $I$ in $\Dia$ the morphism $p$ specializes to a fibered (multi-)category and any {\em Grothendieck opfibration} $\alpha: I \rightarrow J$
in $\Dia$  induces a  diagram
\[ \xymatrix{
\DD(J) \ar[r]^{\alpha^*} \ar[d] & \DD(I) \ar[d]\\
\SSS(J) \ar[r]^{\alpha^*} & \SSS(I) 
}\]
of fibered (multi-)categories, i.e.\@ the top horizontal functor maps Cartesian arrows w.r.t.\@ the $i$-th slot to Cartesian arrows w.r.t.\@ the $i$-th slot.
\item[(FDer3 right)]
For each functor $\alpha: I \rightarrow J$ in $\Dia$ and $S \in \SSS(J)$ the functor
$\alpha^*$ between fibers
\[ \DD(J)_{S} \rightarrow \DD(I)_{\alpha^*S} \]
has a right-adjoint $\alpha_*^S$.
\item[(FDer4 right)]
For each morphism $\alpha: I \rightarrow J$ in $\Dia$, and for any object $j \in J$, and the 2-cell
\[ \xymatrix{  j \times_{/J} I \ar[r]^-\iota \ar[d]_{\alpha_j} \ar@{}[dr]|{\Nearrow^\mu} & I \ar[d]^\alpha \\
\{j\} \ar@{^{(}->}[r]^j & J \\
} \]
we get that the induced natural transformation of functors $  j^* {\alpha}_* \rightarrow {\alpha_j}_* (\SSS(\mu))^\bullet \iota^*$ is an isomorphism\footnote{This is meant to hold w.r.t.\@ all bases $S \in \SSS(J)$.}.

\item[(FDer5 right)] (only needed for the multiderivator case). For any functor $\alpha: I \rightarrow J$ in $\Dia$, and for any a morphism $\xi \in \Hom(S_1, \dots, S_n; T)$ in $\SSS(\cdot)$ for some $n\ge 1$, the natural transformations of functors
\[ \alpha_* (\alpha^*\xi)^{\bullet,i} (\alpha^*-, \cdots, \alpha^*-\ ;\ -) \cong  \xi^{\bullet,i} (-, \cdots, -\ ;\ \alpha_*-) \]
are isomorphisms for all $1 \le i \le n$.
\end{enumerate}
\end{DEF}

\begin{DEF}\label{DEFFIBDER}
A strict morphism of pre-(multi-)derivators $p: \DD \rightarrow \SSS$ with domain $\Dia$
is called a {\bf left fibered (multi-)derivator} with domain $\Dia$, if axioms (Der1--2) hold for $\DD$ and $\SSS$ and (FDer0--5 left) hold for $p$.
Similarly it is called a {\bf right fibered (multi)-derivator} with domain $\Dia$, if instead the corresponding dual axioms (FDer0--5 right)  hold.
It is called just {\bf fibered} if it is both left and right fibered.  
\end{DEF}

The squares in axioms (FDer4 left/right) are in fact homotopy exact and it follows from the axioms (FDer4 left/right) that many more are (see \ref{PROPHOMCART}).

There is some reduncancy in the axioms, cf.\@ \ref{LEMMALEFTRIGHT} and \ref{REDUNDANCY}.

\begin{FRAGE}
It seems natural to allow also (symmetric) multicategories, in particular operads, as {\em domain} for a fibered (symmetric) multiderivator. 
The author however did not succeed in writing down a neat generalization of (FDer3--4) which would encompass (FDer5).
\end{FRAGE}

\begin{LEMMA}\label{LEMMALEFTRIGHT}
For a strict morphism of pre-derivators $\DD \rightarrow \SSS$ such that both satisfy (Der1) and (Der2) and such that it induces a bifibration of multi-categories $\DD(I) \rightarrow \SSS(I)$ for all $I \in \Dia$ we have the following implications:
\begin{eqnarray}
\text{(FDer0 left) for $n$-ary morphisms, $n\ge 1$} &\Leftrightarrow& \text{(FDer5 right)} \label{impl1} \\
\text{(FDer0 right)} &\Leftrightarrow& \text{(FDer5 left)}  \label{impl2} 
\end{eqnarray} 
\end{LEMMA}
\begin{proof}
We will only show the implication (\ref{impl1}), the other being similar.
Choosing pull-back functors $f^\bullet$, the remaining part of (FDer0) says that the natural 2-morphism 
\[ \xymatrix{
\DD(J)_{S_1} \times \cdots \times  \DD(J)_{S_n} \ar[rr]^-{f_\bullet} \ar[d]^{\alpha^*} \ar@{}[rrd]|{\Swarrow} & & \DD(J)_{T} \ar[d]^{\alpha^*} \\
\DD(I)_{\alpha^* S_1} \times \cdots \times  \DD(J)_{\alpha^* S_n} \ar[rr]^-{(\alpha^*f)_\bullet} & &  \DD(I)_{\alpha^* T}
} \]
is an isomorphism. Taking the adjoint of this diagram (of $f_\bullet$ w.r.t.\@ the $i$-th slot) we get the diagram
\[ \xymatrix{
\DD(I)_{\alpha^* S_1}^{\op} \times \mathop{\cdots}\limits^{\widehat{i}} \times \DD(I)_{\alpha^* S_n}^{\op} \ar@{}[r]|-{\times} &  \DD(I)_{\alpha^* T} \ar[d]^{ \alpha_*}  
 \ar@{}[rrd]|{\Swarrow} \ar[rr]^-{(\alpha^*f)^{\bullet,i}} & & \DD(J)_{\alpha^* S_i}  \ar[d]^{\alpha_*} \\
\DD(I)_{S_1}^{\op} \times \mathop{\cdots}\limits^{\widehat{i}} \times  \DD(J)_{S_n}^{\op} \ar[u]^{(\alpha^*)^{\op}}  \ar@{}[r]|-{\times} &  \DD(J)_{T}  \ar[rr]^-{f^{\bullet, i}}  & & \DD(J)_{S_i}
} \]
That its 2-morphism is an isomorphism is the content of (FDer5 left). Hence (FDer0 left) and (FDer5 right) are equivalent in this situation. 

For (\ref{impl2}) note that for both (FDer0 right) and (FDer5 left), the functor $\alpha$ in question is restricted to the class of Grothendieck opfibrations. 
\end{proof}

\begin{BEM}
The axioms (FDer0) and (FDer3--5) are similar to the axioms of a six-functor-formalism (cf.\@ the introduction or the appendix \ref{APPMULTI}). 
It is actually possible to make this analogy precise and {\em define} a fibered multiderivator as a bifibration of 2-multicategories $[\DD] \rightarrow \Dia^{\mathrm{cor}}(\SSS)$ where  $\Dia^{\mathrm{cor}}(\SSS)$ is a certain
category of multicorrespondences of diagrams in $\SSS$, similar to our definition of a usual six-functor-formalism (cf.\@ Definition~\ref{DEF6FU}).
This also clarifies the existence and comparison of the internal and external monoidal structure, resp.\@ duality, 
in a closed monoidal derivator (i.e.\@ fibered multiderivator over $\{\cdot\}$) or more generally for
any fibered multiderivator.  We will explain this in detail in a subsequent article \cite{Hor15}. 
\end{BEM}

\begin{PAR}\label{PARDERSIMPL}
The pre-derivator associated with an $\infty$-category $\mathcal{S}$ is actually a left and right derivator (in the usual sense, i.e.\@ fibered over $\{\cdot\}$)
if $\mathcal{S}$ is complete and co-complete \cite{GPS13}. This includes the case of pre-derivators associated with categories, which is, of course, classical --- axiom (FDer4) expressing nothing else than Kan's formulas. 
\end{PAR}

\begin{PAR}\label{DEFFDERFIBERS}
Let $S \in \SSS(\cdot)$ be an object and $p: \DD \rightarrow \SSS$ be a (left, resp.\@ right) fibered multiderivator. The association
\[ I \mapsto \DD(I)_{p^*S}, \]
where $p: I \rightarrow \cdot$ is the projection, defines a (left, resp.\@ right) derivator in the usual sense which we call its fiber $\DD_S$
over $S$. The axioms (FDer6--7) stated below involve only these fibers.
\end{PAR}

\begin{DEF}
More generally, if $S \in \SSS(J)$ we may consider the association
\[ I \mapsto \DD(I \times J)_{\pr_2^*S}, \]
where $\pr_2: I \times J \rightarrow J$ is the second projection. This defines again a (left, resp.\@ right) derivator in the usual sense which we call its {\bf fiber} $\DD_S$ 
over $S$.
\end{DEF}

\begin{LEMMA}[left] \label{LEMMAMORPHDERLEFT}
Let $\DD \rightarrow \SSS$ be a left fibered multiderivator and let $I \in \Dia$ be a diagram and $f \in \Hom_{\SSS(J)}(S_1, \dots, S_n; T)$ for some $n \ge 1$ be a morphism. Then
the collection of functors for each $J \in \Dia$
\begin{eqnarray*}
f_\bullet: \DD(J \times I)_{\pr_2^*S_1} \times \cdots \times \DD(J \times I)_{\pr_2^*S_n} & \rightarrow & \DD(J \times I)_{\pr_2^*T}  \\
\mathcal{E}_1, \dots, \mathcal{E}_n &\mapsto& (\pr_2^* f)_\bullet(\mathcal{E}_1, \dots, \mathcal{E}_n)
\end{eqnarray*}
defines a morphism of derivators $\DD_{S_1} \times \dots \times \DD_{S_n} \rightarrow \DD_T$. Furthermore, for a collection $\mathcal{E}_k \in \DD(I), k\not=i$ the morphism of derivators:
\begin{eqnarray*}
\DD(J \times I)_{\pr_2^*S_i} & \rightarrow & \DD(J \times I)_{\pr_2^*T}  \\
 \mathcal{E}_i &\mapsto& (\pr_2^* f)_\bullet(\pr_2^* \mathcal{E}_1, \dots, \mathcal{E}_i, \dots, \pr_2^* \mathcal{E}_n)
\end{eqnarray*}
is left continuous (i.e.\@ commutes with left Kan extensions). 
\end{LEMMA}
\begin{proof} The only point which might not be clear is the left continuity of the bottom morphism of pre-derivators.
Consider the following 2-commutative square, where $I, J, J' \in \Dia$, $\alpha: J \rightarrow J'$ is a functor, and $j' \in J'$ 
\[ \xymatrix{
I \times (j' \times_{/J'} J) \ar[rr]^-{(\id , \iota) } \ar[d]_{(\id , p)} \ar@{}[drr]|-{\Nearrow} && I \times J \ar[d]^{(\id, \alpha)} \\
I \times j' \ar[rr] && I \times J'
} \]

It is homotopy exact by \ref{PROPHOMCART}, 4. 
Therefore we have (using FDer3--5 left):
\begin{eqnarray*}
 && (\id, j')^* (\id, \alpha)_!  (\pr_1^*f)_\bullet(\pr_1^* \mathcal{E}_1, \dots, \mathcal{E}_i, \dots, \pr_1^* \mathcal{E}_n)  \\
 &\cong&  (\id, p)_! (\id, \iota)^* (\pr_1^*f)_\bullet(\pr_1^* \mathcal{E}_1, \dots, \mathcal{E}_i, \dots, \pr_1^* \mathcal{E}_n) \\
 &\cong& (\id, p)_! (\pr_1^*f)_\bullet( (\id, \iota)^* \pr_1^* \mathcal{E}_1,  \dots, (\id, \iota)^*\mathcal{E}_i, \dots, (\id, \iota)^* \pr_1^* \mathcal{E}_n) \\
 &\cong&  (\id, p)_! (\pr_1^*f)_\bullet( (\id, p)^* \mathcal{E}_1,  \dots, (\id, \iota)^* \mathcal{E}_i, \dots, (\id, p)^* \mathcal{E}_n) \\
 &\cong&   f_\bullet(  \mathcal{E}_1,  \dots, (\id, p)_! (\id, \iota)^* \mathcal{E}_i, \dots,  \mathcal{E}_n) \\
 &\cong&   f_\bullet(  \mathcal{E}_1,  \dots, (\id, j')^* (\id, \alpha)_!  \mathcal{E}_i, \dots,  \mathcal{E}_n) \\
 &\cong&   (\id, j')^* (\pr_1^*f)_\bullet( \pr_1^* \mathcal{E}_1,  \dots,  (\id, \alpha)_!  \mathcal{E}_i, \dots,  \pr_1^* \mathcal{E}_n)
\end{eqnarray*}
(Note that $(\id, p)$ is trivially a Grothendieck op-fibration). 
A tedious check shows that the composition of these isomorphisms is $(\id, j')^*$ applied to the exchange morphism 
\[  (\id, \alpha)_!  (\pr_1^*f)_\bullet(\pr_1^* \mathcal{E}_1, \dots, \mathcal{E}_i, \dots, \pr_1^* \mathcal{E}_n)  \rightarrow  (\pr_1^*f)_\bullet( \pr_1^* \mathcal{E}_1,  \dots,  (\id, \alpha)_!  \mathcal{E}_i, \dots, \pr_1^*  \mathcal{E}_n)  \]
Since the above holds for any $j' \in J'$ the exchange morphism is therefore an isomorphism by (Der2). 
\end{proof}

In the right fibered situation the analogously defined morphisms $f^{i, \bullet}$ are not expected to be made into a morphism of fibers this way. For a discussion of how this is solved, we 
refer the reader to the article~\cite{Hor15} in preparation, where a fibered multiderivator is redefined as a certain type of six-functor-formalism. This will let appear the discussion and results of this section in a much more
clear fashion. However, we have: 

\begin{LEMMA}[right]  \label{LEMMAMORPHDERRIGHT}
Let $\DD \rightarrow \SSS$ be a right fibered multiderivator and let $I \in \Dia$ be a diagram and $f \in \Hom_{\SSS(J)}(S_1, \dots, S_n; T)$, for some $ n\ge 1$, be a morphism. For each $J \in \Dia$ and
for each collection $\mathcal{E}_k \in \DD(I)$, $k \not=i$, the association
\begin{eqnarray*}
\DD(J \times I)_{\pr_2^*T} & \rightarrow & \DD(J \times I)_{\pr_2^*S_i}  \\
 \mathcal{F} &\mapsto& (\pr_2^* f)^{\bullet, i}(\pr_2^* \mathcal{E}_1, \dots, \pr_2^* \mathcal{E}_n; \mathcal{F})
\end{eqnarray*}
defines a morphism of derivators which is right continuous (i.e.\@ commutes with right Kan extensions). This is the right adjoint in the pre-derivator sense to the morphism of pre-derivators in the previous lemma, as soon as $\DD \rightarrow \SSS$ is left and right fibered. 
\end{LEMMA}
\begin{proof}
Consider the following 2-commutative square where $I, J, J' \in \Dia$, $\alpha: J \rightarrow J'$ is a functor, and $j' \in J'$
\[ \xymatrix{
I \times (J \times_{/J'} j') \ar[rr]^-{(\id , \iota) } \ar[d]_{(\id , p)} \ar@{}[drr]|-{\Swarrow} && I \times J \ar[d]^{(\id, \alpha)} \\
I \times j' \ar[rr] && I \times J'
} \]
It is homotopy exact by \ref{PROPHOMCART}, 4. 

Therefore we have (using FDer3--5 right):
\begin{eqnarray*}
 && (\id, j')^* (\id, \alpha)_!  (\pr_1^*f)^{i,\bullet}(\pr_1^* \mathcal{E}_1, \mathop{\dots}\limits^{\widehat{i}}, \pr_1^* \mathcal{E}_n; \mathcal{F})  \\
 &\cong&  (\id, p)_* (\id, \iota)^* (\pr_1^*f)^{i,\bullet}(\pr_1^* \mathcal{E}_1, \mathop{\dots}\limits^{\widehat{i}},  \pr_1^* \mathcal{E}_n, \mathcal{F}) \\
 &\cong& (\id, p)_* (\pr_1^*f)^{i,\bullet} ( (\id, \iota)^* \pr_1^* \mathcal{E}_1,  \mathop{\dots}\limits^{\widehat{i}}, (\id, \iota)^* \pr_1^* \mathcal{E}_n;  (\id, \iota)^*\mathcal{F}) \\
 &\cong&  (\id, p)_* (\pr_1^*f)^{i,\bullet} ( (\id, p)^* \mathcal{E}_1,  \mathop{\dots}\limits^{\widehat{i}},  (\id, p)^* \mathcal{E}_n; (\id, \iota)^* \mathcal{F}) \\
 &\cong&   f^{i,\bullet} (  \mathcal{E}_1,  \mathop{\dots}\limits^{\widehat{i}},  \mathcal{E}_n ;  (\id, p)_* (\id, \iota)^* \mathcal{F}) \\
 &\cong&   f^{i,\bullet}(  \mathcal{E}_1, \mathop{\dots}\limits^{\widehat{i}},  \mathcal{E}_n ; (\id, j')^* (\id, \alpha)_*  \mathcal{F}) \\
 &\cong&   (\id, j')^* (\pr_1^*f)^{i,\bullet}(  \pr_1^* \mathcal{E}_1,  \mathop{\dots}\limits^{\widehat{i}},  \pr_1^* \mathcal{E}_n; (\id, \alpha)_!  \mathcal{F} )
\end{eqnarray*}
Note that $(\id , \iota)$ is a Grothendieck op-fibration, but $(id, j')$ is not. Hence the last step has to be justified further. 
Consider the 2-commutative diagram: 
\[ \xymatrix{
I \times (J \times_{/J'} j') \ar[rr]^-{(\id , \iota') } \ar[d]_{(\id, p)} \ar@{}[drr]|-{\Swarrow} && I \times J' \ar@{=}[d] \\
I \times j' \ar[rr] && I \times J'
} \]
It is again homotopy exact by \ref{PROPHOMCART}, 4. 
Therefore we have
\begin{eqnarray*}
 &\cong&   f^{i,\bullet}(  \mathcal{E}_1,  \mathop{\dots}\limits^{\widehat{i}},  \mathcal{E}_n ;  (\id, j')^* (\id, \alpha)_*  \mathcal{F}) \\
 &\cong&   f^{i,\bullet}(  \mathcal{E}_1,  \mathop{\dots}\limits^{\widehat{i}},  \mathcal{E}_n ; (\id, p)_* (\id, \iota')^* (\id, \alpha)_*  \mathcal{F}) \\
 &\cong&   (\id, p)_* (\pr_1^*f)^{i,\bullet}( (\id, p)^*  \mathcal{E}_1,  \mathop{\dots}\limits^{\widehat{i}}, (\id, p)^* \mathcal{E}_n ; (\id, \iota')^* (\id, \alpha)_*  \mathcal{F}) \\
 &\cong&   (\id, p)_* (\pr_1^*f)^{i,\bullet}( (\id, \iota')^* \pr_1^*  \mathcal{E}_1,  \mathop{\dots}\limits^{\widehat{i}},  (\id, \iota')^* \pr_1^*  \mathcal{E}_n ; (\id, \iota')^* (\id, \alpha)_*  \mathcal{F}) \\
 &\cong&   (\id, p)_* (\id, \iota')^* (\pr_1^*f)^{i,\bullet}(  \pr_1^*  \mathcal{E}_1,  \mathop{\dots}\limits^{\widehat{i}},  \mathcal{E}_n ; (\id, \iota')^* (\id, \alpha)_*  \mathcal{F}) \\
 &\cong&   (\id, j')^* (\pr_1^*f)^{i,\bullet}(  \pr_1^* \mathcal{E}_1,  \mathop{\dots}\limits^{\widehat{i}},  \pr_1^* \mathcal{E}_n; (\id, \alpha)_!  \mathcal{F} )
\end{eqnarray*}
 Note that $(\id , \iota')$ is a Grothendieck opfibration, too. In other words: the reason why $f^{\bullet, i}$ also commutes with $(\id, j')^*$ in this particular case is that the other argument are
 constant in the $J$ direction.
 
A tedious check shows the composition of the isomorphisms of the previous computations yield $(\id, j')^*$ applied to  the exchange morphism 
\[  (\id, \alpha)_*  (\pr_1^*f)^{i, \bullet}(\pr_1^* \mathcal{E}_1, \mathop{\dots}\limits^{\widehat{i}}, \pr_1^* \mathcal{E}_n;  \mathcal{F})  \rightarrow  (\pr_1^*f)^{i,\bullet}(  \mathcal{E}_1,  \mathop{\dots}\limits^{\widehat{i}},  \mathcal{E}_n; (\id, \alpha)_*  \mathcal{F}).  \]
Since the above holds for any $j' \in J'$ it is therefore an isomorphism by (Der2). 
\end{proof}

\begin{PAR}
Let $p: \DD \rightarrow \SSS$ be a (left, resp.\@ right) fibered {\em multi}-derivator and $S: \{\cdot\} \rightarrow \SSS(\cdot)$ a functor of multicategories. 
This is equivalent to the choice of an object $S \in \SSS(\cdot)$ and a collection of morphisms $\alpha_n \in \Hom_{ \SSS(\cdot)}(\underbrace{S, \dots, S}_{n \text{ times}}; S)$ for all $n\ge 2$, compatible with composition. 
Then the fiber
\[ I \mapsto \DD(I)_{p^*S} \]
defines even a (left, resp.\@ right) multiderivator (i.e.\@ a fibered multiderivator over $\{\cdot\}$). The same holds analogously for a functor of multicategories $S:  \{\cdot\} \rightarrow \SSS(I)$.
\end{PAR}

Axiom (FDer5 left) and Lemma~\ref{LEMMAPROPMULTI} imply the following: 
\begin{PROP}
The definition of a left fibered multiderivator $\DD \rightarrow \{\cdot\}$ is equivalent to the definition of a monoidal left derivator in the sense of Groth \cite{Gro12}. 
It is also, in addition, right fibered if and only if it is a right derivator and closed monoidal in the sense of  [loc.\@ cit.].
\end{PROP}

\begin{DEF}
We call a pre-derivator $\DD$ {\bf strong}, if the following axiom holds:
\begin{itemize}
\item[(Der8)]  For any diagram $K$ in $\Dia$ the `partial underlying diagram' functor
\[ \dia: \DD(K \times \Delta_1) \rightarrow \Hom(\Delta_1, \DD(K)) \]
is full and essentially surjective. 
\end{itemize}
\end{DEF}

\begin{DEF}
Let  $p: \DD \rightarrow \SSS$ be a fibered (left and right) derivator.
We call $\DD$ {\bf pointed} (relative to $p$) if the following axiom holds:
\begin{itemize}
\item[(FDer6)] For any $S \in \SSS(\cdot)$, the category $\DD(\cdot)_S$ has a zero object.
\end{itemize}
\end{DEF}

\begin{DEF}
Let  $p: \DD \rightarrow \SSS$ be a fibered (left and right) derivator.
We call $\DD$ {\bf stable} (relative to $p$) if its fibers are {\em strong} and the following axiom holds:
\begin{itemize}
\item[(FDer7)] For any $S \in \SSS(\cdot)$, in the category $\DD(\Box)_{p^*S}$ an object is homotopy Cartesian if and only if it is homotopy coCartesian.
\end{itemize}
\end{DEF}

This condition can be weakened (cf.\@ \cite[Corollary 8.13]{GS14}).

\begin{PAR}
Recall from \cite{Gro13} that axiom (FDer7) implies that the fibers of a stable fibered derivator are triangulated categories in a natural way. Actually the proof shows that it suffices to have a derivator of domain $\Posf$ (finite posets). 

Since, by Lemma~\ref{LEMMAMORPHDERLEFT} and Lemma~\ref{LEMMAMORPHDERRIGHT} push-forward, resp.\@ pull-back w.r.t.\@ any slot commute with homotopy colimits, resp.\@ homotopy limits, they induce triangulated functors between the fibers.
\end{PAR}

\begin{PAR}[left]
The following is a consequence of (FDer0):
For a functor $\alpha: I \rightarrow J$ and a morphism in $f: S \rightarrow T \in \SSS(J)$, we get a natural isomorphism
\[ \SSS(\alpha^* f)_\bullet \alpha^* \rightarrow \alpha^* \SSS(f)_\bullet. \]
W.r.t.\@ this natural isomorphism we have the following:
\end{PAR}

\begin{LEMMA}[left]\label{LEMMAPASTING}
Given a ``pasting'' diagram
\[ \xymatrix{
N \ar@{}[dr]|{\Swarrow^\nu} \ar[r]^G \ar[d]^{A} & L \ar@{}[dr]|{\Swarrow^\mu} \ar[r]^{B} \ar[d]^a & I \ar[d]^\alpha \\
M \ar[r]^\gamma & K \ar[r]^\beta & J 
}\]
we get for the pasted natural transformation $\nu \odot \mu := (\beta \ast \nu) \circ (\mu \ast G)$ that the following diagram is commutative:
\[ \xymatrix{
A_! \SSS(\beta \ast \nu)_\bullet G^* \SSS(\mu)_\bullet B^* \ar[r] \ar[d]^\sim  & \gamma^* a_! \SSS(\mu)_\bullet B^* \ar[r] & \gamma^* \beta^* \alpha_!  \\
A_! \SSS(\beta \ast \nu)_\bullet \SSS(G \ast \mu)_\bullet G^* B^* \ar[d]^\sim \\ 
A_! \SSS(\nu \odot \mu)_\bullet G^* B^* \ar[rruu] 
} \]
Here the morphisms going to the right are (induced by) the various base-change morphisms.
In particular, the pasted square is homotopy exact if the individual two squares are.
\end{LEMMA}
\begin{proof}
This is an analogue of \cite[Lemma 1.17]{Gro13} and proven similarly.
\end{proof}

\begin{PROP}\label{PROPHOMCART}
\begin{enumerate}
\item Any square of the form
\[ \xymatrix{ I \times_{/J} K \ar[r]^B \ar[d]_{A} \ar@{}[dr]|{\Swarrow^\mu} & I \ar[d]^\alpha \\
K \ar[r]^\beta & J \\
} \]
(where $I \times_{/J} K$ is the slice category) is homotopy exact (in particular the ones from axiom FDer4 left and FDer4 right are).
\item A Cartesian square
\[ \xymatrix{ I \times_{J} K \ar[r]^-B \ar[d]_{A}  & I \ar[d]^\alpha \\
K \ar[r]^\beta & J \\
} \]
(where $I \times_{J} K$ is the fiber product) is homotopy exact, if $\alpha$ is a Grothendieck opfibration or if $\beta$ is a Grothendieck fibration.
\item If $\alpha: I \rightarrow J$ is a morphism of Grothendieck opfibrations over a diagram $E$, then 
\[ \xymatrix{ I_{e}  \ar[r]^{w_{I}} \ar[d]_{\alpha_e}  & I \ar[d]^{\alpha} \\
J_{e}  \ar[r]^{w_{J}} & J \\
} \]
is homotopy exact for all objects $e \in E$.
\item If a square
\[ \xymatrix{ L \ar[r]^B \ar[d]_{A} \ar@{}[dr]|{\Swarrow^\mu} & I \ar[d]^\alpha \\
K \ar[r]^\beta & J \\
} \]
is homotopy exact then so is the square
\[ \xymatrix{ L\times X \ar[r]^B \ar[d]_{A} \ar@{}[dr]|{\Swarrow^\mu} & I \times X \ar[d]^\alpha \\
K \times X \ar[r]^\beta & J \times X \\
} \]
for any diagram $X$. 
\end{enumerate}
\end{PROP}
\begin{proof} This proof is completely analogous to the non-fibered case. We sketch the arguments here (for the left-case only, the other case follows by logical duality):

3.
Let $j$ be an object in $J_e$ and consider the cube:
\begin{equation}\label{eqcube1}
 \xymatrix{
& I_e \times_{/J_e}  j  \ar@{}[ddrr]|(.7){\Swarrow^\mu} \ar[rr]^{\iota_e} \ar[dl]_{w} \ar[dd]^(.7){p_e} && I_e \ar[dd]^{\alpha_e} \ar[ld]_{w_I} \\ 
I \times_{/J} j \ar@{}[ddrr]|(.7){\Swarrow^\mu} \ar[rr]^(.4){\iota} \ar[dd]^{p} && I \ar[dd]_(.7){\alpha} & \\
& \cdot \ar@{=}[dl]_{} \ar[rr]^{j} && J_e \ar[dl]^{w_J} \\
\cdot \ar[rr]^{j} && J }
\end{equation}
where $w$ is given by the inclusions $\iota_{I,e}$ resp.\@ $\iota_{J,e}$.
By standard arguments on homotopy exact squares it suffices to show that the left square is homotopy exact on constant diagrams, i.e.\@ that
\[ p_{e,!} w^* \cong p_!\]
holds true for all usual derivators. By \cite[Proposition 1.23]{Gro13} it suffices to show that $w$ has a left adjoint. 

Denote $\pi_I: I \rightarrow E$ and $\pi_J: J \rightarrow E$ the given opfibrations. Consider the two functors
\[ \xymatrix{   I_e \times_{/J,e} j \ar@<2pt>[rr]^{w}  && \ar@<2pt>[ll]^c I \times_{/J} j   }\]
where $c$ is given by mapping $(i, \mu: \alpha(i) \rightarrow j)$ to
$(i', \mu': \alpha(i')  \rightarrow j)$ where
we chose, for any $i$, a coCartesian morphism $\xi_{i, \mu}: i \rightarrow i'$ over $\pi_I(\mu): \pi_I(i) \rightarrow e$.
Since $\alpha$ maps coCartesian morphisms to coCartesian morphisms by assumption, 
$\alpha(\xi_{i, \mu}): \alpha(i) \rightarrow \alpha(i')$ is coCartesian, and therefore there is a unique factorization
\[ \xymatrix{ \alpha(i) \ar[r]^-{\alpha(\xi_i)} &  \alpha(i') \ar[r]^-{\mu'} & j } \]
of $\mu$. A morphism $\alpha: (i_1, \mu_1: \alpha(i_1) \rightarrow j) \rightarrow (i_2, \mu_2: \alpha(i_2) \rightarrow j)$, by definition of coCartesian, gives rise to a unique morphism $\alpha': i'_1 \rightarrow i'_2$ over $\pi_I(i_1) \rightarrow \pi_I(i_2)$ such that $\alpha' \xi_{i_1,\mu_1} = \xi_{i_2, \mu_2} \alpha'$ holds, and we set $c(\alpha):=\alpha'$.
We have $c \circ w = \id$, and a morphism $\id_{I \times_{/J} j} \rightarrow w \circ c$ given by
$(i, \mu) \mapsto \xi_{i, \mu}$. This makes $w$ right adjoint to $c$.

2.  By axiom (Der2) it suffices to show that for any object $k$ of $K$, the induced morphism
\[ k^* A_!  B^* \rightarrow k^* \beta^* \alpha_! \]
is an isomorphism. Consider the following pasting diagram
\[ \xymatrix{
I \times_{J} k \ar[d]^\pi \ar[r]^-j & I \times_{J} K \times_{/K} k \ar@{}[dr]|{\Swarrow^\mu} \ar[r]^-\iota \ar[d]^p &I \times_{J} K \ar[r]^-{B} \ar[d]^A & I \ar[d]^\alpha \\
\cdot \ar@{=}[r] & \cdot \ar[r]^-k & K \ar[r]^-\beta & J
} \]

Lemma \ref{LEMMAPASTING} shows that the following composition
\[  \pi_! \SSS(\beta \ast \mu \ast j)_\bullet j^* \iota^* B^*  \rightarrow \pi_! j^* \SSS(\beta \ast \mu)_\bullet \iota^* B^*  \rightarrow p_! \SSS(\beta \ast  \mu)_\bullet \iota^* B^* \rightarrow k^* A_!  B^* \rightarrow k^* \beta^* \alpha_! \]
is the base-change associated with the pasting of the 3 squares in the diagram. All morphisms in this sequence are isomorphisms except possibly for the rightmost one. 
The second from the left is an isomorphism because $j$ is a right adjoint \cite[Proposition 1.23]{Gro13}. The base-change morphism of the pasting is an isomorphism because of 3.

1.  By axiom (Der2) it suffices to show that for any object $k$ of $K$ the induced morphism
\[ k^* A_! \SSS(\mu)_\bullet B^* \rightarrow k^*\beta^* \alpha_! \]
is an isomorphism. Consider the following pasting diagram
\[ \xymatrix{
I \times_{/J} k \ar[r]^\iota \ar[d]^p &I \times_{/J} K \ar@{}[dr]|{\Swarrow^\mu} \ar[r]^{B} \ar[d]^A & I \ar[d]^\alpha \\
\cdot \ar[r]^k & K \ar[r]^\beta & J
} \]
Lemma \ref{LEMMAPASTING} shows that the following diagram is commutative
\[ \xymatrix{
p_! \SSS(\mu \iota)_\bullet \iota^* B^* \ar[d]^\sim_{\can} \ar[r]^\sim & k^* \beta^* \alpha_! \\
p_! \iota^* \SSS(\mu)_\bullet B^* \ar[r]^\sim & k^* A_! \SSS(\mu)_\bullet B^* \ar[u] 
}\]
where the bottom horizontal morphism is an isomorphism by 2., and the top horizontal morphism
is an isomorphism by (FDer4 left). Therefore the right vertical morphism is also an isomorphism.

4. (cf.\@ also \cite[Theorem 1.30]{Gro13}). For any $x \in X$ consider the cube

\begin{equation}\label{eqcube1}
 \xymatrix{
& L \ar[rr]^{B} \ar@{}[ddrr]|(.65){\Swarrow^\mu} \ar[dl]_{(\id,x)} \ar[dd]^(.25){A} && I \ar[dd]^{\alpha} \ar[ld]_{(\id,x)} \\ 
L \times X \ar@{}[ddrr]|(.65){\Swarrow^\mu} \ar[rr]_(.35){B} \ar[dd]_{A} && I \times X \ar[dd]^(.65){\alpha} & \\
& K \ar[rr]^(.35){\beta} \ar[dl]_(.4){(\id,x)} && J \ar[dl]^{(\id,x)} \\
K \times X \ar[rr]^{\beta} && J \times X }
\end{equation}

The left and right hand side squares are homotopy exact because of 3., whereas the rear one is homotopy exact by assumption.
Therefore the pasting
\[ \xymatrix{ L  \ar[r] \ar[d]_A   & I \times X \ar[d]^\alpha \\
K \ar[r] & J \times X \\
} \]
 is homotopy exact. Therefore we have an isomorphism
\[ (\id,x)^* A_! \SSS(\mu)_\bullet  B^* \rightarrow (\id,x)^* \beta^* \alpha_! \]
where the morphism is induced by the base change of the given 2-commutative square. We may then conclude by axiom (Der2).
\end{proof}

\begin{PAR}[left]\label{PARPUSHFWDBASECHANGELEFT} If $\SSS$ is strong the pull-backs and push-forwards along a morphism in $\SSS(\cdot)$, or more generally along a morphism in $\SSS(I)$, can be expressed using only the relative Kan-extension functors:
 
Let $p: \DD \rightarrow \SSS$ be a left fibered derivator such that $\SSS$ is strong.
Consider the 2-commutative square
\[ \xymatrix{ I \ar@{=}[r] \ar@{=}[d] \ar@{}[dr]|{\Swarrow^\mu} & I \ar[d]^p \\
I \ar[r]_-\iota & I \times \Delta_1 \\
} \]
and consider a morphism $f: S \rightarrow T$ in $\SSS(I)$. By the strongness of $\SSS$, the morphism $f$ may be lifted to an
object $F \in \SSS(I \times \Delta_1)$, and this means that the morphism 
\[ \SSS(\mu)_\bullet: p^* F \rightarrow \iota^* F   \]
is isomorphic to $f$. Since the square is homotopy exact by Proposition~\ref{PROPHOMCART} 1., we get that the natural transformation
\[ f_\bullet \rightarrow \iota^* p_!   \]
is an isomorphism.
\end{PAR}

\begin{PAR}[left]
 Let $\alpha: I \rightarrow J$ a functor in $\Dia$ and let $f: S \rightarrow T$ be a morphism in $\SSS(J)$.
Axiom (FDer0) of a left fibered derivator implies that we have a canonical isomorphism
\[  (\alpha^*(f))_\bullet \alpha^* = \alpha^* f_\bullet   \]
which is determined by the choice of the push-forward functors. We get an associated exchange morphism
\begin{equation}\label{eqpushfwdlimit} 
\alpha_!  (\alpha^*(f))_\bullet  \rightarrow f_\bullet \alpha_!.  
\end{equation}
\end{PAR}

\begin{PROP}\label{PROPPUSHFWDLIMIT}
If $p: \DD \rightarrow \SSS$ is a left fibered derivator then the natural transformation (\ref{eqpushfwdlimit}) is an isomorphism.
The corresponding dual statement holds for a right fibered derivator.
\end{PROP}
\begin{proof}
Consider the following 2-commutative squares (the third and fourth are even commutative on the nose):
\begin{equation} \label{eqsq} \xymatrix{ I \ar@{=}[r] \ar@{=}[d] \ar@{}[dr]|{\Swarrow^{\mu_I}} & I \ar[d]^{p_I} \\
I \ar[r]_-{\iota_I} & I \times \Delta_1 \\
}\quad \xymatrix{ J \ar@{=}[r] \ar@{=}[d] \ar@{}[dr]|{\Swarrow^{\mu_J}} & J \ar[d]^{p_J} \\
J \ar[r]_-{\iota_J} & J \times \Delta_1 \\
}\quad \xymatrix{ I \ar[r]^-\alpha \ar[d]^{\iota_I} & J \ar[d]^{\iota_J} \\
I \times \Delta_1 \ar[r]_-{\alpha} & J \times \Delta_1 \\
}\quad \xymatrix{ I \ar[r]^-{p_I} \ar[d]^{\alpha} & I \times \Delta_1 \ar[d]^{\alpha} \\
J \ar[r]_-{p_J} & J \times \Delta_1 \\
} \end{equation}
They are all homotopy exact.
Consider the diagram
\[ \xymatrix{
\alpha_!  (\alpha^*(f))_\bullet  \ar[d] \ar[rr] && f_\bullet \alpha_!  \ar[d] \\
\alpha_!  \iota_I^* p_{I,!}  \ar[r] &  \iota_J^* \alpha_! p_{I,!}   \ar[r] & \iota_J^* p_{J,!} \alpha_! 
} \]
where the vertical morphisms come from (\ref{PARPUSHFWDBASECHANGELEFT}) --- these are the base change morphism for the first and second square in (\ref{eqsq}) --- and
the lower horizontal morphisms are respectively the base change for the third diagram in (\ref{eqsq}), and the  natural morphism associated with the commutativity of the fourth diagram in (\ref{eqsq}).
Repeatedly appying Lemma~\ref{LEMMAPASTING} shows that this diagram is commutative. Therefore the upper
horizontal morphism is an isomorphism because all the others in the diagram are.
\end{proof}

\begin{PAR}\label{REDUNDANCY}
The last proposition states that push-forward commutes with homotopy colimits (left case) and pull-back commutes with homotopy limits (right case). 
This is also the content of
(FDer5 left/right) for fibered derivators (not multiderivators), and hence this axiom is implied by the other axioms of left fibered derivators.
Even in the multi-case, by Lemma~\ref{LEMMALEFTRIGHT}, axiom (FDer5 left/right) also follow from both (FDer0 left) {\em and} (FDer0 right). 
\end{PAR}

\begin{PAR}[left]\label{PAREXTENSIONKAN}
Let $\alpha: I \rightarrow J$ be a functor in $\Dia$.
Proposition \ref{PROPPUSHFWDLIMIT} allows us to extend the functor $\alpha_!$ to a functor 
\[ \alpha_!: \DD(I) \times_{\SSS(I)} \SSS(J) \rightarrow \DD(J) \]
which is still left adjoint to $\alpha^*$, more precisely: to $(\alpha^*, p(J))$. Here the fiber product is formed w.r.t.\@ $p(I)$ and $\alpha^*$ respectively.
We sketch its construction:
$\alpha_!(\mathcal{E}, S)$ is given by $\alpha_!^S \mathcal{E}$, where $\alpha_!^S$ is the functor from axiom (FDer3 left) with base $S$.
Let a pair of a morphism $f: S \rightarrow T$ in $\SSS(J)$ and $F: \mathcal{E} \rightarrow \mathcal{F}$ in $\DD(I)$ over $\alpha^*(f)$ be given. 
We define $\alpha_!(F, f)$ as follows:
$F$ corresponds to a morphism
\[ (\alpha^*f)_\bullet \mathcal{E} \rightarrow \mathcal{F}. \]
Applying $\alpha_!^T$ we get a morphism
\[  \alpha_!^T (\alpha^*f)_\bullet \mathcal{E} \rightarrow \alpha_!^T \mathcal{F} \]
and composition with the inverse of the morphism (\ref{eqpushfwdlimit}) yields
\[ f_\bullet  \alpha_!^S \mathcal{E} \rightarrow \alpha_!^T \mathcal{F} \]
or, equivalently, a morphism which we define to be $\alpha_!(F, f)$
\[ \alpha_!^S \mathcal{E} \rightarrow \alpha_!^T \mathcal{F} \]
over $f$.

For the adjunction, we have to give a functorial isomorphism
\[ \Hom_{\alpha^* f} (\mathcal{E}, \alpha^*\mathcal{F}) \cong \Hom_f( \alpha_! (\mathcal{E}, S), \mathcal{F}),   \]
where $\mathcal{E} \in \DD(I)_{\alpha^* S}$ and $\mathcal{F} \in \DD(J)_T$.
We define it to be the following composition of isomorphisms:
\begin{eqnarray*}
&&\Hom_{\alpha^* f} (\mathcal{E}, \alpha^*\mathcal{F})  \\
&\cong&\Hom_{\id_{\alpha^*T}} ((\alpha^* f)_\bullet \mathcal{E}, \alpha^*\mathcal{F}) \\
&\cong&\Hom_{\id_T} (\alpha_! (\alpha^* f)_\bullet \mathcal{E}, \mathcal{F}) \\
&\cong&\Hom_{\id_T} (f_\bullet \alpha_!  \mathcal{E}, \mathcal{F}) \\
&\cong& \Hom_f( \alpha_! \mathcal{E}, \mathcal{F}).
\end{eqnarray*}
A dual statement holds for a right fibered derivator and the functor $\alpha_*$.
\end{PAR}

From Proposition \ref{PROPPUSHFWDLIMIT} we also get a vertical version of Lemma~\ref{LEMMAPASTING}:

\begin{LEMMA}[left]\label{LEMMAPASTINGV}
Given a ``pasting'' diagram
\[ \xymatrix{
N \ar@{}[dr]|{\Swarrow^\nu} \ar[r]^B \ar[d]_{\Gamma} & M \ar[d]^\gamma \\
 L \ar@{}[dr]|{\Swarrow^\mu} \ar[r]^{b} \ar[d]_a & I \ar[d]^\alpha \\
K \ar[r]^\beta & J 
}\]
we get for the pasted natural transformation $\mu \odot \nu := (\mu \ast \Gamma) \circ (\alpha \ast \nu)$ that the following diagram is commutative:
\[ \xymatrix{
a_! \SSS(\mu)_\bullet \Gamma_! \SSS(\alpha \ast \nu)_\bullet B^* \ar[r] \ar[d]^\sim  & a_! \SSS(\mu)_\bullet b^* \gamma_! \ar[r] &  \beta^* \alpha_! \gamma_!  \\
a_! \Gamma_! \SSS(\mu \ast \Gamma)_\bullet \SSS(\alpha \ast  \nu)_\bullet B^* \ar[d]^\sim \\ 
a_! \Gamma_! \SSS(\mu \odot \nu)_\bullet B^* \ar[rruu] 
} \]
Here the morphisms going to the right are (induced by) the various base-change morphisms and the upper horizontal morphism is the isomorphism from Proposition \ref{PROPPUSHFWDLIMIT}.
In particular, the pasted square is homotopy exact if the two individual squares are.
\end{LEMMA}

\subsection{Transitivity}\label{SECTTRANS}

\begin{PROP}
Let 
\[ \xymatrix{ \EE \ar[r]^{p_1} & \DD \ar[r]^{p_2} & \SSS } \]
be two left (resp.\@ right) fibered multiderivators. Then also the composition $p_3=p_2 \circ p_1: \EE \rightarrow \SSS$ is a left (resp.\@ right) fibered multiderivator.
\end{PROP}
\begin{proof}
We will show the statement for left fibered multiderivators. The other statement follows by logical duality. 

Axiom (FDer0): For any $I \in \Dia$, we have a sequence
\[ \EE(I) \rightarrow \DD(I) \rightarrow \SSS(I) \]
of fibered multicategories. It is well-known that then also the composition $\EE(I) \rightarrow \SSS(I)$ is a fibered multicategory (see \ref{APPMULTI}). The other statement of (FDer0) is immediate, too.
Let $\alpha: I \rightarrow J$ be a functor as in axioms (FDer3 left) and (FDer4 left).
We denote the relative homotopy Kan-extension functors w.r.t.\@ the 2 fibered derivators by $\alpha^1_!$, and $\alpha^2_!$, respectively. As always, the base will be understood
from the context or explicitly given as extra argument as in (\ref{PAREXTENSIONKAN}).

Axiom (FDer3 left): 
Let $S \in \SSS(J)$ be given.
We define a functor
\[ \alpha_!^3: \EE(I)_{\alpha^*S} \rightarrow \EE(J)_S \]
in the fiber (under $p_2$) of $\mathcal{E} \in \DD(I)_{\alpha^*S}$  as the composition
\[ \xymatrix{ \EE(I)_{\alpha^*S} \ar[rr]^-{(\nu_\bullet, \alpha_!^2 p_1)} && \EE(I)_{\alpha^*S} \times_{\DD(I)_{\alpha^*S}} \DD(J)_{S} \ar[r]^-{\alpha^1_!} &  \EE(J)_S } \]
where $\nu$ is the unit
\[ \nu:  \mathcal{E} \rightarrow \alpha^* \alpha_!^2 \mathcal{E} \]
and $\alpha_!^1$ with two arguments is the extension given in (\ref{PAREXTENSIONKAN}).

Let $\mathcal{F}_1 \in \EE(I)_{\alpha^*S}$ and  $\mathcal{F}_2 \in \EE(J)_{S}$ be given with images $\mathcal{E}_1$ and $\mathcal{E}_2$, respectively under $p_1$.
The adjunction is given by the following composition of isomorphisms:
\[ \begin{array}{rcll}
& & \Hom_S(\alpha^3_! \mathcal{F}_1, \mathcal{F}_2) \\
&=& \Hom_S(\alpha^1_!( \nu_\bullet \mathcal{F}_1, \alpha_!^2 \mathcal{E}_1), \mathcal{F}_2) & \text{Definition}\\
&=& \{ f \in \Hom_{S}( \alpha_!^2 \mathcal{E}_1,  \mathcal{E}_2) ; \xi \in \Hom_f(\alpha^1_!( \nu_\bullet \mathcal{F}_1, \alpha_!^2 \mathcal{E}_1), \mathcal{F}_2) \} & \text{Definition} \\
&\cong& \{ f  \in \Hom_{S}( \alpha_!^2 \mathcal{E}_1,  \mathcal{E}_2) ;  \xi \in \Hom_{\alpha^* f} ( \nu_\bullet \mathcal{F}_1, \alpha^* \mathcal{F}_2) \} & \text{Adjunction (\ref{PAREXTENSIONKAN})} \\
&\cong& \{ \widetilde{f} \in \Hom_{\alpha^*S}( \mathcal{E}_1,  \alpha^* \mathcal{E}_2) ;  \xi \in \Hom_{\widetilde{f}} ( \mathcal{F}_1, \alpha^* \mathcal{F}_2) \} & \text{Note below} \\
&=&  \Hom_{\alpha^*S}(\mathcal{F}_1, \alpha^* \mathcal{F}_2) \}  & \text{Definition}
\end{array} \]
Note that the composition 
\[ \widetilde{f}: \xymatrix{ \mathcal{E}_1 \ar[r]^-\nu & \alpha^* \alpha_!^2 \mathcal{E}_1 \ar[r]^-{\alpha^* f} & \alpha^* \mathcal{E}_2  }\]
is determined by $f$ via the adjunction of (FDer3 left) for base $S$ and $p_2: \DD \rightarrow \SSS$.

Axiom (FDer4 left):
Let $\mathcal{E}$ be in $\EE(I)_{\alpha^*S}$ and let $\mathcal{F}$ be its image under $p_1$.
We have to show that the natural morphism 
\[ \alpha_{j!}^3 \SSS(\mu)_\bullet^3 \iota^* \mathcal{E} \rightarrow j^* \alpha_{!}^3 \]
is an isomorphism. Inserting the definition of the push-forwards, resp.\@ of the Kan extensions for $p_3$, we get
\[ \alpha_{j!}^1 (\nu_j)_\bullet^1 \text{cart}_\bullet^1 \iota^* \mathcal{E} \rightarrow j^* \alpha_{!}^1 \nu_\bullet^1 \mathcal{E}. \]
Here $\nu_j: \SSS(\mu)^2_\bullet \iota^* \mathcal{F} \rightarrow \alpha_j^* \alpha_{j!}^2 \SSS(\mu)^2_\bullet \iota^* \mathcal{F}$ is the unit and
 $\nu: \mathcal{F} \rightarrow \alpha^* \alpha_!^2 \mathcal{F}$ is the unit.
 `$\text{cart}^1$' is the Cartesian morphism $\iota^* \mathcal{F} \rightarrow \SSS(\mu)_\bullet^2 \iota^* \mathcal{F}$.
Consider the base-change isomorphism (FDer4 for $p_2$) 
\[ \text{bc}:  \alpha_{j!}^2 \SSS(\mu)^2_\bullet \iota^* \mathcal{F} \rightarrow  j^* \alpha_!^2 \mathcal{F},  \]
and the morphism 
\[ \DD(\mu): \iota^* \alpha^* \alpha_!^2 \mathcal{F} \rightarrow \alpha_j^* j^* \alpha_!^2 \mathcal{F}. \]

{\em Claim:} We have the equality 
\[ (\alpha_j^*\text{bc}) \circ  \nu_j \circ \text{cart} = \DD(\mu) \circ \iota^*(\nu). \]

{\em Proof of the claim:} Consider the diagram (which affects only the fibered derivator $p_2: \DD \rightarrow \SSS$, hence we omit superscripts):
\[ \xymatrix{
\alpha_j^* \alpha_{j!} \SSS(\mu)_\bullet \iota^* \mathcal{F} \ar[rr]_-{\alpha_j^* \alpha_{j!} \SSS(\mu)_\bullet \iota^* \nu}  \ar@/^20pt/[rrrr]^{\alpha_j^* \text{bc}}  & &
\alpha_j^* \alpha_{j!} \SSS(\mu)_\bullet  \iota^* \alpha^* \alpha_! \mathcal{F} \ar[r]  & 
\alpha_j^* \alpha_{j!} \alpha_j^* j^* \alpha_! \mathcal{F} \ar[r] &
 \alpha_j^* j^* \alpha_! \mathcal{F} \\
\SSS(\mu)_\bullet \iota^* \mathcal{F} \ar[rr]_-{\SSS(\mu)_\bullet \iota^* \nu} \ar[u]^{\nu_j} & &
\SSS(\mu)_\bullet  \iota^* \alpha^* \alpha_! \mathcal{F} \ar[r]^{\text{induced}} \ar[u] & 
 \alpha_j^* j^* \alpha_! \mathcal{F} \ar@{=}[ur] \ar[u]  \\
 \iota^* \mathcal{F} \ar[rr]_-{\iota^* \nu} \ar[u]^{\text{cart}}  & &
 \iota^* \alpha^* \alpha_! \mathcal{F}  \ar[ur]_{\DD(\mu)} \ar[u]^{\text{cart}} & 
}\] 
Clearly all squares and triangles  in this diagram are commutative. The two given morphisms are the compositions of the extremal paths hence they are equal.

We have a natural isomorphism induced by $\text{bc}$:
\[ \alpha^1_{j!}(\cdots, \alpha_{j!}^2 \SSS(\mu)^2_\bullet \iota^* \mathcal{F}) \cong  \alpha^1_{j!}((\alpha_j^*\text{bc})_\bullet( \cdots), j^* \alpha_!^2 \mathcal{F})  \]
(this is true for any isomorphism).

We therefore have
\begin{eqnarray*}
&& \alpha_{j_!}^1 (\nu_j)_\bullet^1 \text{cart}_\bullet^1 \iota^* \mathcal{E}  \\ 
&\cong& \alpha_{j_!}^1 (\alpha_j^* \text{bc})_\bullet^1 (\nu_j)_\bullet^1 \text{cart}_\bullet^1 \iota ^* \mathcal{E} \\
&\cong& \alpha_{j_!}^1 \DD(\mu)_\bullet^1 (\iota^* \nu)_\bullet^1 \iota^* \mathcal{E}  \\ 
&\cong& \alpha_{j_!}^1 \DD(\mu)_\bullet^1 \iota^* \nu_\bullet^1 \mathcal{E}
\end{eqnarray*}

Thus we are left to show that
\[  \alpha_{j!}^1 \DD(\mu)_\bullet^1 \iota^* \nu_\bullet^1 \mathcal{E}  \rightarrow j^* \alpha_{!}^1 \nu_\bullet^1 \mathcal{E}  \]
is an isomorphism.
A tedious check shows that this {\em is} the base change morphism associated with $p_1$. It is an isomorphism by (FDer4 left) for $p_1$.
\end{proof}

\subsection{(Co)Local morphisms}\label{SECTLOCAL}

\begin{PAR}
Let $\Dia$ be a diagram category and let $\SSS$ be a {\em strong} right derivator with domain $\Dia$. Strongness implies that for each diagram
\[ \xymatrix{
  & U \ar[d] \\
S \ar[r] & T
} \]
in $\SSS(\cdot)$ there exists a homotopy pull-back $\mlq\mlq U \times_T S \mrq \mrq$ which is well-defined up to (non-unique!) isomorphism.
A Grothendieck pre-topology on $\SSS$ is basically a Grothendieck pre-topology in the usual sense on $\SSS(\cdot)$ except
that pull-backs are replaced by homotopy pull-backs. We state the precise definition: 
\end{PAR}

\begin{DEF}\label{DEFPRETOP}
A {\bf  Grothendieck pre-topology} on $\SSS$ is the datum consisting of, for any $S \in \SSS(\cdot)$, a collection of families $\{U_i \rightarrow S\}_{i \in S}$ of morphisms 
in $\SSS(\cdot)$ called {\bf covers}, such that
\begin{enumerate}
\item Every family consisting of isomorphisms is a cover, 
\item If $\{U_i \rightarrow S\}_{i \in S}$ is a cover and $T \rightarrow S$ is any morphism then
 the family $\{ \mlq\mlq U_i \times_{S} T \mrq\mrq \rightarrow T \}_{i \in I}$ is a cover for any
choice of particular members of the family $\{\mlq\mlq U_i \times_{S} T \mrq\mrq\}$. 
\item If $\{U_i \rightarrow S\}_{i \in I}$ is a cover and for each $i$, the family $\{U_{i,j} \rightarrow U_i \}_{j \in J_i}$ is a cover then the family of compositions
$\{U_{i,j} \rightarrow U_i \rightarrow S \}_{i \in I, j \in J_i}$ is a cover.
\end{enumerate}
\end{DEF}

\begin{DEF}[left] \label{DEFDLOCAL}
Let $p: \DD \rightarrow \SSS$ be a left fibered derivator satisfying also (FDer0 right). Assume that pull-backs exist in $\SSS$.
We call a morphism $f: U \rightarrow X$ in $\SSS(\cdot)$ {\bf $\DD$-local} if

\begin{itemize}
\item[(Dloc1 left)] 
The morphism $f$ satisfies {\bf base change}: for any diagram $Q \in \DD(\Box)$  with underlying diagram
\[ \xymatrix{A \ar[r]^-{\widetilde{F}} \ar[d]_-{\widetilde{G}} & B \ar[d]^{\widetilde{g}} \\ C \ar[r]_{\widetilde{f}} & D } \]
such that $p(Q)$ in $\SSS(\Box)$ is a pull-back-diagram, i.e.\@ is (homotopy)
Cartesian, the following holds true: If $\widetilde{F}$ and ${\widetilde{f}}$ are Cartesian, and $\widetilde{g}$ is coCartesian then also $\widetilde{G}$ is 
coCartesian.\footnote{In other words, if 
\[ \xymatrix{\mlq\mlq U \times_X Y \mrq\mrq  \ar[r]^-F \ar[d]_-G & Y \ar[d]^-g \\ U \ar[r]_-f & X } \]
is the underlying diagram of $p(Q)$ then the exchange morphism
\[ G_\bullet F^\bullet \rightarrow f^\bullet g_\bullet  \]
is an isomorphism.}

\item[(Dloc2 left)] The morphism of derivators (cf.\@ Lemma~\ref{LEMMAMORPHDERRIGHT}) 
\[ f^\bullet: \DD_X \rightarrow \DD_U \]
commutes with homotopy colimits.
\end{itemize}
A morphism $f: U \rightarrow X$ in $\SSS(\cdot)$ is called {\bf universally $\DD$-local} if any homotopy pull-back of $f$ is $\DD$-local.
\end{DEF}

\begin{DEF}[left]\label{DEFFIBEREDLOCAL} Assume that $\SSS$ is equipped with a Grothendieck pre-topology (cf.\@ \ref{DEFPRETOP}).
A left fibered derivator $p: \DD \rightarrow \SSS$ as in Definition \ref{DEFDLOCAL} is called {\bf local} w.r.t.\@ the pre-topology on $\SSS$, if
the following conditions hold: 
\begin{enumerate}
\item Every morphism $U_i \rightarrow S$ which is part of a cover is $\DD$-local.
\item For a cover $\{f_i: U_i \rightarrow S\}$ the family
\[ (f_i)^\bullet: \DD(S) \rightarrow \DD(U_i) \]
is jointly conservative.
\end{enumerate}
\end{DEF}

\begin{DEF}[right] \label{DEFDCOLOCAL}
Let $p: \DD \rightarrow \SSS$ be a right fibered derivator satisfying also (FDer0 left). Assume that push-outs exist in $\SSS$.
We call a morphism $f: X \rightarrow U$ in $\SSS(\cdot)$ {\bf $\DD$-colocal} if

\begin{itemize}
\item[(Dloc1 right)] 
The morphism $f$ satisfies {\bf base change}: for any diagram $Q \in \DD(\Box)$ with underlying diagram:
\[ \xymatrix{A   & \ar[l]_{\widetilde{F}} B  \\ C \ar[u]^{\widetilde{G}}  & D \ar[u]_{\widetilde{g}} \ar[l]^{\widetilde{f}} } \]
such that $p(Q)$ in $\SSS(\Box)$ is a pushout-diagram, i.e.\@ is (homotopy)
coCartesian, 
if $\widetilde{F}$ and ${\widetilde{f}}$ are coCartesian, and $\widetilde{g}$ is Cartesian  then also $\widetilde{G}$ is Cartesian.  

\item[(Dloc2 right)] The morphism of derivators (cf.\@ Lemma~\ref{LEMMAMORPHDERLEFT}) 
\[ f_\bullet: \DD_X \rightarrow \DD_U \]
commutes with homotopy limits.
\end{itemize}
A morphism $f: X \rightarrow U$ in $\SSS(\cdot)$ is called {\bf universally $\DD$-colocal} if any homotopy push-out of $f$ is $\DD$-colocal.
\end{DEF}

\begin{DEF}[right]\label{DEFFIBEREDCOLOCAL} Assume that $\SSS$ is equipped with a Grothendieck pre-cotopology, i.e.\@ that $\SSS^{\op}$ is equipped with a Grothendieck pre-topology (cf.\@ \ref{DEFPRETOP}).
A right fibered derivator $p: \DD \rightarrow \SSS$ as in Definition \ref{DEFDCOLOCAL} is called {\bf colocal} w.r.t.\@ the pre-cotopology on $\SSS$, if
\begin{enumerate}
\item Every morphism $S \rightarrow U_i$ which is part of a cocover is $\DD$-colocal.
\item For a cocover $\{f_i: S \rightarrow U_i\}$ the family
\[ (f_i)_\bullet: \DD(\cdot)_S \rightarrow \DD(\cdot)_{U_i} \]
is jointly conservative.
\end{enumerate}
\end{DEF}

\subsection{The associated pseudo-functor}\label{SECTPSEUDOFUNCT}

Let $p: \DD \rightarrow \SSS$ be a morphism of pre-derivators with domain $\Dia$.

\begin{PAR}[left] \label{DEFDIA}
Let $\Dia(\SSS)$ be the 2-category of diagrams over $\SSS$, where the objects are
pairs $(I, F)$ such that $I \in \Dia$ and $F \in \SSS(I)$, the morphisms $(I, F) \rightarrow (J,G)$ are pairs $(\alpha, f)$ such that
$\alpha: I \rightarrow J, f: F  \rightarrow \alpha^* G$ and the 2-morphisms $(\alpha, f) \rightarrow (\beta, g)$ 
are the natural transformations $\mu: \alpha \Rightarrow \beta$ satisfying $\SSS(\mu)(G) \circ f= g$.

We call a morphism $(\alpha, f)$ of {\bf fixed shape} if $\alpha=\id$, and of {\bf diagram type} if $f$ consists of identities. 
Every morphism is obviously a composition of one of diagram type by one of fixed shape. 
\end{PAR}

\begin{PAR}[right] \label{DEFDIAOP}
There is a dual notion of a 2-category $\Dia^{\op}(\SSS)$. Explicitly, the objects are
pairs $(I, F)$ such that $I \in \Dia$ and $F \in \SSS(I)$, the morphisms $(I, F) \rightarrow (J,G)$ are pairs $(\alpha, f)$ such that
$\alpha: I \rightarrow J, f: \alpha^* G \rightarrow F$ and the 2-morphisms $(\alpha, f) \rightarrow (\beta, g)$ 
are the natural transformations $\mu: \alpha \Rightarrow \beta$ satisfying $f \circ \SSS(\mu)(G) = g$ . 

The association 
$(I, F) \mapsto (I^{\op}, F^{\op})$ induces an isomorphism $\Dia^{\op}(\SSS) \rightarrow \Dia(\SSS^{\op})^{2-\op}$.
\end{PAR}

We are interested
in associating to a fibered derivator a pseudo-functor like for classical fibered categories.

\begin{PAR}[left]\label{PARPSEUDOFUNCTOR}
We associate to a morphism of pre-derivators $p: \DD \rightarrow \SSS$ which satisfies (FDer0 right)
 a (contravariant) 2-pseudo-functor
\[ \DD: \Dia(\SSS)^{1-\op} \rightarrow \mathcal{CAT} \]
mapping a pair $(I, F)$ to $\DD(I)_{F}$, and
a morphism $(\alpha, f): (I, F) \rightarrow (J,G)$ to $f^\bullet \circ \alpha^*: \DD(J)_{G} \rightarrow \DD(I)_{F}$.
A natural transformation $\mu: \alpha \Rightarrow \beta$ is mapped to the natural transformation pasted from the following two 2-commutative triangles:
\[ \xymatrix{ 
& \DD(I)_{G \circ \alpha} \ar[dr]^{f^\bullet} & \\
\DD(J)_G \ar[ur]^{\alpha^*} \ar[rd]^{\beta^*}  \ar@{}[r]|{\Downarrow^\mu}  & \ar@{}[r]|{\Downarrow} & \DD(I)_F  \\
& \DD(I)_{G \circ \beta}  \ar[uu]|{\SSS(\mu)(G)^\bullet }  \ar[ur]^{g^\bullet} &
} \]
\begin{proof}[Proof of the pseudo-functor property.]
For a composition $(\beta, g) \circ (\alpha, f) = (\beta \circ \alpha, \alpha^*(g) \circ f)$ we have:
$f^\bullet \circ \alpha^* \circ g^\bullet \circ \beta^*  \cong f^\bullet \circ (\alpha^*g)^\bullet \circ \alpha^* \circ \beta^*$. 
This follows from the isomorphism $\alpha^* \circ g^\bullet \cong  (\alpha^*g)^\bullet \circ \alpha^*$ (FDer0). One checks that
this indeed yields a pseudo-functor.
\end{proof}
\end{PAR}

\begin{PAR}[right]
We associate to a morphism of pre-derivators $p: \DD \rightarrow \SSS$ which satisfies (FDer0 left)
 a (contravariant) 2-pseudo-functor
\[ \DD: \Dia^{\op}(\SSS)^{1-\op} \rightarrow \mathcal{CAT} \]
mapping a pair $(I, F)$ to $\DD(I)_{F(I)}$, and 
a morphism $(\alpha, f): (I, F) \rightarrow (J,G)$ to $ f_\bullet  \circ \alpha^* $ from $\DD(J)_{G} \rightarrow \DD(I)_{F}$.
This defines a functor by the same reason as in~\ref{PARPSEUDOFUNCTOR}.
\end{PAR}

\begin{PAR}[left]\label{DEFCOMMADIA}
We assume that $\SSS$ is a strong right derivator. There is a notion of ``comma object'' in $\Dia(\SSS)$ which we describe here for the case that $\SSS$ is the pre-derivator associated with a category $\mathcal{S}$ and leave
it to the reader to formulate the derivator version. 
In that case the corresponding object will be determined up to 
(non-unique!) isomorphism only. 

Given diagrams $D_1=(I_1, F_1), D_2=(I_2, F_2), D_3=(I_3, F_3)$ in $\Dia(\SSS)$ and morphisms
$\beta_1: D_1 \rightarrow D_3$, $\beta_2: D_2 \rightarrow D_3$, we form the
comma diagram $D_1 \times_{/D_3} D_2$ as follows:
the underlying diagram $I_1 \times_{/I_3} I_2$ has objects being tripels $(i_1, i_2,\mu)$ such that $i_1 \in I_1$, $i_2 \in I_2$, and $\mu: \alpha_1(i_1) \rightarrow \alpha_2(i_2)$ in $I_3$.
A morphism is a pair $\beta_j: i_{j} \rightarrow i_{j}'$ for $j=1,2$ such that 
\[ \xymatrix{ 
\alpha_1(i_{1}) \ar[r]^{\alpha_1(\beta_1)} \ar[d]^{\mu} & \alpha_1(i_{1}') \ar[d]^{\mu'} \\
\alpha_2(i_{2}) \ar[r]^{\alpha_2(\beta_2)} &  \alpha_2(i_{2}')  
}\]
commutes in $I_3$.
The corresponding functor $\widetilde{F} \in \SSS(I_1 \times_{/I_3} I_2)$ maps a tripel $(i_1, i_2, \mu)$ to
\[ F_1(i_1) \times_{F_3(\alpha_2(i_2))} F_2(i_2).  \]

We define $P_j$ to be $(\iota_j, p_j)$ for $j=1, 2$, where
$\iota_j$ maps a tripel $(i_1, i_2,\mu)$ to $i_j$, and $p_j$ is the corresponding projection of the fiber product.
We then get a 2-commutative diagram

\[ \xymatrix{
D_1 \times_{/D_3} D_2 \ar[r]^-{P_1} \ar[d]^{P_2} \ar@{}[dr]|{\Swarrow^\mu} & D_1 \ar[d]^{\beta_1} \\
D_2 \ar[r]_{\beta_2} & D_3 
}\]

If we are given $I_2, I_3$ only and two maps $I_1 \rightarrow I_3$ and $I_2 \rightarrow I_3$ we also form
$D_1 \times_{/I_3} I_2$ by the same underlying category, with functor $F_1 \circ \iota_1$.
\end{PAR}

\begin{PAR}[right]
We assume that $\SSS$ is a strong left derivator. There is a dual notion of ``comma object'' in $\Dia^{\op}(\SSS)$ which we describe here again for the case that $\SSS$ is the pre-derivator associated with a category $\mathcal{S}$ and leave
it to the reader to formulate the derivator version. 
In that case the corresponding object will be determined up to 
(non-unique!) isomorphism only. 

Given three diagrams $D_1^o=(I_1, F_1), D_2^o=(I_2, F_2)$ in $\Dia^{\op}(\SSS)$ mapping to $D_3^o=(I_3, F_3)$, we form the
comma diagram $D_1^o \times_{/D_3^o} D_2^o$ as follows:
the underlying diagram is $I_1 \times_{/D_3} I_2$ which has object being tripels $(i_1, i_2, \mu)$ such that $i_1 \in I_1, i_2 \in I_2$ and $\mu: \alpha_1(i_1) \rightarrow \alpha_2(i_2)$ in $I_3$.
A morphism is a pair $\beta_j: i_{j} \rightarrow i_{j}'$ for $j=1,2$ such that 
\[ \xymatrix{ 
\alpha_1(i_{1}) \ar[r]^{\alpha_1(\beta_1)} \ar[d]^{\mu} & \alpha_1(i_{1}') \ar[d]^{\mu'} \\
\alpha_2(i_{2}) \ar[r]^{\alpha_2(\beta_2)} &  \alpha_2(i_{2}')  
}\]
commutes in $I_3$.
The corresponding functor $\widetilde{F}$ maps a tripel $(i_1, i_2, \mu)$ to
\[ F_1(i_1) \sqcup_{F_3(\alpha_1(i_1))} F_2(i_2).  \]
We then get a 2-commutative diagram
\[ \xymatrix{
D_2^o \times_{/D_3^o} D_1^o \ar[r] \ar[d] \ar@{}[dr]|{\Swarrow^\mu} & D_1^o \ar[d] \\
D_2^o \ar[r] & D_3^o 
}\]
\end{PAR}

This language allows us to restate Lemma~\ref{LEMMAPASTING} and Lemma~\ref{LEMMAPASTINGV} in a more convenient way:
\begin{LEMMA}[left]\label{LEMMAPASTINGDIA}
\begin{enumerate}
\item Given a ``pasting'' diagram in $\Dia(\SSS)$
\[ \xymatrix{
D_1 \ar@{}[dr]|{\Swarrow^\nu} \ar[r]^\Gamma \ar[d]_A & D_3 \ar@{}[dr]|{\Swarrow^\mu} \ar[r]^{B} \ar[d]^a & D_5 \ar[d]^\alpha \\
D_2 \ar[r]^\gamma & D_4 \ar[r]^\beta & D_6
}\]
the pasted natural transformation $\nu \odot \mu := \beta \nu \circ \mu \Gamma$ satsisfies 
\[ \nu_! \odot \mu_! = (\nu \odot \mu)_!. \]
\item Given a ``pasting'' diagram in $\Dia(\SSS)$
\[ \xymatrix{
D_1 \ar@{}[dr]|{\Swarrow^\nu} \ar[r]^B \ar[d]_\Gamma & D_2 \ar[d]^\gamma \\
 D_3 \ar@{}[dr]|{\Swarrow^\mu} \ar[r]^{b} \ar[d]_a & D_4 \ar[d]^\alpha \\
 D_5 \ar[r]^\beta & D_6
}\]
the pasted natural transformation $\nu \odot \mu := \alpha \nu \circ \mu \Gamma$ satisfies 
\[ \mu_! \odot \nu_! = (\mu \odot \nu)_!. \]
\end{enumerate}
\end{LEMMA}

\begin{DEF}
If $\SSS$ is equipped with a Grothendieck pre-topology (cf.\@ \ref{DEFPRETOP}) then we call $(\alpha, f): (I, F) \rightarrow (J,G)$ {\bf $\DD$-local} if
$f_i: F(i) \rightarrow G \circ \alpha(i)$ is $\DD$-local (cf.\@ \ref{DEFDLOCAL}) for all $i \in I$.
Likewise for the notions of universally $\DD$-local, $\DD$-colocal, and universally $\DD$-colocal.
\end{DEF}

\begin{PROP}[left]\label{THEOREMBASECHANGEDIALEFT}
Let $\DD \rightarrow \SSS$ be a left fibered derivator satisfying also (FDer0 right) and such that $\SSS$ is a strong right derivator.
Then the associated pseudo-functor satisfies the following properties:

\begin{enumerate}
\item For a morphism of diagrams $(\alpha, f): D_1 \rightarrow D_2$ the corresponding pull-back
\[ (\alpha, f)^*:  \DD(D_2) \rightarrow \DD(D_1) \]
has a left-adjoint $(\alpha, f)_!$.
\item 
For a diagram like in \ref{DEFCOMMADIA}
\[ \xymatrix{
D_1 \times_{/D_3} D_2 \ar[r]^-{P_1} \ar@{}[dr]|{\Swarrow^\alpha} \ar[d]_-{P_2} & D_1 \ar[d]^{\beta_1} \\
D_2 \ar[r]_{\beta_2} & D_3 
}\]
the corresponding exchange morphism
\[ P_{2!} P_1^* \rightarrow \beta_2^* \beta_{1!} \]
is an isomorphism in $\DD(D_2)$ provided that
$\beta_2$ is $\DD$-local. 
\end{enumerate}
\end{PROP}

\begin{proof}
1. By (FDer0 left) and (FDer3 left) we can form $(\alpha, f)_!:=\alpha_! \circ f_\bullet$ which is clearly left adjoint to $(\alpha, f)^*$.

2. We first reduce to the case where $I_2$ is the trivial category. Indeed consider the diagram
\[ \xymatrix{
D_1 \times_{/D_3} (\{j_2\}, F_2(j_2)) \ar[r]^-{can.} \ar[d] \ar@{}[dr]|{\Swarrow}& D_1 \times_{/D_3} D_2 \times_{/D_2} (\{j_2\}, F_2(j_2)) \ar[r] \ar[d] \ar@{}[dr]|{\Swarrow} & D_1 \times_{/D_3} D_2 \ar[r]^-{P_1} \ar@{}[dr]|{\Swarrow} \ar[d]_-{P_2} & D_1 \ar[d]^{\beta_1} \\
(\{j_2\}, F_2(j_2)) \ar@{=}[r] & (\{j_2\}, F_2(j_2)) \ar[r] &  D_2 \ar[r]_{\beta_2} & D_3 
 } \]
The exchange morphism of the middle square and outmost rectangle are isomorphisms by the reduced case. The morphism $can.$ of the left hand square is of diagram type and its underlying diagram functor has an adjoint. 
The exchange morphism is therefore an isomorphism by~\cite[1.23]{Gro13}. Using Lemma~\ref{LEMMAPASTINGDIA} therefore, applying this for all $j_2 \in I_2$, also the exchange morphism of the right square has to be an isomorphism (this uses axiom Der2).

Now we may assume $D_2 =  (\{j_2\}, F_2(j_2))$. Consider the following diagram, in which we denote
$\beta_1 = (\alpha_1, f_1)$, $\beta_2 = (\alpha_2, f_2)$.

\[ \xymatrix{
(I_1 \times_{/I_3} \{i_2\},  \widetilde{F}) \ar[r]^{p_1} \ar[d]^{p_2}  \ar@{}[dr]|{\Swarrow \quad \numcirc{1} } & (I_1 \times_{/I_3} \{i_2\},  F_1 \circ \iota_1) \ar[r]^-{\iota_1} \ar[d]^{\iota_1^*f_1}  \ar@{}[dr]|{\Swarrow \quad \numcirc{4}} & (I_1, F_1 ) \ar[d]^-{f_1} \\
(I_1 \times_{/I_3} \{i_2\},  \widetilde{F}') \ar[r]^-{p_1'} \ar[d]^{p_2'}  \ar@{}[dr]|{\Swarrow \quad \numcirc{2}} & (I_1 \times_{/I_3} \{i_2\},  F_3 \circ \alpha_1 \circ \iota_1) \ar[r]^-{\iota_1} \ar[d]^-{F_3(\mu)}  \ar@{}[ddr]|{\substack{\Swarrow  \\ \\ \\ \\ \numcirc{5}}}  &  (I_1, F_3 \circ \alpha_1) \ar[dd]^{\alpha_1} \\
 (I_1 \times_{/I_3} \{i_2\}, F_2(i_2) ) \ar[r]^-{\iota_2^*f_2} \ar[d]^-{\iota_2}  \ar@{}[dr]|{\Swarrow  \quad \numcirc{3}} & (I_1 \times_{/I_3} \{i_2\},  F_3(\alpha_2(i_2))) \ar[d]^-{\iota_2} & \\
(\{i_2\}, F_2(i_2) ) \ar[r]^{f_2}  & (\{i_2\}, F_3(\alpha_2(i_2)) ) \ar[r]^{\alpha_2} & (I_3, F_3 )  
 } \]
 where 
 $\widetilde{F}$ is the functor defined in \ref{DEFCOMMADIA} mapping a tripel $(i_1, i_2, \mu: \alpha_1(i_1) \rightarrow \alpha_2(i_2))$ to
 \[ F_1(i_1) \times_{F_3(\alpha_2(i_2))} F_2(i_2)  \]
and
$\widetilde{F}'$ is the functor mapping a tripel $(i_1, i_2, \mu: \alpha_1(i_1) \rightarrow \alpha_2(i_2))$ to
 \[ F_3(\alpha_1(i_1)) \times_{F_3(\alpha_2(i_2))}  F_2(i_2).   \]

We have to show that the exchange morphism for the outer square is an isomorphism. Using Lemma~\ref{LEMMAPASTINGDIA} it suffices to show this for the squares 1--5.
That the exchange morphism for the squares 1 and 2, where the morphisms are of fixed shape, is an isomorphism can be checked point-wise by (Der2). Then it 
boils down to the base change condition (Dloc1 left).  
Note that the squares are pull-back squares in $\mathcal{S}$ by construction of $\widetilde{F}'$ resp $\widetilde{F}$.
The exchange morphism for 4 is an isomorphism by (FDer0). 
The exchange morphism for 3 is an isomorphism because of (Dloc2 left). 
The exchange morphism for 5 is an isomorphism because of (FDer4 left).
\end{proof}

Dualizing, there is a right-variant of the theorem, which uses $\Dia^{\op}(\SSS)$ instead. We leave its formulation to the reader.

\section{(Co)homological descent}\label{SECTCOHOMDESCENT}

\subsection{Categories of $\SSS$-diagrams}

\begin{DEF}\label{DEFSDIAGRAMCAT}
Let $\SSS$ be a strong right derivator with Grothendieck pre-topology.

A {\bf category of $\SSS$-diagrams in $\Cat(\SSS)$} is a full sub-2-category $\mathcal{DIA} \subset \Cat(\SSS)$,
satisfying the following axioms:
\begin{itemize}
\item[(SDia1)] The empty diagram $(\emptyset, -)$, the diagrams $(\cdot,S)$ for any $S \in \SSS(\cdot)$, and $(\Delta_1, f)$ for any $f \in \SSS(\Delta_1)$ are objects of $\mathcal{DIA}$.
\item[(SDia2)] $\mathcal{DIA}$ is stable under taking finite coproducts and such fibered products, where one of the morphisms is of pure diagram type.
\item[(SDia3)] For each morphism $\alpha: D_1 \rightarrow D_2$ with $D_i=(I_i, F_i)$ in $\mathcal{DIA}$ and for each object $i \in I_2$ and morphism $U \rightarrow F_2(i)$ being part of a cover in the chosen pre-topology, the slice diagram
$D_1 \times_{/D_2} (i, U)$ is in $\mathcal{DIA}$, and if $\alpha$ is of pure diagram type then also $(i, F_2(i)) \times_{/D_2} D_1$ is in $\mathcal{DIA}$. 
\end{itemize}
A {\bf category of $\SSS$-diagrams}  $\mathcal{DIA}$ is called {\bf infinite}, if it satisfies in addition:
\begin{itemize}
\item[(SDia5)] $\mathcal{DIA}$ is stable under taking arbitrary coproducts.
\end{itemize}
\end{DEF}

There is an obvious dual notion of a category of $\SSS$-diagrams in $\Cat^{\op}(\SSS)$. If $\SSS$ is the trivial derivator both definitions boil down to the previous definition of
a diagram category~\ref{DEFDIAGRAMCAT}.

\subsection{Fundamental (co)localizers}

\begin{DEF}
A class of morphisms $\mathcal{W}$ in a category is called {\bf weakly saturated}, if it satisfies the following properties:
\begin{enumerate}
\item[(WS1)] Identities are in $\mathcal{W}$.
\item[(WS2)] $\mathcal{W}$ has the 2-out-of-3 property.
\item[(WS3)] If $p: Y \rightarrow X$ and $s: X \rightarrow Y$ are morphisms such that $p \circ s = \id_X$ and $s \circ p \in \mathcal{W}$ then $p \in \mathcal{W}$ (and hence 
$s \in \mathcal{W}$ by (WS2)). 
\end{enumerate}
\end{DEF}

\begin{DEF}\label{DEFFUNLOCREL}
Let $\SSS$ be a strong right derivator with Grothendieck pre-topology (\ref{DEFPRETOP}).
Let $\mathcal{DIA} \subset \Cat(\SSS)$ be a category of $\SSS$-diagrams (cf.\@ \ref{DEFSDIAGRAMCAT}). 

Consider a family of subclasses $\mathcal{W}_S$ of morphisms in $\mathcal{DIA} \times_{/\mathcal{DIA}} (\cdot, S)$ parametrized by all objects $S \in \SSS(\cdot)$. 
Such a family $\{\mathcal{W}_S\}_S$ is called a {\bf system of relative localizers} if
 the following properties are satisfied:
\begin{enumerate}
\item[(L0)] For any morphism $S_1 \rightarrow S_2$ the induced functor $\mathcal{DIA} \times_{/\mathcal{DIA}} (\cdot, S_1) \rightarrow \mathcal{DIA} \times_{/\mathcal{DIA}} (\cdot, S_2)$ maps $\mathcal{W}_{S_1}$ to $\mathcal{W}_{S_2}$.
\item[(L1)] Each $\mathcal{W}_S$ is weakly saturated.
\item[(L2 left)] If $D=(I,F) \in \mathcal{DIA}$, and $I$ has a final object $e$, then the projection $D \rightarrow (e, F(e))$ is in $\mathcal{W}_{F(e)}$.
\item[(L3 left)] For any commutative diagram in $\mathcal{DIA}$ over $(\cdot, S)$
\[ \xymatrix{ D_1 \ar[rd] \ar[rr]^{w} & & D_2 \ar[dl]  \\
  & D_3=(E,F)     }\]
and for any chosen covers $\{U_{e,i} \rightarrow F(e)\}$ for all $e \in E$, the following implication holds true:
\[ \forall e\in E\ \forall i \quad w \times_{/D_3} (e, U_{e,i}) \in \mathcal{W}_{U_{e,i}} \quad \Rightarrow \quad w \in \mathcal{W}_{S}. \]

\item [(L4 left)] For any  morphism $w: D_1 \rightarrow D_2=(E,F)$ of pure diagram type over $(\cdot, S)$ the following implication holds true:
\[ \forall e\in E  \quad (e, F(e)) \times_{/D_2} D_1 \rightarrow (e, F(e))  \in \mathcal{W}_{(e, F(e))} \quad \Rightarrow \quad w \in \mathcal{W}_{S}. \]
\end{enumerate}
\end{DEF}

There is an obvious dual notion of a {\bf system of colocalizers} in $\mathcal{DIA} \subset \Cat^{\op}(\SSS)$ where $\SSS$ is supposed to be a strong left derivator with Grothendieck pre-cotopology. 

\begin{DEF}\label{DEFFUNLOC}
Let $\SSS$ be a strong right derivator. Assume we are given a Grothendieck pre-topology on $\SSS$ (cf.\@ \ref{DEFPRETOP}).
Let $\mathcal{DIA} \subset \Cat(\SSS)$ be a category of $\SSS$-diagrams (cf.\@ \ref{DEFSDIAGRAMCAT}). 

A subclass $\mathcal{W}$ of morphisms in $\mathcal{DIA}$ is called an {\bf absolute localizer} (or just {\bf localizer}) if
 the following properties are satisfied:
\begin{enumerate}
\item[(L1)]  $\mathcal{W}$ is weakly saturated.
\item[(L2 left)] If $D=(I,F) \in \mathcal{DIA}$, and $I$ has a final object $e$, then the projection $D \rightarrow (e, F(e))$ is in $\mathcal{W}$.
\item[(L3 left)] For any commutative diagram in $\mathcal{DIA}$
\[ \xymatrix{ D_1 \ar[rd] \ar[rr]^{\alpha} & & D_2 \ar[dl]  \\
  & D_3=(E,F)   }\]
and chosen covering $\{U_{i,e} \rightarrow F_3(e)\}$ for all $e \in E$, the following implication holds true:
\[ \forall e\in E\ \forall i \quad w \times_{/D_3} (e, U_i) \in \mathcal{W} \quad \Rightarrow \quad w \in \mathcal{W}. \]

\item [(L4 left)] For any morphism $w: D_1 \rightarrow D_2=(E,F)$  of pure diagram type, the following implication holds true:
\[ \forall e\in E  \quad (e, F(e)) \times_{/D_2} D_1 \rightarrow (e, F(e))  \in \mathcal{W} \quad \Rightarrow \quad w \in \mathcal{W}. \]
\end{enumerate}
\end{DEF}

There is an obvious dual notion of absolute {\bf colocalizer} in $\mathcal{DIA} \subset \Cat^{\op}(\SSS)$ where $\SSS$ is supposed to be a strong left derivator with Grothendieck pre-cotopology.

Recall the identification 
\begin{eqnarray*} \Cat(\SSS) &\rightarrow& \Cat^{\op}(\SSS^{\op})^{2-\op}  \\
 (I, F) &\mapsto& (I^{\op}, F^{\op}) 
 \end{eqnarray*}
By abuse of notation, we denote the image of $\mathcal{DIA}$ under this identification by $\mathcal{DIA}^{\op}$.
Note that if $\SSS$ is a strong right derivator with Grothendieck pre-topology, then $\SSS^{\op}$ is a strong left derivator with Grothendieck pre-cotopology. 

\begin{BEM}
\begin{enumerate}
\item If $\mathcal{W}$ is a localizer in $\mathcal{DIA}$, then $\mathcal{W}^{\op}$ is a colocalizer in $\mathcal{DIA}^{\op}$ and vice versa. The same holds true for
systems of relative localizers. 
\item If $\SSS$ is the trivial derivator, then a system of relative localizers or a localizer are the same notion, and (L1--L3 left) are precisely the definition of fundamental localizer of Grothendieck. 
\end{enumerate}
\end{BEM}

\begin{PROP}[Grothendieck]If $\SSS = \{\cdot\}$ is the trivial derivator, then $\Cat(\cdot) = \Cat^{\op}(\cdot)$ as 2-categories. If $\mathcal{DIA}$ is self-dual, i.e.\@ if $\mathcal{DIA}^{\op} = \mathcal{DIA}$ under this identification, then the notions of
localizer, localizer without (L4 left), colocalizer, and colocalizer without (L4 right) are all equivalent.
\end{PROP}
\begin{proof}
\cite[Proposition 1.2.6]{Cis04}
\end{proof}

\begin{BEM}\label{REMWMIN}
The class of localizers is obviously closed under intersection, hence there is a smallest localizer $\mathcal{W}^{\mathrm{min}}_{\mathcal{DIA}}$. 
Furthermore the smallest localizer in  $\mathcal{DIA}$ and the smallest colocalizer in $\mathcal{DIA}^{\op}$ correspond.
If $\SSS$ is the trivial derivator and $\mathcal{DIA} = \Cat$, Cisinski \cite[Th\'eor\`eme 2.2.11]{Cis04} has shown that $\mathcal{W}^{\mathrm{min}}_{\Cat}$ is precisely the class  $\mathcal{W}_\infty$ of functors $\alpha: I \rightarrow J$ such that $N(\alpha)$ is a weak equivalence in the classical sense (of simplicial sets, resp.\@ topological spaces).
For a localizer in the sense of Definition \ref{DEFFUNLOC} this implies the following: 
\end{BEM}

\begin{SATZ}\label{SATZWE}
If $\mathcal{DIA}= \Cat(\mathcal{\SSS})$ and $\mathcal{W}$ is an absolute localizer in $\mathcal{DIA}$ and
$\alpha \in \mathcal{W}_\infty$, i.e.\@  $\alpha: I \rightarrow J$ is a functor such that $N(\alpha)$ is a weak equivalence of topological spaces, the morphism
$(\alpha, \id): (I, p_I^* S) \rightarrow (J, p_J^* S)$ is in $\mathcal{W}$ for all $S \in \SSS(\cdot)$. The same holds analogously for a system of relative localizers. 
\end{SATZ}
\begin{proof} The class of functors $\alpha: I \rightarrow J$ in $\Cat$ such that $(\alpha, \id): (I, p_I^* S) \rightarrow (J, p_J^* S)$ is in $\mathcal{W}$ obviously form a fundamental localizer in the
classical sense. 
\end{proof}

\begin{PAR}\label{COPRODUCTSINDIA}
We will for (notational) simplicity
assume that the following properties hold:
\begin{enumerate}
\item $\SSS$ has all {\bf relative finite coproducts} (i.e.\@ for each Grothendieck opfibration with finite {\em discrete} fibers $p: O \rightarrow I$ the functor $p^*$ has a left adjoint $p_!$ and Kan's formula holds true w.r.t.\@ it). 
\item For all finite families $(S_i)_{i \in I}$ of objects in $\SSS(\cdot)$ the collection $\{ S_i \rightarrow \coprod_{j \in I} S_j \}_{i \in I}$ is a cover.
\end{enumerate}

Let $\emptyset$ be the initial object of $\mathcal{S}$ (which exists by 1.). Then the map
\[ \emptyset \rightarrow (\cdot, \emptyset), \] 
where $\emptyset$ on the left denotes the empty diagram, is in $\mathcal{W}$ (resp.\@ in $\mathcal{W}_{\emptyset}$, and hence in all $\mathcal{W}_S$) by (L3 left) applied to the empty cover.

From this and (L3 left) again it follows that for a finite collection $(S_i)_{i \in I}$ of objects of $\SSS(\cdot)$ the map
\[ (I, (S_i)_{i \in I}) \rightarrow (\cdot, \coprod_{i \in I} S_i ) \]
is in $\mathcal{W}$ (resp.\@ in $\mathcal{W}_{\coprod_{i \in I} S_i}$).  More generally, if we have a Grothendieck opfibration with finite {\em discrete} fibers $p: O \rightarrow I$ and a diagram $F \in \SSS(O)$ (over $S \in \SSS(\cdot)$),
then the morphism
\[ (O, F) \rightarrow (I, p_! F) \]
is in $\mathcal{W}$ (resp.\@ in $\mathcal{W}_{S}$). 
\end{PAR}

\begin{BEISPIEL}[Mayer-Vietoris]\label{MV1}
For the simplest non-trivial example of a non-constant map in $\mathcal{W}$ consider a cover $\{ U_1 \rightarrow S, U_2 \rightarrow S\}$ in $\SSS(\cdot)$ consisting of  two {\em mono}morphisms\footnote{For an arbitrary $\SSS$ this means that the projections $\mlq\mlq U_i \times_S U_i \mrq\mrq \rightarrow U_i$ are isomorphisms.}.
Then the projection
\[ p: \left( \vcenter{ \xymatrix{ \mlq\mlq U_1 \times_S U_2 \mrq\mrq \ar[r] \ar[d] & U_1 \\ U_2 & } } \right) \rightarrow S \]
is in $\mathcal{W}$ (resp.\@ in $\mathcal{W}_{S}$)  as is easily derived from the axioms (L1--L4). See \ref{MV2} for how the Mayer-Vietoris long exact sequence is related to this.
\end{BEISPIEL}

\begin{PAR}\label{ELEMENTARYHOMOTOPIC}
Let $\alpha, \beta: D_1 \rightarrow D_2$ be two morphisms in $\mathcal{DIA}$. 
Recall that it is the same to give a 2-morphism $\alpha \Rightarrow \beta$ or a morphism
$D_1 \times \Delta_1 \rightarrow D_2$ such that for $i=1,2$ the compositions $\xymatrix{D_1 \ar[r]^-{e_i} & D_1 \times \Delta_1 \ar[r]&  D_2 }$ are $\alpha$ and $\beta$ respectively. 
We call $\alpha$ and $\beta$ {\bf homotopic} if they are equivalent for the smallest equivalence relation containing by the following relation: $\alpha \sim \beta$, if there
exists a 2-morphism $\alpha \Rightarrow \beta$. In other words $\alpha$ and $\beta$ are homotopic if there is a finite set of 1-morphisms $\gamma_0,\gamma_1, \dots, \gamma_n: D_1 \rightarrow D_2$ such that $\gamma_0 = \alpha$ and $\gamma_n = \beta$ and a chain of 2-morphisms:
\[ \gamma_0 \Leftarrow \gamma_1 \Rightarrow \gamma_2 \Leftarrow \cdots \Rightarrow \gamma_n. \]
\end{PAR}

\begin{PROP}\label{PROPPROPERTIESLOCALIZER}
Let $\mathcal{DIA}$ be a category of $\SSS$-diagrams  (cf.\@ \ref{DEFSDIAGRAMCAT}) and let 
$\mathcal{W}$ be localizer in $\mathcal{DIA}$ (resp.\@ let $\{\mathcal{W}_S\}_S$ be a system of relative localizers).
Then $\mathcal{W}$ (resp.\@ $\{\mathcal{W}_S\}_S)$ satisfies the following properties:
\begin{enumerate}
\item The localizer $\mathcal{W}$ (resp.\@ each $\mathcal{W}_S$) is closed under coproducts. 

\item Let $\widetilde{s}=(s, \id): D_2=(I_2,s^*F) \rightarrow D_1=(I_1, F)$ be a morphism in $\mathcal{DIA}$ (resp.\@ over $(\cdot, S)$) of pure diagram type such that $s$ has a left adjoint $p: I_1 \rightarrow I_2$. 
Then the obvious morphisms
$\widetilde{p}: D_1 \rightarrow D_2$ and $\widetilde{s}$ are in $\mathcal{W}$ (resp.\@ in $\mathcal{W}_{S}$). 

\item Given a commutative diagram in $\mathcal{DIA}$ (resp.\@ one over $(\cdot, S)$)
\[ \xymatrix{ D_1 \ar[rd] \ar[rr]^{w} & & D_2 \ar[dl]  \\
  & D_3   }\]
  where the underlying functors of the morphisms to $D_3$ are Grothendieck opfibrations and the underlying functor of $w$ is a morphism of opfibrations, and
  coverings $\{U_{e,i} \rightarrow F_3(e)\}$ for all $e \in I_3$, then (in the relative case)
\[ \forall e \in I_3\ \forall i \quad w \times_{D_3} (e, U_{e,i}) \in \mathcal{W}_{U_{e,i}} \quad \Rightarrow \quad w \in \mathcal{W}_S \]
  or (in the absolute case)
\[ \forall e \in I_3\ \forall i \quad w \times_{D_3} (e, U_{e,i}) \in \mathcal{W} \quad \Rightarrow \quad w \in \mathcal{W} \]

\item If $f: D_1 \rightarrow D_2$ is in $\mathcal{W}$ (resp.\@ in $\mathcal{W}_{S}$) then also $f \times E: D_1 \times E \rightarrow D_2 \times E$ is in $\mathcal{W}$ (resp.\@ in $\mathcal{W}_{S}$) for any $E \in \Cat$ such that
the morphism $f \times E$ is a morphism in $\mathcal{DIA}$.
\item Any morphism which is homotopic (in the sense of \ref{ELEMENTARYHOMOTOPIC}) to a morphism in $\mathcal{W}$ (resp.\@ in $\mathcal{W}_{S}$) is in $\mathcal{W}$ (resp.\@ in $\mathcal{W}_{S}$).
\end{enumerate}
\end{PROP}
\begin{proof}
1.\@ This property follows immediately from (L3 left) applied to a diagram
\[ \xymatrix{
\coprod_{i \in I} D_{1,i} \ar[rr] \ar[rd] && \coprod_{i \in I} D_{2, i}  \ar[ld] \\
& (I, p_I^*S) 
} \]
where $I$ is considered to be a discrete category. (In the absolute case let $S$ be the final object of $\SSS(\cdot)$.) 

2.\@ We first show that $\widetilde{p} \in \mathcal{W}$.
Using (L3 left), it suffices to show that $\widetilde{p}_i: D_1 \times_{/I_2} i \rightarrow D_2 \times_{/I_2} i$ is in $\mathcal{W}$ (resp.\@ in $\mathcal{W}_S$) for all $i \in I_2$, however by the adjunction we have
$I_1 \times_{/I_2} i = I_1 \times_{/I_1} s(i)$ and therefore $I_1 \times_{/I_2} i$ has a final object. 
In the diagram
\[ \xymatrix{
D_1 \times_{/I_2} i \ar[r]^{\widetilde{p}_i} \ar[d] & D_2 \times_{/I_2} i \ar[d] \\
(\cdot, s(i)^*F) \ar@{=}[r] & (\cdot, s(i)^*F) 
} \]
the vertical morphisms are thus in $\mathcal{W}$ (resp.\@  $\mathcal{W}_S$) and so is the upper horizontal morphism.
That $\widetilde{s}$ is in $\mathcal{W}$ (resp.\@ in $\mathcal{W}_S$)  will follow from 4.\@ because this implies that $\widetilde{s} \circ \widetilde{p}$ and $\widetilde{p} \circ \widetilde{s}$ are in $\mathcal{W}$ (resp.\@ in $\mathcal{W}_S$)  therefore by (L1) also $\widetilde{s}$ is in $\mathcal{W}$ (resp.\@ in $\mathcal{W}_S$). For note that the unit and the counit extend to 2-morphisms of diagrams. 

3.\@ Using (L3 left), we have to show that $D_1 \times_{/D_3} (e, U_{e,i}) \rightarrow D_2 \times_{/D_3} (e, U_{e,i})$ is in $\mathcal{W}$ (resp.\@ in $\mathcal{W}_{U_{e,i}}$). 
Since the underlying functors of $D_1\rightarrow D_3$ and $D_2 \rightarrow D_3$ are Grothendieck opfibrations, we have a diagram
over $(e, U_{e,i})$:
\[ \xymatrix{
D_1 \times_{D_3} (e, U_{e,i}) \ar[r] \ar@<2pt>[d]^{\iota_e} & D_2 \times_{D_3} (e, U_{e,i})  \ar@<2pt>[d]^{\iota_e} \\
D_1 \times_{/D_3} (e, U_{e,i}) \ar[r] \ar@<2pt>[u]^{s_e} & D_2 \times_{/D_3} (e, U_{e,i})  \ar@<2pt>[u]^{s_e}
} \]
where the underlying functor of $\iota_e$ is of diagram type and is right adjoint to $s_e$. Therefore $s_e$ is in $\mathcal{W}$ (resp.\@ in $\mathcal{W}_{U_{e,i}}$) by 2.\@ and hence the same holds for $\iota_e$ because $s_e \iota_e = \id$ (using L1). Note: we are not using the not yet proven part of 2. Since the top arrow is in $\mathcal{W}$ (resp.\@ in $\mathcal{W}_{U_{e,i}}$) the same holds for the bottom arrow.

4.\@ This is a special case of 2.  

5.\@ A natural transformation $\mu: f \Rightarrow g$ for $f, g: D_1 \rightarrow D_2$ can be seen as a morphism of diagrams
$\mu: \Delta_1 \times D_1 \rightarrow D_2$ such that $\mu \circ e_0 = f$ and $\mu \circ e_1 = g$. Since the projection $p: \Delta_1 \times D_1 \rightarrow D_1$ is in $\mathcal{W}$ by 3.\@ also the morphisms
$e_{0,1}: D_1 \rightarrow \Delta_1 \times D_1$ are in $\mathcal{W}$. Since $\mu \circ e_0 = f$ and $\mu \circ e_1 = g$, the morphism $f$ is in $\mathcal{W}$ if and only if $g \in \mathcal{W}$.  
\end{proof}

\begin{PROP}\label{PROPL4}
Axiom (L4 left) is, in the presence of (L1--L3 left), equivalent to the following, apparently weaker axiom:
\begin{enumerate}
\item[(L4' left)] Let $w: D_1 \rightarrow D_2$ be a morphism (resp.\@ a morphism over $(\cdot,S)$) of pure diagram type such that the underlying functor is a Grothendieck fibration. Then (in the relative case)
\[ \forall e\in I_2  \quad (e, F_2(e)) \times_{D_2} D_1 \rightarrow (e, F_2(e))  \in \mathcal{W}_{F_2(e)} \quad \Rightarrow \quad w \in \mathcal{W}_S \]
or (in the absolute case)
\[ \forall e\in I_2  \quad (e, F_2(e)) \times_{D_2} D_1 \rightarrow (e, F_2(e))  \in \mathcal{W} \quad \Rightarrow \quad w \in \mathcal{W}. \]
\end{enumerate}
\end{PROP}
\begin{proof}(L4' left) implies (L4 left): 
Consider the following 2-commutative diagram
\[ \xymatrix{
(E, F) \times_{/(E, F)} (I, p^*F) \ar[rr] \ar[d]  \ar@{}[rrd]|{\Nearrow} && (I, p^*F) = D_1 \ar[d] \\
 (E, F) \ar@{=}[rr] && (E, F) = D_2
} \]
The underlying diagram functor of the top horizontal map (which is not purely of diagram type) is a Grothendieck opfibration and hence by Proposition~\ref{PROPPROPERTIESLOCALIZER},~3.\@ it is in $\mathcal{W}$ (resp.\@ $\mathcal{W}_S$), provided that the morphisms of the fibers $(E \times_{/E} e, \pr_1^*F) \rightarrow (\cdot, F(e))$ are in $\mathcal{W}$ (resp.\@ in $\mathcal{W}_{F(e)}$). However $E \times_{/E} e$ has the final object $\id_e$ whose value under $\pr_1^*F$ is $F(e)$. The morphisms of the fibers are therefore in $\mathcal{W}$ (resp.\@ in $\mathcal{W}_{F(e)}$) by (L2 left). The underlying diagram functor of the left vertical map is a Grothendieck fibration and  $\pr_1^*F$ is constant along the fibers. Therefore
the fact that all $(e, F(e)) \times_{/D_2} D_1 \rightarrow (e, F(e))$ are in $\mathcal{W}$ (resp.\@ in $\mathcal{W}_{F(e)}$) implies that the left vertical map is in $\mathcal{W}$ (resp.\@ in $\mathcal{W}_{S}$) by (L4' left). Thus also the right vertical map is in $\mathcal{W}$ (resp.\@ in $\mathcal{W}_{S}$).
(This uses Proposition~\ref{PROPPROPERTIESLOCALIZER},~5.\@ and the fact that the two compositions in the diagram are homotopic).

(L4 left) implies (L4' left): If $D_1 \rightarrow D_2=(E,F)$ is a morphism whose underlying functor is a Grothendieck fibration as in Axiom (L4' left), the morphism of {\em constant} diagrams $(e, F(e)) \times_{D_2} D_1 \rightarrow (e, F(e)) \times_{/D_2} D_1$ is in $\mathcal{W}$ (resp.\@ $\mathcal{W}_{F(e)}$) (their underlying functors being part of an adjunction), therefore (L4 left) applies. 
\end{proof}

\subsection{Simplicial objects in a localizer}\label{SECTSIMP}

\begin{PAR}\label{PARTENSORSTRUCTURE}
In this section, we fix a strong right derivator $\SSS$ equipped with a Grothendieck pre-topology and satisfying the assumptions of \ref{COPRODUCTSINDIA} and a category of $\SSS$-diagrams $\mathcal{DIA}$  (cf.\@ \ref{DEFSDIAGRAMCAT}).
Assume that for all $S_\bullet \in \SSS(\Delta^{\op})$ the diagrams $(\Delta^{\op}, S_\bullet)$  and also all truncations $((\Delta^{\le n})^{\op}, S_\bullet)$ are in $\mathcal{DIA}$.
Later we will assume that also $((\Delta^{\circ})^{\op}, S_\bullet)$ for all $S_\bullet \in \SSS((\Delta^{\circ})^{\op})$ (injective simplex diagram) and all truncations $((\Delta^{\circ,\le n})^{\op}, S_\bullet)$ are in $\mathcal{DIA}$.
The reasoning in this section uses little of the explicit definition of $\Delta^{\op}$. 
For comparison with classical texts on 
cohomological descent we stick to the particular diagram $\Delta^{\op}$. 

Consider the category $\SSS(\Delta^{\op})$. 
Since $\SSS$ has all (relative) finite coproducts, $\SSS(\cdot)$ is actually tensored over $\mathcal{SETF}$, hence $\SSS(\Delta^{\op})$ will be tensored over $\mathcal{SETF}^{\Delta^{\op}}$.
We sketch this construction. A finite simplicial set, i.e.\@ a functor $\xi: \Delta^{\op} \rightarrow \mathcal{SETF}$, can be seen as a functor with values in finite discrete categories. 
The corresponding Grothendieck construction yields a Grothendieck op-fibration $\pi_\xi: \int \xi \rightarrow \Delta^{\op}$. We define for $X_\bullet \in \SSS(\Delta^{\op})$:
\[ \xi \otimes X_\bullet := ( \pi_\xi)_! ( \pi_\xi)^* X_\bullet. \]
Recall that the notion `$\SSS$ has {\em relative finite coproducts}' means that all functors $( \pi_\xi)_!$ arising this way exist and can be computed fiber-wise.
\end{PAR}

\begin{PAR}
Consider the full subcategory $\Delta^{\le n}$ of $\Delta$ consisting of $\Delta_0, \dots, \Delta_{n}$. 
Since $\SSS$ is assumed to be a right derivator, the restriction functor 
\[ \iota^*: \SSS(\Delta^{\op}) \rightarrow \SSS((\Delta^{\le n})^{\op}) \]
has a right adjoint $\iota_*$, which is usually called the {\bf coskelet} and denoted $\mathrm{cosk}^n$. 

Let some simplicial object $Y_\bullet \in \SSS(\Delta^{\op})$ and
a morphism $\alpha: X_{\le n} \rightarrow \iota^* Y_\bullet$ be given. Consider the full subcategory $(\Delta^{\op} \times \Delta_1)^{0-\le n}$ 
of all objects $\Delta_i \times [1]$ for all $i \in \N_0$, and $\Delta_i \times [0]$ for $i \le n$. The restriction  
\[ \iota^*: \SSS(\Delta^{\op} \times \Delta_1 ) \rightarrow \SSS((\Delta^{\op} \times \Delta_1)^{0-\le n}) \]
has again an adjoint $\iota_*$. 
Since $\SSS$ is assumed to be strong we can consider $\alpha$ as an object over $(\Delta \times \Delta_1)^{0-\le n}$. 
The first row of $\iota_* \alpha$ is called the {\bf relative coskelet} $\mathrm{cosk}^n(X_{\le n} | Y_\bullet)$ of $X_{\le n}$. For $n=-1$ we understand
$\mathrm{cosk}^{-1}(- | Y_\bullet)=Y_\bullet$.

These constructions work the same way with $\Delta$ replaced by $\Delta^{\circ}$. The functor `coskelet' and `relative coskelet' is in both cases even
{\em the same functor}, i.e.\@ these functors commute with the restriction of a simplicial to a semi-simplicial object\footnote{To see this, e.g., for the case of the `coskelet', observe that there is an adjunction: 
\[ \xymatrix{ \Delta_m \times_{/(\Delta^{\circ})^{\op}} (\Delta^\circ_{\le n})^{\op}   \ar@<2pt>[r] & \ar@<2pt>[l] \Delta_m \times_{/\Delta^{\op}} \Delta^{\op}_{\le n}.} \]}. This would not at all be true for the corresponding left adjoint, the functor `skelet'.
\end{PAR}

We call a diagram $I$ in a diagram category $\Dia$ {\bf contractible}, if $I \rightarrow \cdot$ lies in every fundamental localizer on $\Dia$. 
\begin{LEMMA}\label{LEMMASIMPLDIA}
\begin{enumerate}
 \item
 Let $\Delta$ be the simplex category, and $I$ a category admitting a final object $i$. Let $N(I)$ be the nerve of $I$. Then the category 
 \[ \int_{(\Delta^{\circ})^{\op}} N(I)  \]
 is contractible.
 \item
 Let $\Delta$ be the simplex category, and $I$ a category admitting a final object $i$. Let $N(I)$ be the nerve of $I$. Then the category 
 \[ \int_{\Delta^{\op}} N(I)  \]
 is contractible.
\item
Let $\Delta^{\circ}$ be the injective simplex category and $I$ a {\em directed} category admitting a final object $i$. Let $N^\circ(I)$ be the semi-simplicial nerve of $I$, defined by
letting $N^\circ(I)_m$ be the set of functors $[n] \rightarrow I$ such that no non-identity morphism is mapped to an identity. Then the category 
 \[ \int_{(\Delta^{\circ})^{\op}} N^\circ(I)  \]
 is contractible. 
\end{enumerate}
\end{LEMMA}
\begin{proof}
1.\@ is shown in \cite[Proposition 2.2.3]{Cis04}.
2.\@ is the same but considering $N(I)$ as a functor from $(\Delta^\circ)^{\op}$ to $\mathcal{SET}$. The same proof works when $(\Delta^\circ)^{\op}$ is replaced by $\Delta^{\op}$.
3.\@ is also just a small modification of [loc.\@ cit.]. Define a functor $\xi: \int N^\circ(I) \rightarrow \int N^\circ(I)$ as follows:
an object $(n, x)$, where $x \in N^\circ(I)_n$ is mapped to $(n, x)$ if $x(n)=i$ and to $(n+1,x')$ with
\[ x'(k) = \begin{cases} x(k) & k \le n \\ i & k= n+1.  \end{cases}\]
otherwise.
There are natural transformations 
\[ \id_{\int N^\circ(I)} \Rightarrow \xi \qquad i \Rightarrow \xi \]
where $i$ denotes here the constant functor with value $(0, i)$, showing that $\int N^\circ(I)$ is contractible. 
\end{proof}

\begin{KOR}
The diagrams $\Delta$, $\Delta^{\circ}$, $\int_{\Delta^{\op}} \Delta_n$, $\int_{(\Delta^{\circ})^{\op}} \Delta^\circ_n$, $\int_{\Delta^{\op}} \Delta_n \times \Delta_m$ and $\Delta_m \times_{/\Delta^{\op}} (\Delta^{\circ})^{\op} = \int_{(\Delta^{\circ})^{\op}} \Delta_m$ are contractible.  
\end{KOR}
\begin{proof}
The simplicial set $\Delta_n$ is just the nerve $N$ of $[n]$. Likewise the semi-simplicial set $\Delta^{\circ}_n$ is the semi-simplicial nerve $N^\circ$ of $[n]$.
\end{proof}

Note that the diagram $\int \Delta^\circ_n$ is even {\em finite}.

\begin{LEMMA}\label{LEMMADIAG}
Let $\mathcal{W}$  be a localizer (resp.\@ let $\{\mathcal{W}_S\}_S$ be a system of relative localizers) in $\mathcal{DIA}$.

Let $((\Delta^{\op})^2, F_{\bullet, \bullet}) \in \mathcal{DIA}$ be a bisimplicial diagram (resp.\@ a bisimplicial diagram over $(\cdot, S)$) and let $\delta: \Delta^{\op} \rightarrow (\Delta^{\op})^2$ be the diagonal.
Then the morphism
\[ (\Delta^{\op}, \delta^* F_{\bullet, \bullet}) \rightarrow ((\Delta^{\op})^2, F_{\bullet, \bullet})\]
is in $\mathcal{W}$ (resp.\@ $\mathcal{W}_S$). 
\end{LEMMA}

\begin{BEM}The statement of the Lemma is false when $\Delta$ is replaced by $\Delta^\circ$.
\end{BEM}

\begin{proof}[Proof of Lemma~\ref{LEMMADIAG}.] 
We focus on the absolute case. For the relative case the proof is identical. 
Since the morphism in the statement is of pure diagram type, we may check the condition of (L4 left):
we have to show that the category
\[ (\Delta_m \times \Delta_n) \times_{/(\Delta^{\op})^2}  \Delta^{\op}    \]
is contractible, say, on the diagram category of diagrams $I$ such that $(I, F_{m,n}) \in \mathcal{DIA}$.
Equivalently we may prove this for the dual category. 
Objects of that category are diagrams of the form:
\[ \xymatrix{
\Delta_{m'} \ar[rd] \ar[d]  &  \\
\Delta_m& \Delta_n 
} \]
This is the category $\Delta / (\Delta_m \times \Delta_n)$ which is contractible by Lemma~\ref{LEMMASIMPLDIA},~2. 
Note that this
is the only feature of $\Delta^{\op}$ used in the proof of this Lemma. 
\end{proof}

\begin{BEM}
The previous lemma should be seen in the following context: the Grothendieck construction 
gives a way of embedding the category of simplicial sets into the category of small categories. This construction maps weak equivalences to weak
equivalences and induces an equivalence between the corresponding homotopy categories. A bisimplicial set can be seen as a simplicial object
in the category of simplicial sets. Its homotopy colimit is given by the diagonal simplicial set. On the other hand the homotopy colimit
in the category of small categories is just given by the Grothendieck construction. From this perspective, the lemma is clear if
$\SSS$ is the derivator associated with the category of sets (equipped with the discrete topology). 
\end{BEM}

\begin{LEMMA}\label{LEMMASIMPLEX}
Let $\mathcal{W}$ be a localizer (resp.\@ let $\{\mathcal{W}_S\}_S$ be a system of relative localizers) in $\mathcal{DIA}$.

Consider a simplicial diagram $(\Delta^{\op}, F_\bullet) \in \mathcal{DIA}$ (resp.\@ a simplicial diagram over $(\cdot, S)$). The morphism
\[ (\Delta^{\op}, F_\bullet \otimes \Delta_n) \rightarrow (\Delta^{\op}, F_\bullet) \]
is in $\mathcal{W}$ (resp.\@ $\mathcal{W}_S$). 
\end{LEMMA}

\begin{proof}
We focus on the absolute case. For the relative case the proof is identical. 
The diagram $(\Delta^{\op}, F_\bullet \otimes \Delta_n)$ is equivalent to $(\int \Delta_n, \pi^* F_\bullet)$ by definition~(see \ref{PARTENSORSTRUCTURE}). 
We apply the criterion of (L4 left) to the resulting map 
\[ (\int \Delta_n, \pi^* F_\bullet) \rightarrow (\Delta^{\op}, F_\bullet)  \]
and have to show that
\[ \Delta_m \times_{/\Delta^{\op}} \int \Delta_n  \]
is contractible. 
This category is again (dual to) the category of objects 
\[ \xymatrix{
\Delta_{k} \ar[rd] \ar[d]  &  \\
\Delta_m& \Delta_n 
} \]
and we have already seen in the proof of Lemma~\ref{LEMMADIAG} that it is contractible. 
\end{proof}

\begin{KOR}\label{LEMMACOSKHOMOTOPY}
Let $\mathcal{W}$ be a localizer (resp.\@ let $\{\mathcal{W}_S\}_S$ be a system of relative localizers) in $\mathcal{DIA}$.

Let $f,g: (\Delta^{\op}, F_\bullet) \rightarrow  (\Delta^{\op}, G_\bullet)$ be two homotopic morphisms of simplicial objects (resp.\@ morphisms over $(\cdot, S)$). Then $f \in \mathcal{W}$ (resp.\@ in $\mathcal{W}_S$) if and only if $g \in \mathcal{W}$ (resp.\@ in $\mathcal{W}_S$).
\end{KOR}
\begin{proof}
The statement follows by the standard argument because the projection $(\Delta^{\op}, F_\bullet \otimes \Delta_1) \rightarrow (\Delta^{\op}, F_\bullet)$ is in $ \mathcal{W}$  (resp.\@ in $\mathcal{W}_S$) by Lemma~\ref{LEMMASIMPLEX}. 
\end{proof}

\begin{PROP}[\v{C}ech resolutions are in $\mathcal{W}$]\label{LEMMACECH}
Let $\mathcal{W}$ be a localizer (resp.\@ let $\{\mathcal{W}_S\}_S$ be a system of relative localizers) in $\mathcal{DIA}$.

Let $U \rightarrow S$ be a local epimorphism in $\SSS(\cdot)$. Then the morphism
\[  p: (\Delta^{\op}, \cosk^0(U|S)) \rightarrow (\cdot, S) \]
is in $\mathcal{W}$ (resp.\@ $\mathcal{W}_S$). 
\end{PROP}
\begin{proof}
To simplify the exposition we focus on the case in which $\SSS$ is associated with a category $\mathcal{S}$. The reader may check however that
everything goes through in the general case because the only constructions involved can be expressed as right Kan extensions. 
The assumption means that there is a cover $\mathcal{U} = \{U_i \rightarrow S\}$ in the given pre-topology, such that for all indices $i$, the induced map
\[ p_i: U \times_S U_i \rightarrow U_i \]
has a section $s_i$. By axiom (L3 left) it suffices to show that for all $i$ the map
\[  \widetilde{p}_i: (\Delta^{\op}, \cosk^0(U\times_S U_i\,|\,U_i)) \rightarrow (\cdot, U_i) \]
is in $\mathcal{W}$ (resp.\@ in $\mathcal{W}_{U_i}$). Explicitly the simplicial object $\cosk^0(U\times_S U_i\,|\,U_i)$  is given by
\[ \xymatrix{ \cdots \ar@<-1.5ex>[r]\ar@<-0.5ex>[r]\ar@<0.5ex>[r]\ar@<1.5ex>[r]  &   U \times_S U  \times_S U  \times_S U_i \ar@<1ex>[r]\ar@<0ex>[r]\ar@<-1ex>[r] & U \times_S U  \times_S U_i \ar@<0.5ex>[r] \ar@<-0.5ex>[r] & U \times_S U_i  } \]

Since $\Delta^{\op}$ is contractible (in particular the morphism $(\Delta^{\op}, p^*T) \rightarrow (\cdot, T)$ is in $\mathcal{W}$, resp.\@ in $\mathcal{W}_T$, for any $T \in \SSS(\cdot)$), it suffices to show that the map
\[  \widetilde{p}_i: (\Delta^{\op}, \cosk^0(U \times_S U_i\,|\,U_i)) \rightarrow (\Delta^{\op}, p^* U_i) \]
is in $\mathcal{W}$ (resp.\@ in $\mathcal{W}_{U_i}$). There is a section 
\[  \widetilde{s}_i: (\Delta^{\op}, p^* U_i) \rightarrow (\Delta^{\op}, \cosk^0(U \times_S U_i\,|\,U_i)) \]
induced by $s_i$ such that $\widetilde{p}_i \circ \widetilde{s}_i = \id$. By (L1) it then suffices to check that $\widetilde{s}_i \circ \widetilde{p}_i \in \mathcal{W}$  (resp.\@ in $\mathcal{W}_{U_i}$). We will construct
a homotopy between $\id$ and $\widetilde{s}_i \circ \widetilde{p}_i$
\[(\Delta^{\op}, \Delta_1 \times \cosk^0(U \times_S U_i\,|\,U_i)) \rightarrow (\Delta^{\op}, \cosk^0(U\times_S U_i\,|\,U_i))\]
 in the sense of simplicial objects. This will suffice by Corollary~\ref{LEMMACOSKHOMOTOPY}.
Since $\Delta_1$ is 1-coskeletal, and the $\cosk^0$'s anyway, it suffices to construct the homotopy in degrees 0 and 1:
\[ \xymatrix{
     \Hom(\Delta_1, \Delta_1) \times  U \times_S U \times_S U_i \ar@<0.5ex>[r] \ar@<-0.5ex>[r]  \ar[d] & \Hom(\Delta_0, \Delta_1) \times U \times_S U_i \ar[d] \\
     U \times_S U \times_S U_i \ar@<0.5ex>[r] \ar@<-0.5ex>[r] & U \times_S U_i 
  } \]
This can be achieved by mapping $\id_{\Delta_1} \times  (U \times_S U_i) \times_{U_i} (U  \times_S U_i)$ to $(U\times_S U_i) \times_{U_i} (U \times_S U_i)$ via $(s_i \circ \pr_2) \times \id$. 
\end{proof}

\begin{DEF}\label{DEFHYPERCOVER}
A morphism $X_\bullet \rightarrow Y_\bullet$ of simplicial objects 
is called a {\bf hypercover} if the following two equivalent conditions hold:
\begin{enumerate}
\item In any diagram of simplicial objects
\[ \xymatrix{
\partial \Delta_n \otimes U \ar[r] \ar[d] & X_\bullet \ar[d] \\
\Delta_n \otimes U \ar[r] & Y_\bullet
} \]
there is a cover $\mathcal{U} = \{U_i \rightarrow U\}$ such that for all $i$ there is a lift (indicated by a dotted arrow) in the diagram
\[ \xymatrix{
\partial \Delta_n \otimes U_i \ar[r] \ar[d] & \partial \Delta_n \otimes U \ar[r]&  X_\bullet \ar[d] \\
\Delta_n \otimes U_i \ar@{.>}[rru] \ar[r] & \Delta_n \otimes U \ar[r]  & Y_\bullet
} \]
\item For any $n \ge 0$ the morphism 
\[ X_{n} \rightarrow \cosk^{n-1}(\iota_{\le n-1}^* X_\bullet\,|\,Y_\bullet)_{n}  \]
admits local sections in the pre-topology on $\SSS$ (i.e.\@ it is a local epimorphism).  
\end{enumerate}
\end{DEF}

\begin{BEM}
\begin{enumerate}
\item In particular the notion of hypercover depends only on the Grothendieck topology generated by the pre-topology because 
a morphism is a local epimorphism precisely if the sieve generated by it is a covering sieve. 

\item The equivalent condition 1.\@ of the definition of hypercover shows that, if $\SSS$ is the derivator associated with the category $\mathcal{SET}$ equipped with the discrete topology, then a hypercover is precisely a trivial Kan fibration.
\end{enumerate}
\end{BEM}

\begin{DEF}If in condition 2.\@ of Definition~\ref{DEFHYPERCOVER} the morphism is even an isomorphism for all sufficiently large $n$, then $\alpha$ is called a {\bf finite (or bounded) hypercover}. 
Equivalently we have $X_\bullet \cong \cosk^{n}(\iota_{\le n}^* X_\bullet\,|\,Y_\bullet)$ for some $n$.
\end{DEF}

\begin{LEMMA}\label{LEMMAELEMENTARYHYPERCOVER}
Let $\mathcal{W}$ be a localizer (resp.\@ $\{\mathcal{W}_S\}_S$ be a system of relative localizers) in $\mathcal{DIA}$.

For a finite hypercover $X_\bullet \rightarrow Y_\bullet$ (resp.\@ one over $(\cdot, S)$) such that 
$X_\bullet \cong \cosk^{i+1}(X_\bullet\,|\,Y_\bullet)$ and 
$\iota_{\le i-1}^* {X}_\bullet \cong \iota_{\le i-1}^* {Y}_\bullet$ the morphism $(\Delta^{\op}, X_\bullet) \rightarrow (\Delta^{\op}, Y_\bullet)$ is in $\mathcal{W}$ (resp.\@ in $\mathcal{W}_{S}$).
\end{LEMMA}
\begin{proof}
Again, to simplify the exposition we focus on the case in which $\SSS$ is associated with a category $\mathcal{S}$. 
We may assume $i \ge 1$ because otherwise we are in the situation of Lemma \ref{LEMMACECH}.
The assumptions imply that the map ${X}_{i} \rightarrow {Y}_{i}$ is a local epimorphism. Indeed, this is the map
$X_{i} \rightarrow Y_i = \cosk^{i-1}(\iota_{\le i-1}^* X_\bullet\,|\,Y_\bullet)_{i}$ in this case. Therefore the morphism 
${X}_{j} \rightarrow {Y}_{j}$ is actually a local epimorphism for all $j$. 

Consider the following diagram in $\mathcal{DIA}$:
\[ \xymatrix{
(\Delta^{\op} \times \Delta^{\op}, (X_\bullet \times_{Y_\bullet} X_\bullet\,|\,X_\bullet)) \ar[r] \ar[d] & (\Delta^{\op} \times \Delta^{\op}, (X_\bullet\,|\,Y_\bullet)) \ar[d] \\
(\Delta^{\op}, X_\bullet) \ar[r] & (\Delta^{\op}, Y_\bullet)
} \]
where 
\[ (X_\bullet\,|\,Y_\bullet)_{m,n} := \cosk^0(X_n\,|\,Y_n)_m = \underbrace{X_n \times_{Y_n} \cdots \times_{Y_n}X_n}_{m+1 \text{ factors}}. \]
\[ (X_\bullet \times_{Y_\bullet} X_\bullet\,|\,X_\bullet)_{m,n} := \cosk^0(X_n\times_{Y_n} X_n\,|\,X_n)_m = \underbrace{X_n \times_{Y_n} \cdots \times_{Y_n}X_n}_{m+2 \text{ factors}}. \]
The vertical morphisms are in $\mathcal{W}$ by Proposition~\ref{PROPPROPERTIESLOCALIZER},~3.\@ because
its  columns are in  $\mathcal{W}$ by Lemma~\ref{LEMMACECH}. 
Again by Proposition~$\ref{PROPPROPERTIESLOCALIZER}$,~3.\@ it then suffices to show that the rows  
\[ \xymatrix{
p: (\Delta^{\op}, (X_\bullet \times Y_\bullet\,|\,X_\bullet)_{m, \bullet}) \ar[r] & (\Delta^{\op}, (X_\bullet\,|\,Y_\bullet)_{m, \bullet}) 
} \]
of the top horizontal morphism are in $\mathcal{W}$. These are again hypercovers of the form considered in this Lemma, in particular $i$-coskeletal, where the $i$-truncation is given by
\[ \xymatrix{
\underbrace{X_i\times_{Y_i} \cdots \times_{Y_i} X_i}_{m+2} \ar[d] \ar@<1ex>[r] \ar@<-1ex>[r]^-{\vdots} & X_{i-1}=Y_{i-1} \cdots  \ar@{=}[d] & \cdots \ar@<0.5ex>[r] \ar@<-0.5ex>[r] & X_0=Y_0 \ar@{=}[d] \\
\underbrace{ X_i\times_{Y_i} \cdots \times_{Y_i} X_i}_{m+1} \ar@<1ex>[r] \ar@<-1ex>[r]^-{\vdots} & Y_{i-1} \cdots & \cdots \ar@<0.5ex>[r] \ar@<-0.5ex>[r] & Y_0
} \]
where the left-most vertical arrow is induced by the map $\Delta_{m+1} \rightarrow \Delta_{m+2}$, $i \mapsto i$. There is a section $s$, with $s_i$ induced by the map 
\[ \Delta_{m+2} \rightarrow \Delta_{m+1}, \quad  i \mapsto \begin{cases} i & i<m+2, \\ m+1 & i=m+2. \end{cases} \] 
We will construct a homotopy $\mu: \id \Rightarrow s \circ p$ of truncated simplicial objects:
\[ \xymatrix{
\Hom(\Delta_i, \Delta_1) \times \underbrace{X_i\times_{Y_i} \cdots \times_{Y_i} X_i}_{m+2} \ar[d]^{\mu_i}  \ar@<1ex>[r] \ar@<-1ex>[r]^-{\vdots} & \Hom(\Delta_{i-1}, \Delta_1) \times Y_{i-1} \cdots  \ar[d]^{\mu_{i-1}} & \cdots \ar@<0.5ex>[r] \ar@<-0.5ex>[r] & \Hom(\Delta_0, \Delta_1) \times Y_0 \ar[d]^{\mu_0} \\
\underbrace{X_i\times_{Y_i} \cdots \times_{Y_i} X_i}_{m+2}  \ar@<1ex>[r] \ar@<-1ex>[r]^-{\vdots} & Y_{i-1} \cdots   & \cdots \ar@<0.5ex>[r] \ar@<-0.5ex>[r] & Y_0  \\
} \]
The morphism $\mu_i$ at the constant morphism $0: \Delta_i \rightarrow \Delta_1$ is given by the identity, at the 
constant morphism $1: \Delta_i \rightarrow \Delta_1$ given by $s_i \circ p_i$, and at the other morphisms $\Delta_i \rightarrow \Delta_1$ arbitrarily. 
The existence of this homotopy allows by Lemma~\ref{LEMMACOSKHOMOTOPY} and by (L1) to conclude.
\end{proof}

\begin{SATZ}\label{SATZHYPERCOVER}
Let $\mathcal{W}$ be a localizer (resp.\@ $\{\mathcal{W}_S\}_S$ be a system of relative localizers) in $\mathcal{DIA}$.

Any finite hypercover (resp.\@ one over $S$) considered as a morphism of diagrams  in $\mathcal{DIA}$
\begin{equation}\label{eqmordia}
 (\Delta^{\op}, X_\bullet) \rightarrow (\Delta^{\op}, Y_\bullet) 
\end{equation}
is in $\mathcal{W}$ (resp.\@ in $\mathcal{W}_S$). 

Let $\iota: (\Delta^{\circ})^{\op} \rightarrow \Delta^{\op}$ be the inclusion. If the morphism (\ref{eqmordia}) exists in $\mathcal{DIA}$ then also
\[ ((\Delta^{\circ})^{\op}, \iota^* X_\bullet) \rightarrow ((\Delta^\circ)^{\op}, \iota^* Y_\bullet) \]
is in $\mathcal{W}$ (resp.\@ in $\mathcal{W}_S$).
\end{SATZ}
\begin{proof}
Any finite hypercover is a finite succession of hypercovers of the form considered in Lemma~\ref{LEMMAELEMENTARYHYPERCOVER}.
The additional statement is a consequence of the following Lemma. 
\end{proof}

\begin{LEMMA}
Let $\mathcal{W}$ be a localizer (resp.\@ let $\{\mathcal{W}_S\}_S$ be a system of relative localizers) in $\mathcal{DIA}$.

Let $\iota: (\Delta^{\circ})^{\op} \rightarrow \Delta^{\op}$ be the inclusion and
let $(\Delta^{\op}, X_\bullet)$ be a simplicial diagram in $\mathcal{DIA}$ (resp.\@ a simplicial diagram over $(\cdot, S)$). Then the morphism
\[ ((\Delta^{\circ})^{\op}, \iota^* X_\bullet) \rightarrow (\Delta^{\op}, X_\bullet) \]
(if in $\mathcal{DIA}$) is in $\mathcal{W}$ (resp.\@ in $\mathcal{W}_S$).  
\end{LEMMA}
\begin{proof}
We focus on the absolute case. For the relative case the proof is identical. 
Since the morphism in the statement is of pure diagram type, we may check the condition of (L4 left):
we have to show that the category
\[ \Delta_m \times_{/\Delta^{\op}}  (\Delta^{\circ})^{\op}    \]
is contractible, say, on the diagram category of diagrams $I$ such that $(I, X_{m}) \in \mathcal{DIA}$. This is true by Lemma~\ref{LEMMASIMPLDIA},~1.
\end{proof}

\subsection{Cartesian and coCartesian objects}\label{SECTCART}

\begin{DEF} Let $\DD \rightarrow \SSS$ be a fibered derivator of domain $\Dia$. 
Let $I, E \in \Dia$ be diagrams and let $\alpha: I \rightarrow E$ be a functor in $\Dia$. 
We say that an object \[ X \in \DD(I)\] is $E$-{\bf (co-)Cartesian}, if for any morphism
$\mu: i \rightarrow j$ in $I$ mapping to an identity in $E$, the corresponding morphism
$\DD(\mu): i^*X \rightarrow j^*X$ is (co-)Cartesian. 

If $E$ is the trivial category, we omit it from the notation, and talk about (co-)Cartesian objects. 
\end{DEF}

These notions define full subcategories $\DD(I)^{E-\mathrm{cart}}$ (resp.\@ $\DD(I)^{E-\mathrm{cocart}}$) of $\DD(I)$, and $\DD(I)_F^{E-\mathrm{cart}}$ (resp.\@  $\DD(I)_F^{E-\mathrm{cocart}}$) of $\DD(I)_F$ for any $F \in \SSS(I)$.

\begin{LEMMA}
The functor $\alpha^*$ w.r.t. a morphism $\alpha: D_1 \rightarrow D_2$ in $\Dia(\SSS)$ maps Cartesian objects to Cartesian objects. 
The functor $\alpha^*$ for a morphism $\alpha: D_1 \rightarrow D_2$ in $\Dia^{\op}(\SSS)$ maps coCartesian objects to coCartesian objects. 
\end{LEMMA}

\begin{BEM}\label{BL}
The categories of coCartesian objects are a generalization of the {\bf equivariant derived categories} of Bernstein and Lunts \cite{BL94}. For this let $\DD \rightarrow \SSS^{\op}$ be the stable fibered derivator
of sheaves of abelian groups on (nice) topological spaces, where $\SSS$ is the pre-derivator associated with the category of (nice) topological spaces. 
Let $G$ be a topological group acting on a space $X$. Then we may form the following simplicial space which is an object of $\SSS(\Delta^{\op})$:
\[ [G \backslash X]_\bullet: \xymatrix{ \cdots \ar@<-1.5ex>[r]\ar@<-0.5ex>[r]\ar@<0.5ex>[r]\ar@<1.5ex>[r]  &   G \times G \times X \ar@<1ex>[r]\ar@<0ex>[r]\ar@<-1ex>[r] & G \times X \ar@<0.5ex>[r] \ar@<-0.5ex>[r] & X,  } \]
cf.\@ \cite[B1]{BL94}.
Then the category 
\[ \DD(\Delta)_{[G \backslash X]_\bullet}^{\mathrm{cocart}} \]
is equivalent to the (unbounded) equivariant derived category, cf.\@ \cite[Proposition B4]{BL94}. Note that all pull-back functors are exact in this context.
\end{BEM}

\begin{DEF}\label{DEFCARTPROJ}
Let $\DD \rightarrow \SSS$ be a fibered derivator of domain $\Dia$.
We say that $\DD \rightarrow \SSS$ {\bf admits left Cartesian projections} if for all functors $\alpha: I \rightarrow E$ in $\Dia$ and $S \in \SSS(\cdot)$, the fully-faithful inclusion 
\[ \DD(I)^{E-\mathrm{cart}}_F \rightarrow \DD(I)_F \]
has a left adjoint $\Box_!^E$.  
 More generally we have four notions with the following notations:
\[ \begin{array}{rll}
 \Box_!^E & \text{left adjoint} & \text{left Cartesian projection} \\
 \blacksquare_!^E & \text{right adjoint} & \text{right Cartesian projection} \\
 \blacksquare_*^E & \text{left adjoint} & \text{left coCartesian projection} \\
 \Box_*^E & \text{right adjoint} & \text{right coCartesian projection}
\end{array} \]
\end{DEF}
We will, in general, only use left Cartesian and right coCartesian projection, the others being somewhat unnatural.
In~\ref{THEOREMEXISTRIGHTPROJ} we will show (using Brown representability) that for an infinite fibered derivator whose fibers are stable and well-generated a right coCartesian projection exists. 
Similarly if, in addition, Brown representability for the dual holds, e.g.\@ if the fibers are compactly generated, then a left Cartesian projection exists (see \ref{THEOREMEXISTLEFTPROJ}) in many cases. 
Note that for a usual (non fibered) derivator, the notions `Cartesian' and `coCartesian' are equivalent. 
If for a fibered derivator with stable fibers both left and right Cartesian projections exist, then there is actually a {\em recollement} \cite[Proposition 4.13.1]{Kra09}:
\[ \xymatrix{
\DD(I)_F^{E-\mathrm{cart}} \ar[rr]^{\text{incl.}} && \ar@<3ex>[ll]_{\Box_!} \ar@<-3ex>[ll]_{\blacksquare_*} \DD(I)_F  \ar[rr] && \ar@<3ex>[ll] \ar@<-3ex>[ll] \DD(I)_F/\DD(I)_F^{E-\mathrm{cart}}
} \] 

\begin{BEISPIEL}
The projections are difficult to describe explicitly, except in very special situations. Here a rather trivial example where this is possible. 
Let $\DD$ be a stable derivator and consider $I = \Delta_1$, the projection $p: \Delta_1 \rightarrow \cdot$ and the inclusions $e_0, e_1: \cdot \rightarrow \Delta_1$. 
Then a left and a right Cartesian projection exist and the recollement above is explicitly given by:
\[ \xymatrix{
 \DD(\Delta_1)^{\mathrm{cart}} \cong \DD(\cdot) \ar[rr]^-{p^*} && \ar@<3ex>[ll]_-{e_1^*} \ar@<-3ex>[ll]_-{e_0^*} \DD(\Delta_1)  \ar[rr]^{\mathrm{C}} && \ar@<3ex>[ll]_{[-1] \circ e_{0,*}} \ar@<-3ex>[ll]_{e_{1!}} \DD(\cdot)
} \] 
Note that the functor $\mathrm{C}$ (Cone) may be described as either $[1] \circ e_{0}^! $ or $e_{1}^? $ (cf.\@ \cite[\S 3]{Gro13}) and that the essential image of $p^*$ is precisely the kernel of $\mathrm{C}$, which also coincides with the full subcategory of Cartesian=coCartesian objects. 
\end{BEISPIEL}

\subsection{Weak and strong $\DD$-equivalences}\label{SECTEQUIV}

\begin{DEF}[left]
Let $\Dia$ be a diagram category and let $\SSS$ be a strong right derivator with domain $\Dia$ equipped with a Grothendieck pre-topology. 
Let $\DD \rightarrow \SSS$ be a left fibered derivator satisfying (FDer0 right) and let $S \in \SSS(\cdot)$.
A morphism  $f: D_1 \rightarrow D_2$ in $\Dia(\SSS) / (\cdot, S)$ is called a {\bf weak $\DD$-equivalence} relative to $S$ if the natural transformation
\[  p_{1!} p_{1}^* \rightarrow p_{2!} p_{2}^*  \]
is an isomorphism of functors.
A morphism $f \in \Dia(\SSS)$ is called a  {\bf strong $\DD$-equivalence} if the functor $f^*$ induces an equivalence of categories
\[ f^*: \DD(D_2)^{\mathrm{cart}} \rightarrow \DD(D_1)^{\mathrm{cart}}. \]
\end{DEF}

Note that {\em weak} is a relative notion whereas {\em strong} is absolute.

\begin{DEF}[right]
Let $\Dia$ be a diagram category, and let $\SSS$ be a strong left derivator with domain $\Dia$ equipped with a Grothendieck pre-cotopology. 
Let $\DD \rightarrow \SSS$ be a right fibered derivator satisfying (FDer0 left) and $S \in \SSS(\cdot)$.
A morphism $f: D_1 \rightarrow D_2$ in $\Dia^{\op}(\SSS) / (\cdot, S)$ is called a {\bf weak $\DD$-equivalence} relative to $S$, if the natural transformation
\[  p_{2*} p_{2}^* \rightarrow p_{1*} p_{1}^*  \]
is an isomorphism of functors. 
A morphism $f \in \Dia^{\op}(\SSS)$ is called a  {\bf strong $\DD$-equivalence} if the functor $f^*$ induces an equivalence of categories
\[ f^*: \DD(D_2)^{\mathrm{cocart}} \rightarrow \DD(D_1)^{\mathrm{cocart}}. \]
\end{DEF}

For a (left and right) derivator, i.e.\@ for $\SSS = \cdot$, there is no difference between $\Dia(\SSS)$ and $\Dia^{\op}(\SSS)$ and then also the two different definitions of 
weak, resp.\@ strong $\DD$-equivalence coincide (for the case of weak $\DD$-equivalences, note that the two conditions become adjoint to each other). 
These notions of {\em $\DD$-equivalence} (right version) should be compared to the classical notions of {\em cohomological descent}, see \cite[Expos\'e V$^{\text{bis}}$]{SGAIV2}. 

\begin{LEMMA}[left] Let $f: D_1 \rightarrow D_2$ be a morphism in $\Dia(\SSS) / (\cdot, S)$. Then the following implication holds:
\[ \text{ $f$ strong $\DD$-equivalence}  \quad \Rightarrow \quad \text{$f$ weak $\DD$-equivalence relative to $S$. }\]
\end{LEMMA}
\begin{proof}
The morphism in the definition of weak $\DD$-equivalence is induced by the counit w.r.t.\@ the adjunction $f^*, f_!$:
\[ p_{1!} p_{1}^* \cong p_{2!} f_! f^* p_{2}^* \rightarrow  p_{2!} p_{2}^* \]
Now let $f_{!\Box}$ be an inverse to $f^*$, as required by the definition of strong $\DD$-equivalence.  
From $f^* p_2^* \cong p_1^*$ follows $p_{2!} f_{!\Box} \cong p_{1!}$ and moreover the diagram 
\[ \xymatrix{
 p_{2!} f_! f^* p_{2}^* \ar[r] \ar[d]^\sim &p_{2!} p_{2}^* \\
 p_{2!} f_{!\Box} f^* p_{2}^* \ar[ru] & 
} \]
is commutative. Since the diagonal morphism is a natural isomorphism the statement follows. 
\end{proof}

Of course there is an analogous right version of this lemma.
The goal of this section is to prove the following two theorems: 

\begin{HAUPTSATZ}[right]\label{MAINTHEOREMCOHOMDESCENT}
Let $\Dia$ be a diagram category and let $\SSS$ be a strong left derivator with domain $\Dia$ equipped with a Grothendieck pre-cotopology. 
\begin{enumerate}
\item Let $\DD \rightarrow \SSS$ be a fibered derivator with domain $\Dia$ which is colocal in the sense of Definition~\ref{DEFFIBEREDCOLOCAL} for the Grothendieck pre-cotopology on $\SSS$. Then the collection of classes $\{\mathcal{W}_{\DD, S}\}_S$, where $\mathcal{W}_{\DD, S}$ consists of those morphisms $f: D_1 \rightarrow D_2$ in $\Dia^{\op}(\SSS)$ which are {\em weak $\DD$-equivalences} relative to $S \in \SSS(\cdot)$, forms a system of relative colocalizers.

\item Let $\DD \rightarrow \SSS$ be an infinite fibered derivator with domain $\Dia$ which is colocal in the sense of Definition~\ref{DEFFIBEREDCOLOCAL} for the Grothendieck pre-cotopology on $\SSS$, with stable, compactly generated fibers.
The class $\mathcal{W}_\DD$ consisting of those morphisms $f: D_1 \rightarrow D_2$ in $\Dia^{\op}(\SSS)$ which are {\em strong $\DD$-equivalences}
forms an absolute colocalizer.
\end{enumerate}
\end{HAUPTSATZ}

\begin{HAUPTSATZ}[left]\label{MAINTHEOREMHOMDESCENT}
Let $\Dia$ be a diagram category and let $\SSS$ be a strong right derivator with domain $\Dia$ equipped with a Grothendieck pre-topology. 
\begin{enumerate}
\item Let $\DD \rightarrow \SSS$ be a fibered derivator with domain $\Dia$, which is local in the sense of Definition~\ref{DEFFIBEREDLOCAL} for the Grothendieck pre-topology on $\SSS$. 
Then the collection of classes $\{\mathcal{W}_{\DD,S}\}_S$, where $\mathcal{W}_{\DD,S}$ consists of those morphisms $f: D_1 \rightarrow D_2$ in $\Dia(\SSS)$ which are {\em weak $\DD$-equivalences} relative to $S \in \SSS(\cdot)$, forms a system of relative localizers.

\item Let $\Dia'(\SSS) \subset \Dia(\SSS)$ be the full subcategory of the diagrams which consist of universally $\DD$-local morphisms.

Let $\DD \rightarrow \SSS$ be an infinite fibered derivator with domain $\Dia$, which is local in the sense of Definition~\ref{DEFFIBEREDLOCAL} for the Grothendieck pre-topology on $\SSS$, with stable, compactly generated fibers. The class $\mathcal{W}_\DD$ consisting of those morphisms $f: D_1 \rightarrow D_2$ in $\Dia'(\SSS)$ which are {\em strong $\DD$-equivalences}
forms an absolute localizer in $\Dia'(\SSS)$.
\end{enumerate}
\end{HAUPTSATZ}

\begin{BEM}
The restriction onto $\Dia'(\SSS)$ in the left-variant of the theorem is needed because otherwise we do not know whether left Cartesian projections exist (cf.\@ Theorem \ref{THEOREMEXISTLEFTPROJ}).
\end{BEM}

The weak $\DD$-equivalences for the case of usual derivators (i.e.\@ for $\SSS=\{ \cdot\}$) were called just `$\DD$-equivalences' by Cisinski \cite{Cis08} and it is
rather straight-forward to see from the definition of derivator that they from a fundamental localizer in the classical sense (= absolute localizer for $\SSS=\{\cdot\}$, =\nobreakspace{}system of relative localizers for $\SSS=\{\cdot\}$).

We will only prove the left-variant of the theorem. The other follows by logical duality and the
restriction to $\Dia'(\SSS)$ is not necessary because Lemma~\ref{LEMMAPROJECTIONRIGHT} is used
instead of Lemma~\ref{LEMMAPROJECTIONLEFT}. Before proving the theorem we need a couple of lemmas.
We assume for the rest of this section that $\Dia$ is a diagram category and that $\SSS$ is a strong right derivator with domain $\Dia$ equipped with a Grothendieck pre-topology.

\begin{DEF}\label{DEFHOMOTOPYEQUIV}
Two morphisms (in $\Dia(\SSS)$ or in $\Dia^{\op}(\SSS)$) 
\[ \xymatrix{
 D_1 \ar@<0.5ex>[r]^p & D_2 \ar@<0.5ex>[l]^s }
\]
such that chains of 2-morphisms
\[  p \circ s \Rightarrow \cdots \Leftarrow \cdots \Rightarrow  \id_{D_1} \qquad  s \circ p \Rightarrow \cdots \Leftarrow \cdots \Rightarrow \id_{D_2}   \]
exist are called a {\bf homotopy equivalence} (or $p$ is called as such if an $s$ with this property exists).
\end{DEF}

\begin{LEMMA}[left]\label{LEMMAHOMOTOPY}
Let $\DD$ be a left fibered derivator satisfying (FDer0 right) and let $D_1, D_2 \in \Dia(\SSS)$. 
Given any homotopy equivalence $(p, s)$, then the functors $p^*$ and $s^*$ induce an equivalence
\[ \xymatrix{
 \DD(D_2)^{\mathrm{cart}} \ar@<0.5ex>[r]^{p^*} & \DD(D_1)^{\mathrm{cart}} \ar@<0.5ex>[l]^{s^*} }
\]
\end{LEMMA}
\begin{proof}
The 2-morphisms $\mu: (\alpha, f) \Rightarrow (\beta, g)$ in Definition~\ref{DEFHOMOTOPYEQUIV} induce morphisms between the pull-back functors
\[  (\alpha, f)^*\mathcal{E} \rightarrow (\beta, g)^* \mathcal{E} \]
which are isomorphisms on Cartesian objects. 
\end{proof}

\begin{BEISPIEL}[cf.\@ also Proposition~\ref{PROPPROPERTIESLOCALIZER}, 2.\@]
Let $I_1, I_2$ be diagrams in $\Dia$. If
\[ \xymatrix{
 I_1 \ar@<0.5ex>[r]^p & I_2 \ar@<0.5ex>[l]^s }
\]
is an adjunction where $p$ is left adjoint to $s$, and if $F \in \SSS(I_1)$ then we get an equivalence
\[ \xymatrix{
 \DD(D_2)^{\mathrm{cart}} \ar@<0.5ex>[r]^{p^*} & \DD(D_1)^{\mathrm{cart}} \ar@<0.5ex>[l]^{s^*} }
\]
where $D_1=(I_1, F)$ and $D_2=(I_2,F \circ s)$.
\end{BEISPIEL}

\begin{LEMMA}[left]\label{LEMMAPROJECTIONLEFT}
Let $\Dia$ be an infinite diagram category and
let $\DD \rightarrow \SSS$ be an infinite fibered derivator with domain $\Dia$ with stable, compactly generated fibers. 
Consider a morphism $D=(I, F) \rightarrow (\cdot, S)$ such that $F$ is a diagram of universally $\DD$-local morphisms. 
Let $U \rightarrow S$ be a universally $\DD$-local morphism. 
Write $D_U := D \times_{(\cdot, S)} (\cdot, U)$ in $\Dia(\SSS)$.
Then the following diagram is 2-commutative (via the exchange natural isomorphism):
\[ \xymatrix{ 
\DD(D)  \ar[d]_{\pr_1^*} \ar[rr]^-{\Box_!} && \DD(D)^{\mathrm{cart}} \ar[d]^{\pr_1^*} \\
\DD(D_U) \ar[rr]^-{\Box_! } && \DD(D_U)^{\mathrm{cart}}
} \]
\end{LEMMA}
Note that left Cartesian projectors exist for $D$ and $D_U$ by Theorem~\ref{THEOREMEXISTLEFTPROJ}.

\begin{proof} The functor $\pr_1^*$ has a right adjoint $\pr_{1*}$ by (Dloc2 left) and then Brown representability theorem.
 (Dloc1 left) says that $\pr_1^*$ preserves coCartesian morphisms, hence ${\pr_{1*}}$ preserves Cartesian morphisms. 
Therefore the right adjoint of the given diagram is the following {\em commutative} diagram:
\[ \xymatrix{ 
\DD(D)  & \DD(D)^{\mathrm{cart}} \ar[l]  \\
\DD(D_U) \ar[u]^{\pr_{1*}}  & \DD(D_U)^{\mathrm{cart}}   \ar[u]_{\pr_{1*}} \ar[l] 
} \]
Consequently the exchange morphism of the diagram in the statement is also a natural isomorphism. 
\end{proof}

\begin{LEMMA}[right]\label{LEMMAPROJECTIONRIGHT}
Let $\Dia$ be an infinite diagram category and
let $\DD \rightarrow \SSS$ be an infinite fibered derivator with domain $\Dia$ with stable, compactly generated fibers. 
Consider a morphism $D=(I, F) \rightarrow (\cdot, S)$.
Let $S \rightarrow U$ be a universally $\DD$-colocal morphism. 
Write $D_U := D \times_{(\cdot, S)} (\cdot, U)$ in $\Dia^{\op}(\SSS)$.
Then the following diagram is 2-commutative (via the exchange natural isomorphism):
\[ \xymatrix{ 
\DD(D) \ar[d]_{\pr_1^*} \ar[rr]^-{\Box_*} && \DD(D)^{\mathrm{cocart}} \ar[d]^{\pr_1^*} \\
\DD(D_U) \ar[rr]^-{\Box_* } && \DD(D_U)^{\mathrm{cocart}}
} \]
\end{LEMMA}
Note that right Cartesian projectors exist for $D$ and $D_U$ by Theorem~\ref{THEOREMEXISTRIGHTPROJ}.

\begin{proof} The functor $\pr_1^*$ has a left adjoint $\pr_{1!}$ by (Dloc2 right) and by the Brown representability theorem for the dual.
 (Dloc1 right) says that $\pr_1^*$ preserves Cartesian morphisms, hence ${\pr_{1!}}$ preserves coCartesian morphisms. 
Therefore the right adjoint of the given diagram is the following {\em commutative} diagram:
\[ \xymatrix{ 
\DD(D)  & \DD(D)^{\mathrm{cocart}} \ar[l]  \\
\DD(D_U) \ar[u]^{\pr_{1!}}  & \DD(D_U)^{\mathrm{cocart}}   \ar[u]_{\pr_{1!}} \ar[l] 
} \]
Consequently the exchange morphism of the diagram in the statement is also a natural isomorphism. 
\end{proof}

\begin{LEMMA}[left]\label{LEMMAFUNDAMENTALLEFT}
Let $\DD \rightarrow \SSS$ be a fibered derivator with domain $\Dia$ admitting a left Cartesian projection (cf.\@ \ref{DEFCARTPROJ}). For any Grothendieck opfibration
\[ \xymatrix{ I \ar[d]^\pi   \\
  E   }\]
in $\Dia$, for any diagram in $F \in \SSS(I)$, and for each element $e \in E$, 
the following diagram is 2-commutative:
\[ \xymatrix{ 
\DD(I)_F  \ar[d]_{\iota^*} \ar[r]^-{\Box^E_!} & \DD(I)^{E-\mathrm{cart}}_{F} \ar[d]^{\iota^*} \\
\DD(I_e)_{F_e} \ar[r]^-{\Box_!} & \DD(I_e)^{\mathrm{cart}}_{F_e}
} \]
where $\iota: I_e \rightarrow I$ is the inclusion of the fiber.
\end{LEMMA}

\begin{LEMMA}[right]\label{LEMMAFUNDAMENTALRIGHT}
Let $\DD \rightarrow \SSS$ be a fibered derivator with domain $\Dia$ admitting a right coCartesian projection (cf.\@ \ref{DEFCARTPROJ}). For a Grothendieck fibration
\[ \xymatrix{ I \ar[d]^\pi   \\
  E   }\]
in $\Dia$, for any diagram in $F \in \SSS(I)$, and for each element $e \in E$, 
the following diagram is 2-commutative:
\[ \xymatrix{ 
\DD(I)_F  \ar[d]_{\iota^*} \ar[r]^-{\Box^E_*} & \DD(I)_F^{E-\mathrm{cocart}} \ar[d]^{\iota^*} \\
\DD(I_e)_{F_e} \ar[r]^-{\Box_*} & \DD(I_e)_{F_e}^{\mathrm{cocart}} 
} \]
where $\iota: I_e \rightarrow I$ is the inclusion of the fiber.
\end{LEMMA}
\begin{proof}We restrict to the right-variant, the other being dual.
We will show that the functor ${\iota_{!}}$ maps coCartesian objects to $E$-coCartesian ones. Then
the left adjoint of the given diagram is the diagram
\[ \xymatrix{ 
\DD(I)_F  & \DD(I)_F^{E-\mathrm{cocart}} \ar[l]  \\
\DD(I_e)_{F_e} \ar[u]^{\iota_{!}}  & \DD(I_e)_{F_e}^{\mathrm{cocart}}   \ar[u]_{\iota_{!}} \ar[l] 
} \]
which is commutative. Consequently also the diagram of the statement is 2-commutative via the natural exchange morphism.

Let $f$ in $E$ be an object and
let $\nu: i_1 \rightarrow i_2$ be a morphism in $I$ mapping to $\id_f$. 
Let $\alpha_k$ be the inclusions of $\cdot$ into $I$ with image $i_k$.
The morphism $\nu$ yields a natural transformation
\[  \nu: \alpha_1 \Rightarrow \alpha_2. \]
Consider the diagram
\[ \xymatrix{ 
 e \times_{/ E} f \ar@<.5ex>[r]^-{c_k} \ar[d]^{p'}  &  I_e  \times_{/ I} i_k \ar[r]^-{A_k} \ar@<.5ex>[l]^-{\pi} \ar[d]_{p_k} \ar@{}[dr]|{\Swarrow^{\mu_k}} & I_e  \ar[d]^{\iota}  \\
 \cdot  \ar@{=}[r]  & \cdot \ar[r]^{\alpha_k} & I
} \]
where $c_k$ is given on a morphism $\beta: e \rightarrow f$ in $E$ by the choice of a Cartesian arrow $i_k' \rightarrow i_k$. It is right adjoint to $\pi$ by
the definition of Cartesian arrow.

There is a functor (composition with $\nu$):
\[ \widetilde{\nu}: I_e  \times_{/ I} i_1  \rightarrow I_e \times_{/ I} i_2  \]
such that $A_2 \widetilde{\nu} = A_1$ and $p_2 \widetilde{\nu} = p_1$.
We have therefore a natural (point-wise) coCartesian morphism $\SSS(\mu_1)_\bullet \widetilde{\nu}^* \rightarrow \widetilde{\nu}^* \SSS(\mu_2)_\bullet$ 
of functors $\DD(I_e \times_{/ I} i_2)_{A_2^* F_e} \rightarrow \DD(I_e \times_{/ I} i_1)$.

We have also a natural transformation $\rho:  \widetilde{\nu} c_1 \rightarrow c_2$ defined for a morphism
$\beta: e\rightarrow f$ in $E$ as the unique arrow $\rho(\beta)$ over $\id_e$ making the following diagram commutative:
\[ \xymatrix{
i_1' \ar[r]^{c_1(\beta)} \ar[d]_{\rho(\beta)}&i_1 \ar[d]^\nu \\
i_2' \ar[r]^{c_2(\beta)} &i_2 
}\]
The resulting morphism $\DD(\rho): c_1^*\widetilde{\nu}^* \rightarrow c_2^*$ is point-wise coCartesian on coCartesian objects.

We get a commutative diagram of natural transformations
\[ \xymatrix{
\SSS(\mu_1)_\bullet A_1^*   \ar[r] \ar[d]^{\sim} &\SSS(\mu_1)_\bullet A_1^* \iota^* \iota_!  \ar[rr]^-{\DD(\mu_1)'} \ar[d]^{\sim} &&  p^*_1 \alpha_1^*  \ar[d]^{\sim}  \iota_!  \\
 \SSS(\mu_1)_\bullet  \widetilde{\nu}^* A_2^*  \ar[d]^{} \ar[r]^{} &\SSS(\mu_1)_\bullet  \widetilde{\nu}^*   A_2^* \iota^* \iota_! \ar[d]^{} &&  \widetilde{\nu}^* p_2^* \alpha_1^*  \iota_!  \ar[d]^{\widetilde{\nu}^* p_2^*(\DD(\nu))} \\
\widetilde{\nu}^* \SSS(\mu_2)_\bullet   A_2^*  \ar[r]^{} & \widetilde{\nu}^* \SSS(\mu_2)_\bullet   A_2^* \iota^* \iota_!  \ar[rr]^-{\widetilde{\nu}^*(\DD(\mu_2)')} &&  \widetilde{\nu}^* p_2^* \alpha_2^*  \iota_!  \\
} \]
where the first two top vertical morphisms are the natural isomorphisms induced by $A_2 \widetilde{\nu} = A_1$, the third top vertical morphism is the natural isomorphism induced by
$p_2 \widetilde{\nu} = p_1$, and the first two lower vertical morphisms are point-wise coCartesian. 
Here we use the notation $\DD(\mu_1)'$ for the morphism $\SSS(\mu_1)_\bullet X \rightarrow Y$ induced by a morphism $\DD(\mu_1): X \rightarrow Y$.

Now we apply $p_{1!}$ to the outer square:
\[ \xymatrix{
p_{1!} \SSS(\mu_1)_\bullet A_1^*   \ar[r] \ar[d]^{}  & p_{1!} p^*_1 \alpha_1^*  \ar[d]^{}  \iota_!  \\
p_{1!} \widetilde{\nu}^* \SSS(\mu_2)_\bullet   A_2^*  \ar[r]^{}  & p_{1!} \widetilde{\nu}^* p_2^* \alpha_2^*  \iota_!  \\
} \]
The left vertical map is still coCartesian (homotopy colimits preserve coCartesian morphisms).

There is a canonical isomorphism $p'_! c^*_i \rightarrow p_{i!}$ \cite[Prop. 1.23]{Gro13} 
and the natural transformation $\DD(\rho): c^*_1 \widetilde{\nu}^* \rightarrow c^*_2 $ is an isomorphism on coCartesian objects over constant diagrams.
Consider the commutative diagram:
\[ \xymatrix{ 
p'_! c^*_1 \widetilde{\nu}^* \ar[r]^-{\sim} \ar@{=}[d]  & p_{2!} \widetilde{\nu}_! c_{1!} c^*_1 \widetilde{\nu}^*   \ar[r] & p_{2!} \ar@{=}[dd] \\
p'_! c^*_1 \widetilde{\nu}^*  \ar[r]^-{\sim} & p_{2!} c_{2!} c^*_1 \widetilde{\nu}^*  \ar[u]_{\DD(\rho)^{ad}} \ar[d]^{\DD(\rho)} & \\
& p_{2!} c_{2!} c_2^*  \ar[r]^-\sim & p_{2!} } \]
where the rightmost horizontal morphisms are the respective counits.
Since $\DD(\rho)$ is an isomorphism on coCartesian objects over constant diagrams, so is the morphism $p'_! c^*_1 \widetilde{\nu}^* \rightarrow p_{2!}$.
Now we have the commutative diagram
\[ \xymatrix{
p'_! c^*_1 \widetilde{\nu}^*   \ar[r] \ar[d]^{\sim}  & p_{2!}    \\
p_{1!} \widetilde{\nu}^*  \ar[r]^{\sim}  & p_{2!} \widetilde{\nu}_! \widetilde{\nu}^*  \ar[u]^{}  \\
} \]
which shows that also the natural map $p_{1!} \widetilde{\nu}^* \rightarrow p_{2!}$ is an isomorphism on coCartesian objects over constant diagrams.

We get a commutative diagram
\[ \xymatrix{
p_{1!} \SSS(\mu_1)_\bullet A_1^*   \ar[r] \ar[d]^{}  & p_{1!} p^*_1 \alpha_1^*  \ar[d]^{}  \iota_!  \ar[r] & \alpha_1^*  \iota_! \ar[dd]^{\DD(\nu)}  \\
p_{1!} \widetilde{\nu}^* \SSS(\mu_2)_\bullet   A_2^*  \ar[d]^{} \ar[r]^{}  & p_{1!} \widetilde{\nu}^* p_2^* \alpha_2^*  \iota_!  \ar[d]^{} &  \\
p_{2!} \SSS(\mu_2)_\bullet  A_2^*  \ar[r]^{}  & p_{2!}  p_2^* \alpha_2^*  \iota_!  \ar[r]^{}  &  \alpha_2^*  \iota_! \\
} \]
where the composition of the left vertical arrows is coCartesian on coCartesian objects because the functor $\SSS(\mu_2)_\bullet   A_2^*$ 
maps coCartesian objects to coCartesian objects over constant diagrams.
The composition of the horizontal morphisms in the top and bottom rows are isomorphisms by (FDer4 left). Hence the rightmost vertical map is coCartesian, too. 
\end{proof}

\begin{proof}[Proof of Main Theorem \ref{MAINTHEOREMHOMDESCENT}, 1.] This is the case of weak $\DD$-equivalences.

(L0) and (L1) are clear. 

For (L2 left), let $D_1=(I, F)$ and $D_2=(\{e\}, F(e))$. The
projection $p$ and the inclusion $i$ of the final object induce morphisms:
\[ \xymatrix{
 D_1 \ar@<0.5ex>[r]^p & D_2 \ar@<0.5ex>[l]^i }
\]
We have $p \circ i = \id$ and there is a 2-morphism $\beta: \id \Rightarrow i \circ p$.
Therefore the statement is clear for weak $\DD$-equivalences over any base $S$.

(L3 left):
Let 
\[ \xymatrix{
D_1 \ar[rr]^w  \ar@/_10pt/[rdd]_{p_1}  \ar[dr]_{p_1'} && D_2 \ar[dl]^{p_2'}  \ar@/^10pt/[ldd]^{p_2} \\
& D_3=(E, F) \ar[d]^p \\
& (\cdot, S) 
} \]
 be a morphism as in (L3 left) over a base $S \in \SSS(\cdot)$.
We have to show that 
\[ p_{1!}\,p_{1}^* \rightarrow p_{2!}\,p_2^*  \]
is an isomorphism
and it suffices to show that the morphism 
\[ p_{1!}'\,(p_{1})^* \rightarrow p_{2!}'\,(p_2)^*  \]
is an isomorphism. This may be checked point-wise by (Der2)  and after pull-back to an open cover by condition 2.\@ of `local' for a fibered derivator (see Definition~\ref{DEFFIBEREDLOCAL}), so fix $e \in E$ and consider the 2-commutative diagrams
\[ \xymatrix{
D_{i} \times_{/D_3} (e, U_j) \ar[d]_-{p_{i,e,j}'} \ar[r]^-{\iota_{i,e,j}} & D_{i} \ar[d]^{p_i'} \\
(e, U_j) \ar[r]^{\epsilon_{e,j}} & D_3
} \]
and let $p_{i,e}: D_{i} \times_{/D_3} (e, U_i) \rightarrow D_i$ be the projection. 
Applying the functor $\epsilon_{e,j}^*$, we get
\[ \epsilon_{e,j}^* p_{1!}' (p_{1})^* \rightarrow \epsilon_{e,j}^* p_{2!}'(p_2)^*  \]
which is, using Proposition~\ref{THEOREMBASECHANGEDIALEFT} (note that $\iota_{e,j}$ is $\DD$-local by assumption), the same as
\[ (p_{1,e,j}')_! (\iota_{i,e,j})^* (p_{1})^* \rightarrow  (p_{2,e,j}')_!  (\iota_{i,e,j})^* (p_{2})^*.  \]
Now $p_{i} \circ \iota_{i,e,j} = \pi_j \circ p_{i,e,j}'$, where $\pi: (\cdot, U_j) \rightarrow (\cdot, S)$ is the structural morphism. 
Therefore we get:
\[ (p_{1,e,j}')_! (p_{1,e,j}')^* \pi_j^* \rightarrow  (p_{2,e,j}')_!  (p_{2,e,j}')^* \pi_j^* .  \]
By Lemma~\ref{LEMMAPASTINGDIA} this is induced by the canonical natural transformation which is an isomorphism by assumption.

(L4 left): By Lemma~\ref{PROPL4} we may prove axiom (L4' left) instead.
Consider a morphism $p: D_1 \rightarrow (E, F)=D_2$ in $\Dia(\SSS)$ of pure diagram type, where the underlying functor of $p$ is a Grothendieck fibration. 
It suffices to show that  the counit
\[ p_!\,p^* \rightarrow \id \]
is an isomorphism. 
This is the same as showing that the unit
\[ \id \rightarrow p_*\,p^*  \]
is an isomorphism. 
Note that $p_*$ exists because this is a morphism of diagram type and $\DD \rightarrow \SSS$ is assumed to be a right fibered derivator, too (this is
the only place, where this assumption is used for the case of weak $\DD$-equivalences). Now, since $p$ is a Grothendieck fibration, $p_*$ can be computed fiber-wise. So we have to show that
\[ \id \rightarrow p_{e,*}\,p_e^{*}  \]
is an isomorphism or, equivalently, that
\[  p_{e,!}\,p_e^* \rightarrow \id \]
is an isomorphism. This holds true because by assumption the map of fibers $I_e \rightarrow e$ is in $\mathcal{W}_{F(e)}$. 
\end{proof}

We proceed to state some consequences of the fact that {\em weak} $\DD$-equivalences form a fundamental localizer.

\begin{BEISPIEL}[Mayer-Vietoris]\label{MV2}
Let $\SSS$ be a strong right derivator (e.g.\@ associated with a category with limits) with a Grothendieck pre-topology.
We saw in Example~\ref{MV1} that for a cover $\{ U_1 \rightarrow S, U_2 \rightarrow S\}$ consisting of 2 {\em mono}morphisms, the projection
\[ p: \left( \vcenter{ \xymatrix{ \mlq \mlq U_1 \times_S U_2 \mrq \mrq \ar[r] \ar[d] & U_1 \\ U_2 & } } \right) \rightarrow S \]
belongs to any fundamental localizer. If $\DD \rightarrow \SSS$ is a fibered derivator which is local w.r.t.\@ the pre-topology on $\SSS$, Theorem~\ref{MAINTHEOREMHOMDESCENT} implies therefore that $p$ is a weak $\DD$-equivalence in $\Dia(\SSS)/(\cdot, S)$, i.e.\@ for $A \in \DD(\cdot)_S$ we have
 \[ p_!\,p^* A \cong A,\] i.e.\@ the
homotopy colimit of 
\[ \xymatrix{i_{1,2,\bullet}\,i_{1,2}^\bullet A \ar[r] \ar[d] & i_{1,\bullet}\,i_1^\bullet A \\ i_{2,\bullet}\,i_2^\bullet A & } \]
is isomorphic to $A$. If $\DD$ has stable fibers, this translates to the usual 
distinguished triangle
\[ i_{1,2,\bullet}\,i_{1,2}^\bullet A \rightarrow  i_{1,\bullet}\,i_1^\bullet A \oplus i_{2,\bullet}\,i_2^\bullet A \rightarrow A \rightarrow  i_{1,2,\bullet}\,i_{1,2}^\bullet A[1]   \]
in the language of triangulated categories.

Dually, if $\DD \rightarrow \SSS^{\op}$ is a fibered derivator which is colocal w.r.t.\@ the pre-cotopology on $\SSS^{\op}$, Theorem~\ref{MAINTHEOREMCOHOMDESCENT} implies  that $p^{\op}$ is a weak $\DD$-equivalence in $\Dia^{\op}(\SSS^{\op})/(\cdot, S)$, i.e.\@ for $A \in \DD(\cdot)_S$ we have
 \[ A \cong p_* p^* A.\] 
This means that the homotopy limit of 
\[ \xymatrix{  & i_1^{\bullet}\,i_{1,\bullet} A \ar[d]  \\ i_2^{\bullet}\,i_{2,\bullet} A \ar[r] & i_{1,2}^\bullet\,i_{1,2, \bullet} A } \]
is isomorphic to $A$. If $\DD$ has stable fibers, this translates to the usual 
distinguished triangle
\[ A \rightarrow  i_{1}^\bullet\,i_{1,\bullet} A \oplus i_{2}^\bullet\,i_{2,\bullet} A \rightarrow  i_{1,2}^{\bullet}\,i_{1,2,\bullet} A \rightarrow  A[1]   \]
in the language of triangulated categories. Note that $i_\bullet$ denotes a left adjoint push-forward along a morphism {\em in $\SSS^{\op}$}, i.e.\@ 
a left adjoint pull-back along a morphism in $\SSS$. 
\end{BEISPIEL}

\begin{BEISPIEL}[(Co)homological descent]
Let $\SSS$ be a strong right derivator with a Grothendieck pre-topology and
let $X_\bullet \in \SSS(\Delta^{\op})$ be a simplicial diagram over $S \in \SSS(\cdot)$ with underlying diagram
\[ \xymatrix{ \cdots \ar@<-1.5ex>[r]\ar@<-0.5ex>[r]\ar@<0.5ex>[r]\ar@<1.5ex>[r]  &   X_2 \ar@<1ex>[r]\ar@<0ex>[r]\ar@<-1ex>[r] & X_1 \ar@<0.5ex>[r] \ar@<-0.5ex>[r] & X_0  } \]
such that $(\id,p): (\Delta^{\op}, X_\bullet) \rightarrow (\Delta^{\op}, \pi^* S)$ is a finite hypercover.  Here $\pi: \Delta^{\op} \rightarrow \cdot$ denotes the projection. 
If $\DD \rightarrow \SSS$ is a fibered derivator which is local w.r.t.\@ the pre-topology on $\SSS$, Theorem~\ref{MAINTHEOREMHOMDESCENT} implies that $(\pi, p)$ is a weak $\DD$-equivalence in $\SSS(\cdot)/(\cdot, S)$, i.e.\@ for $A \in \DD(\cdot)_S$ we have
 \[ A \cong \pi_!\,p_\bullet\,p^\bullet\,\pi^* A.\] 
 This means that the 
homotopy colimit of $p_\bullet \,p^\bullet\, \pi^* A$ is equal to $A$. If the fibers of $\DD \rightarrow \SSS$ are in fact {\em derived categories}, this
yields a spectral sequence of homological descent because the homotopy colimit over a simplicial complex is 
the total complex of the associated double complex (a well-known fact). This double complex looks like
\[ \xymatrix{ \cdots \ar[r] & p_{2,\bullet}\,p_2^\bullet A \ar[r] & p_{1,\bullet}\,p_1^\bullet A \ar[r] & p_{0,\bullet}\,p_0^\bullet A. } \]
The point is that we get a {\em coherent} double complex. Knowing the individual morphisms
$p_{i,\bullet}\,p_i^\bullet A \rightarrow p_{i-1,\bullet}\,p_{i-1}^\bullet A$ as morphisms in the derived category $\DD(\cdot)_S$ would not be sufficient!

Dually 
(applying everything to a fibered derivator $\DD \rightarrow \SSS^{\op}$, and working in $\Dia^{\op}(\SSS^{\op}))$, one obtains the more classical spectral sequence of {\em co}homological descent. 
\end{BEISPIEL}

\begin{proof}[Proof of Main Theorem \ref{MAINTHEOREMHOMDESCENT}, 2.] This is the case of strong $\DD$-equivalences.

(L1) is clear.

For (L2 left), let $D_1=(I, F)$ and $D_2=(e, F(e))$. The
projection $p$ and the inclusion $i$ of the final object induce morphisms:
\[ \xymatrix{
 D_1 \ar@<0.5ex>[r]^p & D_2 \ar@<0.5ex>[l]^i }
\]
We have $p \circ i = \id$ and there is a 2-morphism $\beta: \id \Rightarrow i \circ p$.
Therefore the statement follows from Lemma~\ref{LEMMAHOMOTOPY}.
(Actually $(i \circ p)^*$ is left adjoint to the inclusion $\DD(D_1)^{\mathrm{cart}} \rightarrow \DD(D_1)$.)

(L3 left): It suffices to prove the following two statements:

\begin{enumerate}
\item Consider a morphism of diagrams $w=(\alpha, f): D_1=(I_1, F_1) \rightarrow D_2=(I_2, F_2)$
such that we have a commutative diagram
\[ \xymatrix{
I_1 \ar[rr]^\alpha \ar[dr]_{p_1}  && I_2 \ar[dl]^{p_2}  \\
& E 
} \]
and such that $w \times_{/E} e$ is a strong $\DD$-equivalence for all objects $e$ in $E$. 
Then $w$ is a strong $\DD$-equivalence. 

\item  Consider a morphism of diagrams $w: D_1=(I_1, F_1) \rightarrow D_2=(I_2, F_2)$ over $(\cdot, S)$ and let 
$\{U_i \rightarrow S\}$ be a covering. If $w \times_{(\cdot, S)} (\cdot, U_i)$ is a strong $\DD$-equivalence for all $i$ then $w$ is a strong $\DD$-equivalence. 
\end{enumerate}

We proceed by showing statement 1. 
Consider the following diagram over $E$ 
\[  \xymatrix{
D_1\ar[d] \ar[r]^w & D_2 \ar[d] \\
D_1\times_{/E} E \ar[r]^{w'} & D_2 \times_{/E} E
} \]
where the vertical morphisms are of pure diagram type. 
We have an adjunction
\[ \xymatrix{ I_i \ar@<2pt>[rr]^-{\kappa_i} && \ar@<2pt>[ll]^-{\iota_i} I_i \times_{/E} E } \]
where $\kappa_i$ maps an object $i$ to $(i, \id_{p(i)})$. 
We have a natural transformation  $\kappa_i \circ \iota_i \Rightarrow \id_{I_i \times_{/E} E}$ and moreover $\iota_i \circ \kappa_i = \id_E$ holds.
Actually this defines an adjunction with $\kappa_i$ left-adjoint to $\iota_i$. Furthermore, we get lifts to diagrams
\[ \xymatrix{ D_i \ar@<2pt>[rr]^-{\widetilde{\kappa}_i} && \ar@<2pt>[ll]^-{\widetilde{\iota}_i} (I_i \times_{/E} E, \iota_i \circ F) = D_1\times_{/E} E, } \]
and a 2-morphism $\widetilde{\kappa}_i \circ \widetilde{\iota}_i \Rightarrow \id_{D_1\times_{/E} E}$, and we have $\widetilde{\iota}_i \circ \widetilde{\kappa}_i = \id_{D_1}$.

Hence, by Lemma~\ref{LEMMAHOMOTOPY},
the pull-backs along $\widetilde{\iota}_1$ and $\widetilde{\iota}_2$ induce equivalences on Cartesian objects, so
 we are reduced to showing that the pull-back along $w'$ induces an equivalence on Cartesian objects. 
The underlying diagrams $I_k \times_{/E} E$ are Grothendieck opfibratons over $E$ and the functor underlying $w'$ is a map of Grothendieck opfibrations
(the push-forward along a map $\mu: e \rightarrow f$ in $E$ being given by mapping $(i, \nu: p(i) \rightarrow e)$ to $(i, \nu \circ \mu)$).
Hence w.l.o.g. we may assume that  $I_1 \rightarrow E$ is a Grothendieck opfibration and the morphism $I_1 \rightarrow I_2$ underlying $f$ is
a morphism of Grothendieck opfibrations. 

We keep the notation $w: D_1 \rightarrow D_2$, however, and the assumption translates to the statement that
the composition
\[ \xymatrix{ \DD(D_{2,e})^{\mathrm{cart}} \ar[r]^-{w_e^*} & \DD(D_{1,e})^{\mathrm{cart}} }  \]
for the fibers is an equivalence with inverse $\Box_! w_{e,!}$.

Consider the two functors:
\[ \xymatrix{ \DD(D_2)^{E-\mathrm{cart}} \ar[r]^-{incl.} & \DD(D_2) \ar[r]^-{w^*} & \DD(D_1). }  \]

We first show that the counit 
\[ \Box_!^E w_! w^* \mathcal{E} \rightarrow \mathcal{E} \]
is an isomorphism for every $E$-Cartesian $\mathcal{E}$.

This can be checked after pulling back to the fibers. Let $\iota_k: I_{k,e} \rightarrow I_k$ be the inclusion of the fiber over some $e \in E$.

We have the isomorphisms
\[  \iota_2^* \Box_!^E w_! w^* \mathcal{E} \cong \Box_!  w_{e,!}  \iota_1^* w^*  \mathcal{E} \cong  \Box_! w_{e,!}  w^{e,*} \iota_2^* \mathcal{E} \cong \iota_2^* \mathcal{E}, \]
where we used the isomorphism $\iota_2^* \Box_!^E \cong  \Box_! \iota_2^*$ (Lemma~\ref{LEMMAFUNDAMENTALLEFT}) and the isomorphism $\iota_2^* w_{!} \cong w_{e,!} \iota_1^*$ (exists for morphisms of pure diagram type because we have a morphism of Grothendieck cofibrations, see Proposition~\ref{PROPHOMCART},~3. and for morphisms of fixed shape by axiom (FDer0 left)).  The morphism $\Box_! w_{e,!}  w^{e,*}\mathcal{E} \rightarrow \mathcal{E}$ is an isomorphism for Cartesian $\mathcal{E}$ by assumption.

We now show that the unit
\[  \mathcal{E} \rightarrow w^* \Box_!^E w_! \mathcal{E}  \] 
is an isomorphism for every $E$-Cartesian $\mathcal{E}$. This can be checked again on the fibers:
\[ \iota_1^* w^* \Box_!^E w_! \mathcal{E} \cong w_e^* \iota_2^* \Box_!^E w_! \mathcal{E} \cong w_e^* \Box_! w_{e,!} \iota_1^* \mathcal{E} \cong \iota_1^* \mathcal{E}. \]

Therefore we have already proven that the functors
\[ \xymatrix{ \DD(D_2)^{E-\mathrm{cart}} \ar@<2pt>[rr]^-{w^*}  && \ar@<2pt>[ll]^-{\Box_!^E w_!} \DD(D_1)^{E-\mathrm{cart}} } \]
form an equivalence. 

We conclude by showing that $\Box_!^E w_! $ maps Cartesian objects to Cartesian objects:
Let $\nu: e \rightarrow f$ be a morphisms of $E$. It induces a morphism (choice of push-forward for $I_k \rightarrow E$)
\[  \widetilde{\nu}_k: D_{k,e} \rightarrow D_{k,f}  \]
(not of diagram type!) and a 2-morphism: $\iota_{k,e} \rightarrow \iota_{k,f} \circ \widetilde{\nu}_k$.

\begin{itemize}
\item {\em Claim:} It suffices to show that for all $\nu: e \rightarrow f$ the induced morphism
\[ \iota_{2,e}^* \Box_!^E w_!  \mathcal{E} \rightarrow  \widetilde{\nu}_2^* \iota_{2,f}^* \Box_!^E w_! \mathcal{E} \] 
is an isomorphism for every Cartesian $\mathcal{E}$.
\end{itemize}

{\em Proof of the claim:} Every morphism $\mu: i \rightarrow i''$ in $I$ with $p(\mu) = \nu$, say,
is the composition of a coCartesian $\mu'$ and some morphism $\mu''$ with $p(\mu'') = \id_f$. Since $\mathcal{E}$ is $E$-Cartesian, the
morphism $\mathcal{E}(\mu'')$ is Cartesian. Hence to show that $\mathcal{E}(\mu)$ is Cartesian it suffices to see that $\mathcal{E}(\mu')$ is Cartesian. 
A reformulation is, however, that the morphism of the claim be an isomorphism. \qed

Using the same argument as in the first part of the proof, we have to show that 
\[ \Box_!  w_{e,!} \iota_{1,e}^* \mathcal{E} \rightarrow  \widetilde{\nu}_2^* \Box_! w_{f,!}  \iota_{1,f}^* \mathcal{E} \] 
is an isomorphism for every Cartesian $\mathcal{E}$.
Since both sides are Cartesian objects, this can be checked after applying $w_{e}^*$ which is an equivalence on Cartesian objects:
\[ w_{e}^* \Box_!  w_{e,!} \iota_{1,e}^* \mathcal{E} \rightarrow w_{e}^* \widetilde{\nu}_2^* \Box_! w_{f,!}  \iota_{1,f}^* \mathcal{E}. \] 
We have $w_{e}^* \widetilde{\nu}_2^* = \widetilde{\nu}_1^* w_{f}^*$ because the map of diagrams underlying $w$ is a morphism of Grothendieck opfibrations. Hence,  after applying $w_{e}^*$, we get
\[ w_{e^*} \Box_!  w_{e,!} \iota_{1,e}^* \mathcal{E} \rightarrow  \widetilde{\nu}_1^* w_{f^*} \Box_! w_{f,!}  \iota_{1,f}^* \mathcal{E}. \] 
Since $w_{e^*} \Box_!  w_{e,!}$ and $w_{f^*} \Box_!  w_{f,!}$ are equivalences on Cartesian objects, we get
\[ \iota_{1,e}^* \mathcal{E} \rightarrow \widetilde{\nu}_1^*  \iota_{1,f}^* \mathcal{E}. \] 
A slightly tedious check shows that this is again the morphism induced by the 2-morphism  $\iota_{1,e} \rightarrow \iota_{1,f} \circ \widetilde{\nu}_1$. It is an isomorphism because $\mathcal{E}$ is Cartesian. 

We will now show statement 2. 
Consider a diagram
\[ \xymatrix{
D_1 \ar[rr]^w \ar[dr]_{p_1}  && D_2 \ar[dl]^{p_2}  \\
& (\cdot, S) 
} \]

For any $i$ (index of the cover in L4 left) we have the following commutative diagrams of objects in $\Dia(\SSS)$:
\[ \xymatrix{
D_1 \times_S U_i \ar[r]^{w_i} \ar[d]_{\pr_1^{(i)}} & D_2 \times_S U_i \ar[d]^{\pr_1^{(i)}} \\
D_1 \ar[r]^w & D_2
} \]
The morphisms $\pr_1^{(i)}$ are of fixed shape. 
We first show that the unit is an isomorphism
\[  \mathcal{E} \rightarrow w^* \Box_!\, w_!\, \mathcal{E}  \] 
for any Cartesian $\mathcal{E}$. 
Note that by the stability axiom of a Grothendieck pre-topology also the collections $(D_1 \times_S U_i)_j \rightarrow D_{1,j}$
are covers for any $j \in I_1$, where $I_1$ is the underlying diagram of $D_1$. 
Since $\DD$ is local w.r.t.\@ the Grothendieck pre-topology (and by axiom Der2), the family
$(\pr_1^{(i)})^*$ is jointly conservative. 
Therefore it suffices to show that the unit is an isomorphism after applying $(\pr_1^{(i)})^*$.
We get
\[ (\pr_1^{(i)})^* \mathcal{E} \rightarrow (\pr_1^{(i)})^*\, w^* \Box_!\, w_!\, \mathcal{E}  \] 
which is the same as
\[ (\pr_1^{(i)})^* \mathcal{E} \rightarrow w_i^*\, (\pr_1^{(i)})^* \Box_!\, w_!\, \mathcal{E}.  \] 
Since $(\pr_1^{(i)})^* $ commutes with $\Box_!$ (Lemma \ref{LEMMAPROJECTIONLEFT}) and with $w_!$ (Proposition~\ref{THEOREMBASECHANGEDIALEFT},~2.\@), we get
\[ (\pr_1^{(i)})^* \mathcal{E} \rightarrow w_i^*  \Box_!\, w_{i,!}\,  (\pr_1^{(i)})^* \mathcal{E}.  \] 
This morphism is an isomorphism by assumption.
In the same way one shows that the counit is an isomorphism.

(L4 left): By Lemma~\ref{PROPL4} we may prove axiom (L4' left) instead.
We have shown during the proof for (L4' left) for the case of weak $\DD$-equivalences that
\[ p_!\, p^* \rightarrow \id \]
is an isomorphism, hence on Cartesian objects the same holds for the natural transformation
\[ \Box_!\, p_!\, p^* \rightarrow \id. \]
We have to show that also the counit
\begin{equation}\label{eqpbox} \id \rightarrow  p^* \Box_!\, p_!  \end{equation}
is an isomorphism on Cartesian objects. 
First note that $p_*$ also is a right adjoint of $p^*$ when restricted to the full subcategories of Cartesian objects because $p_*$ preserves Cartesian objects. 
Indeed, $p_*$ can be computed fiber-wise because $p$ is a Grothendieck fibration. The fibers being contractible in the
sense of any localizer on $\Dia$ implies that the functors $p_e^*, p_{e,*}$ induce an equivalence $\DD(D_e)^{\cart} \cong \DD(\cdot)_{F(e)}$. Note: This uses
that (L1--L3 left) hold for the class of strong $\DD$-equivalences on the fiber $\DD_{F(e)}$, a fact which has been proven already.  
Therefore we pass to the right adjoints of the functors in~(\ref{eqpbox}) and have to show that the counit
\[ p^*\, p_* \rightarrow  \id  \]
is an isomorphism on Cartesian objects. 
Again this can be checked fiber-wise, i.e.\@ we have to show that the counit
\[  p_e^{*}\, p_{e,*}  \rightarrow \id  \]
is an isomorphism on Cartesian objects. But the pair of functors is an equivalence as we have seen, and the claim follows.
\end{proof}

We proceed to state some consequences of the fact that {\em strong} $\DD$-equivalences form a fundamental localizer.

\begin{KOR}[left]
Let $\SSS$ be a strong right derivator.
If $\DD \rightarrow \SSS$ is an infinite fibered derivator which is local w.r.t.\@ the pre-topology on $\SSS$ (cf.\@ \ref{DEFPRETOP}) with stable, compactly generated fibers then 
for any finite hypercover $f: X_\bullet \rightarrow Y_\bullet$ considered in $\Dia(\SSS)'$ the functor $f^*$ induces an equivalence
\[ \DD(Y_\bullet)^{\mathrm{cart}} \rightarrow \DD(X_\bullet)^{\mathrm{cart}}. \]
Here $\Dia(\SSS)'$ is the full subcategory of diagrams with universally $\DD$-local morphisms. 
\end{KOR}

\begin{KOR}[right]
Let $\SSS$ be a strong right derivator.
If $\DD \rightarrow \SSS^{\op}$ is an infinite fibered derivator which is colocal w.r.t.\@ the pre-cotopology on $\SSS^{\op}$ (cf.\@ \ref{DEFPRETOP}) with stable, compactly generated fibers then 
for any finite hypercover $f: X_\bullet \rightarrow Y_\bullet$ considered in $\Dia^{\op}(\SSS^{\op})$ the functor $f^*$ induces an equivalence
\[ \DD(Y_\bullet)^{\mathrm{cocart}} \rightarrow \DD(X_\bullet)^{\mathrm{cocart}}. \] 
\end{KOR}

\begin{KOR}
If $\DD$ is an infinite derivator (not fibered) with domain $\Cat$ which is stable and well-generated, then for each homotopy type $I$, we get a
category $\DD(I)^{\cart}$ well-defined up to equivalence of categories. Moreover each morphism $I \rightarrow J$ of homotopy types gives rise to
a corresponding functor $\alpha^*:  \DD(J)^{\cart} \rightarrow  \DD(I)^{\cart}$. It is, however, not possible to arrange those as
a pseudo functor $\mathcal{HOT} \rightarrow \mathcal{CAT}$, but it is possible to arrange them as a pseudo-functor $\mathcal{HOT}^{(2)} \rightarrow \mathcal{CAT}$ where $\mathcal{HOT}^{(2)}$ is the homotopy 2-category (2-truncation) of
any model for the homotopy theory of spaces (cf.\@ also \ref{APPMULTI}).
\end{KOR}

\section{Representability}\label{SECTBROWN}

In this section we exploit the consequences that Brown representability type results have for fibered derivators. In particular these results are useful
to see that under certain circumstances a left fibered (multi\nobreakdash-)derivator is already a right fibered (multi\nobreakdash-)derivator, provided that its fibers are nice (i.e.\@ stable and well-generated derivators).
Furthermore they provide us with (co)Cartesian projectors that are needed for the strong form of (co)homological descent. 
In contrast to the rest of the article the results are stated in a rather unsymmetric form. This is due to the fact that in applications 
the stable derivators will rather be well-generated or compactly generated whereas their duals will rather not satisfy this condition. 
All the auxiliary results are taken from \cite{Kra09} and \cite{Nee01}.

\subsection{Well-generated triangulated categories and Brown representability}

\begin{DEF}[cf.\@ {\cite[5.1, 6.3]{Kra09}}]\label{DEFGENER}
Let $\mathcal{D}$ be a category with zero object and small coproducts. 
We call $\mathcal{D}$ {\bf perfectly generated} if there is a set of objects $\mathcal{T}$ in $\mathcal{D}$ such that the following conditions hold:
\begin{enumerate} 
\item An object $X \in \DD(\cdot)$ is zero if and only if $\Hom(T, X) = 0$ for all $T \in \mathcal{T}$.
\item If $\{ X_i \rightarrow Y_i \}$ is any set of maps, and $\Hom(T, X_i) \rightarrow \Hom(T, Y_i)$ is surjective for all $i$, then
$\Hom(T, \coprod_i X_i) \rightarrow \Hom(T,  \coprod_i Y_i)$ is also surjective.
\end{enumerate} 
The category $\mathcal{D}$ is called {\bf well-generated} if there is a set of objects $\mathcal{T}$ in $\mathcal{D}$ such that in addition to 1., 2.\@ 
there is a regular cardinal $\alpha$ such that the following condition holds:
\begin{enumerate} 
\item[3.] All objects $T \in \mathcal{T}$ are $\alpha$-small, cf.\@ \cite[6.3]{Kra09}.
\end{enumerate} 

The category $\mathcal{D}$ is called {\bf compactly generated} if there is a set of objects $\mathcal{T}$ in $\mathcal{D}$ such that in addition to 1., 2.\@ the following two equivalent conditions hold:
\begin{enumerate} 
\item[4.] All $T \in \mathcal{T}$ are $\aleph_0$-small.
\item[4'.] All $T \in \mathcal{T}$ are compact, i.e.\@ for each morphism $\gamma: T \rightarrow \coprod_{i \in I} X_i$ there is a finite subset $J \subseteq I$ such that $\gamma$ factors through $\coprod_{i \in J} X_i$. 
\end{enumerate} 

\end{DEF}

\begin{DEF}
A pre-derivator $\DD$ whose domain $\Dia$ is infinite (i.e.\@ closed under infinite disjoint unions) is called {\bf infinite} if the restriction-to-$I_j$ functors induce an equivalence
\[ \DD( \coprod_{j \in J} I_j ) \cong \prod_{j \in J} \DD(I_j) \]
for all sets $J$. 
\end{DEF}

Recall (cf.\@ \cite[4.4]{Kra09}) that a functor from a triangulated category $\mathcal{D}$ to an abelian category is called {\bf cohomological}
if it sends distinguished triangles to exact sequences.

We recall the following theorem:
\begin{SATZ}[right Brown representability]\label{THEOREMTRIANGREP}
Let $\mathcal{D}$ be a perfectly generated triangulated category with small coproducts. Then a functor $F: \mathcal{D}^{\op} \rightarrow \mathcal{AB}$ is cohomological 
and sends coproducts to products if and only if it is representable.  An exact functor $\mathcal{D} \rightarrow \mathcal{E}$ between triangulated categories
commutes with coproducts if and only if it has a right adjoint.
\end{SATZ}
\begin{proof}
See \cite[Theorem 5.1.1]{Kra09}.
\end{proof}

It can be shown that for a compactly generated triangulated category $\mathcal{D}$ with small coproducts, $\mathcal{D}^{\op}$ is 
perfectly generated and has small coproducts. Therefore the dual version of the previous theorem holds in this case: 

\begin{SATZ}[left Brown representability]\label{THEOREMTRIANGREPLEFT}
Let $\mathcal{D}$ be a compactly generated triangulated category with small coproducts. Then a functor $F: \mathcal{D} \rightarrow \mathcal{AB}$ is homological 
and sends products to products if and only if $F$ is representable.  An exact functor $\mathcal{D} \rightarrow \mathcal{E}$ between triangulated categories
commutes with products if and only if it has a left adjoint.
\end{SATZ}

\begin{SATZ}\label{THEOREMTRIANGLOC}
Let $\mathcal{D}$ be a well-generated triangulated category with small coproducts. Consider a functor $F: \mathcal{D} \rightarrow \mathcal{AB}$ which is cohomological 
and commutes with coproducts. Then there exists a right adjoint to the inclusion of the full subcategory of objects $X$ such that $F(X[n]) = 0$ for all $n \in \Z$ (i.e.\@ this subcategory is reflective).
\end{SATZ}
\begin{proof}
See \cite[Theorem 7.1.1]{Kra09}.
\end{proof}

\begin{LEMMA}\label{LEMMAGENERATION}
Let $\Dia$ be an infinite diagram category (\ref{DEFDIAGRAMCAT}).
Let $\DD \rightarrow \SSS$ be an infinite left fibered derivator with domain $\Dia$.
If $\DD(\cdot)_S$ for all $S \in \SSS(\cdot)$ is perfectly generated (resp.\@ well-generated, resp.\@ compactly generated), then the same holds for $\DD(I)_{S'}$ for all $I \in \Dia$ and for all $S' \in \SSS(I)$. Furthermore the categories $\DD(I)_{S'}$ all have small coproducts. 
\end{LEMMA}
\begin{proof}
A set of generators as requested is given by the set $\mathcal{T}_I:=\{i_! T\}_{i \in I, T \in \mathcal{T}}$.
Indeed, an object $X \in \DD(I)$ is zero if  $i^*X$ is zero for all $i \in X$ by (Der2). Therefore
$X$ is zero if $\Hom(i_!T, X) = \Hom(T, i^*X) = 0$ for all $i \in I$ and for every $T \in \mathcal{T}$.
We have to show that $\Hom(i_! T, \coprod_i X_i) \rightarrow \Hom(i_! T,  \coprod_i Y_i)$ is an isomorphism for a family $\{X_i \rightarrow Y_i\}_{i \in O}$ of morphisms as in Definition~\ref{DEFGENER}, 2.
We have $\Hom(i_! T, \coprod_i X_i) = \Hom(T, i^* \coprod_i X_i) = \Hom(T, \coprod_i i^* X_i)$, where we used that $i^*$ commutes with coproducts. This follows because
the Cartesian diagram
\[\xymatrix{
O \ar[r] \ar[d] & O \times I \ar[d] \\
 \cdot \ar[r] & I
}\]
is homotopy exact. Note that, since $\DD$ is infinite, coproducts exist and are equal to the corresponding homotopy coproducts.
The map $\Hom(T, \coprod_i i^* X_i) \rightarrow \Hom(T,  \coprod_i i^* Y_i)$ is surjective by assumption.

We have to show that a morphism
\[  i_! T \rightarrow \coprod_{i \in I} Y_i \]
in $\DD(I)_{S'}$ factors trough $\coprod_{i \in J} Y_i$ for some subset $J \subset I$ of cardinality less than $\alpha$.
By the same reasoning as above, we get a morphism
\[  T \rightarrow \coprod_{i \in I} i^* Y_i \]
Hence, there is some subset $J \subset I$, as required, such that this morphism factors through it. The same then holds for the original morphism. Since there
is no need to enlarge $J$, the same statement holds for finite subsets. 

The categories $\DD(I)_{S'}$ have small coproducts because $\DD \rightarrow \SSS$ is infinite and left fibered. 
\end{proof}

\begin{DEF}\label{DEFGENFIBERS}
Let $\DD \rightarrow \SSS$ be an infinite left fibered derivator with domain $\Dia$.
We will say that $\DD \rightarrow \SSS$ has {\bf perfectly-generated} (resp.\@ {\bf well-generated}, resp.\@ {\bf compactly-generated}) {\bf fibers}, if 
all categories $\DD(\cdot)_{S}$ are perfectly-generated (resp.\@ well-generated, resp.\@ compactly-generated) for all $S \in \SSS(\cdot)$. 
It follows from the previous Lemma that, in this case, for all $I \in \Dia$ and for all $S' \in \SSS(I)$ the category $\DD(I)_{S'}$ is also perfectly-generated (resp.\@ well-generated, resp.\@ compactly-generated). 
\end{DEF}

\subsection{Left and Right}

\begin{SATZ}[left]\label{SATZLEFTBROWNDER}
Let $\Dia$ be an infinite diagram category (cf.\@ \ref{DEFDIAGRAMCAT}).
Let $\DD$ and $\EE$ be infinite left derivators with domain $\Dia$ such that for all $I \in \Dia$ the pre-derivators $\DD_I$ and $\EE_I$ are stable (left and right) derivators with domain $\Posf$.
Assume that $\DD$ is perfectly generated.
Then a morphism of derivators $F: \DD \rightarrow \EE$ commutes with all homotopy colimits w.r.t.\@ $\Dia$ if and only if it has a right adjoint. 
\end{SATZ}
\begin{proof}
Let $I$ be in $\Dia$. Since $\DD_I$ and $\EE_I$ are stable, $\DD(I)$ is canonically triangulated, and we may use Theorem~\ref{THEOREMTRIANGREP} of
right Brown representability.
It follows that the functor $F(I): \DD(I) \rightarrow \EE(I)$ has a right adjoint $G(I)$, 
because it is triangulated, commutes with small coproducts and $\DD(I)$ is perfectly generated.
To construct a morphism of derivators out of this collection, for any $\alpha: I \rightarrow J$, we have to give an isomorphism:
$G(J) \alpha^* \rightarrow \alpha^*G(I)$. We may take the adjoint of the isomorphism $\alpha_! F(J) \rightarrow F(I) \alpha_!$ expressing that 
$F$ commutes with all homotopy colimits (see \cite[Lemma 2.1]{Gro13} for details).
\end{proof}

Analogously, using Theorem~\ref{THEOREMTRIANGREPLEFT} of left Brown representability, we obtain:

\begin{SATZ}[right]Let $\Dia$ be an infinite diagram category (cf.\@ \ref{DEFDIAGRAMCAT}).
Let $\DD$ and $\EE$ be infinite right derivators with domain $\Dia$ such that for all $I \in \Dia$, the pre-derivators $\DD_I$ and $\EE_I$ are stable (left and right) derivators with domain $\Posf$.
Assume that $\DD$ is compactly generated.
Then a morphism of derivators $F: \DD \rightarrow \EE$ commutes with all homotopy limits w.r.t.\@ $\Dia$ if and only if it has a left adjoint. 
\end{SATZ}

\begin{SATZ}[left]Let $\Dia$ be an infinite diagram category (cf.\@ \ref{DEFDIAGRAMCAT}).
Let $\DD \rightarrow \SSS$ be an infinite left fibered (multi-)derivator with domain $\Dia$ whose {\em fibers} $\DD_S$ for every $I \in \Dia$ and all $S \in \SSS(I)$ are stable (left and right) derivators with domain $\Posf$. Assume that $\DD$ has perfectly generated fibers.
Then $\DD$ is a {\em right} fibered (multi-)derivator, too.
\end{SATZ}

\begin{proof}
Let $I \in \Dia$ and let $f \in \Hom_{\SSS(I)}(S_1, \dots, S_n; T)$ be a multimorphism. By Lemma~\ref{LEMMAMORPHDERLEFT}, fixing $\mathcal{E}_1, \overset{\widehat{i}}{\dots}, \mathcal{E}_n$, the association
\begin{eqnarray*} 
 \DD(I \times J)_{p^*S_i} &\rightarrow& \DD(I)_{p^*T} \\
\mathcal{E}_i &\mapsto& (p^*f)_\bullet(p^* \mathcal{E}_1, \dots,\ \mathcal{E}_i,\  \dots, p^* \mathcal{E}_n)
\end{eqnarray*}
defines a morphism of derivators
\[ \DD_{S_i} \rightarrow \DD_{T} \]
which is left continuous. Hence by Theorem~\ref{SATZLEFTBROWNDER} it has a right adjoint. This shows the first part of (FDer0 right), i.e.\@ the functor $\DD(I) \rightarrow \SSS(I)$ is an opfibration, too, for
every $I \in \Dia$. 
Then axiom (FDer5 left) implies the remaining assertion of (FDer0 right) while (FDer0 left) implies (FDer5 right), see Lemma~\ref{LEMMALEFTRIGHT}.

Similarly a morphism $\alpha: I \rightarrow J$ in $\Dia$ induces a morphism of derivators
\[ \alpha^* : \DD_S \rightarrow \DD_{\alpha^*S}. \]
It commutes with homotopy colimits by Proposition~\ref{PROPHOMCART},~2.
Therefore $\alpha^*$ has a right adjoint $\alpha_*$ by the previous theorem, i.e.\@ (FDer3 right) holds.
(FDer4 right) is then a consequence of Lemma~\ref{PROPHOMCART}, 1. 
\end{proof}

Analogously, using Theorem~\ref{THEOREMTRIANGREPLEFT} of left Brown representability, we obtain:

\begin{SATZ}[right]
Let $\Dia$ be an infinite diagram category (cf.\@ \ref{DEFDIAGRAMCAT}).
Let $\DD \rightarrow \SSS$ be an infinite right fibered (multi-)derivator with domain $\Dia$, whose {\em fibers} $\DD_S$ for every $I \in \Dia$ and for all $S \in \SSS(I)$ are stable (left and right) derivators with domain $\Posf$. Assume that $\DD$ has compactly generated fibers.
Then $\DD$ is a {\em left} fibered (multi-)derivator, too.
\end{SATZ}

\subsection{(Co)Cartesian projectors}

\begin{SATZ}[right]\label{THEOREMEXISTRIGHTPROJ}
 Let $\DD \rightarrow \SSS$ be an infinite fibered left derivator (w.r.t.\@ $\Dia$) whose fibers are stable derivators w.r.t.\@ $\Posf$.
Assume that $\DD(\cdot)_S$ is well-generated for every $S \in \SSS(\cdot)$.  
Then for all $I \in \Dia$, for all $F \in \SSS(I)$, and for all functors $I \rightarrow E$ in $\Dia$ the fully-faithful inclusion 
\[ \DD(I)^{E-\mathrm{cocart}}_F \rightarrow \DD(I)_F \]
has a right adjoint $\Box_*^E$.  

If $\DD \rightarrow \SSS$ also satisfies (FDer0 right) and
if $F$ is such that $F(\mu)$ satisfies (Dloc2 left) for every $\mu$ mapping to an identity in $E$, then the fully-faithful inclusion
\[ \DD(I)^{E-\mathrm{cart}}_F \rightarrow \DD(I)_F \]
has a right adjoint $\blacksquare_*^E$.
\end{SATZ}
\begin{proof}
Consider the set $O$ of morphisms $\mu: i \rightarrow j$ which map to an identity in $E$, and
for each morphism $\mu \in O$
the composition $D_\mu$:
\[  \xymatrix{  \DD(I)_{F} \ar[r]^-{\mu^*} &  \DD(\rightarrow)_{\mu^*F } \ar[r]^-{F(\mu)_\bullet} & \DD(\rightarrow)_{i^*F} \ar[r]^-{\Cone} & \DD(\cdot)_{i^*F}} \] 
We define a functor $D$ by
\[ \prod_{\mu \in O} D_\mu: \DD(I)_{F} \rightarrow  \prod_{\mu \in O} \DD(\cdot)_{i^*F}  = \DD(O)_{\iota^*F}, \] 
where $\iota: O \rightarrow I$ is the map ``source''.
$D$ commutes with coproducts, as all functors in the succession do, and it is exact. Therefore by \cite[Theorem 7.4.1]{Kra09} the 
triangulated subcategory $\DD(I)^{E-\mathrm{cocart}}_F = \ker D$
is well-generated and hence the inclusion in the statement has a right adjoint. 
In the Cartesian case, $F(\mu)^\bullet$ commutes with coproducts only if $F(\mu)$ satisfies (Dloc2 left). 
\end{proof}

\begin{SATZ}[left]\label{THEOREMEXISTLEFTPROJ}
Let $\DD \rightarrow \SSS$ be an infinite fibered derivator (with domain $\Dia$) whose fibers are stable.
Assume that all $\DD(\cdot)_S$ for $S \in \SSS(\cdot)$ are compactly generated.  

Let $I \rightarrow E$ be a functor in $\Dia$ and let $F \in \SSS(I)$. Suppose that $F(\mu)$ satisfies (Dloc2 left) for every morphism $\mu$ in $I$ that maps to an identity in $E$.  
Then the fully-faithful inclusion 
\[ \DD(I)^{E-\mathrm{cart}}_F \rightarrow \DD(I)_F \]
has a left adjoint $\Box_!^E$. 
\end{SATZ}
\begin{proof}
As in the proof of the previous theorem we have an exact functor 
\[ F^{\cart}:  \DD(I)_F \rightarrow \mathcal{T} \]
into another triangulated category which commutes with coproducts and such that the subcategory of $E$-Cartesian objects is precisely its kernel. 
Lemma \ref{LEMMAGENERATION} implies that $\DD(I)_F$ is compactly generated, and hence $\DD(I)_F^{\op}$ is perfectly generated. Furthermore,
Theorem~\ref{THEOREMEXISTRIGHTPROJ} implies that the categories $\DD(I)_F/\DD(I)_F^{E-\cart}$ are locally small. 
Note that 
\[ \DD(I)_F^{\op}/(\DD(I)_F^{E-\cart})^{\op} = (\DD(I)_F/\DD(I)_F^{E-\cart})^{\op}. \]
Therefore \cite[Proposition 5.2.1]{Kra09} implies that a right adjoint to the inclusion $(\DD(I)_F^{E-\cart})^{\op} \rightarrow (\DD(I)_F)^{\op}$ exists. 
So a left adjoint to the inclusion
\[ \DD(I)_F^{E-\cart} \rightarrow \DD(I)_F \]
exists. 
\end{proof}

In the compactly generated case, $\Box_!^E$ should exist unconditionally, but we were not able to prove this.

\section{Constructions}\label{SECTCONSTR}

\subsection{The fibered multiderivator associated with a fibered multicategory}

\begin{PAR}The most basic situation in which a (non-trivial) fibered (multi-)derivator can be constructed arises from a bifibration of (locally small) multicategories
\[ p: \mathcal{D} \rightarrow  \mathcal{S} \]
where we are given a set of weak equivalences $\mathcal{W}_S \subset \Mor(\mathcal{D}_S)$ for each object $S$ of $\mathcal{S}$.
In the examples we have in mind, the objects of $\mathcal{S}$ are spaces (or schemes), the objects of $\mathcal{D}$ are chain complexes of sheaves (coherent, etale Abelian, etc.) on them, and
the morphisms in $\mathcal{W}_S$ are the quasi-isomorphisms. In these examples the multicategory-structure arises from the tensor product and it is even,
in most cases, the more natural structure because no particular tensor-product is chosen a priori. 
\end{PAR}

\begin{DEF}\label{DEFFIBDERMODEL}
In the situation above, 
let $\SSS$ be the pre-multiderivator associated with the multicategory $\mathcal{S}$. We define a pre-multiderivator
$\DD$ as follows (cf.\@ \ref{PROPMULTILOC} for localizations of multicategories): 
\[ \DD(I) = \Hom(I, \mathcal{D})[\mathcal{W}_I^{-1}] \]
where $\mathcal{W}_I$ is the set of natural transformations which are element-wise in the union $\bigcup_S \mathcal{W}_S$.
The functor $p$ obviously induces a morphism of pre-multiderivators
\[ \widetilde{p}: \DD \rightarrow \SSS \]
Observe that morphisms in $\mathcal{W}_I$, by definition, necessarily map to identities in $\Hom(I, \mathcal{S})$.
\end{DEF}

In this section we prove that the above morphism of pre-(multi-)derivators is a left (resp.\@ right) fibered (multi-)derivator on directed (resp.\@ inverse) diagrams, provided that the fibers
are model categories whose structure is compatible with the structure of bifibration. We use the definition of a model category from \cite{Hov99}.
We denote the cofibrant replacement functor by $Q$ and the fibrant replacement functor by $R$.

\begin{DEF}\label{PARQUILLENN}
A {\bf bifibration of (multi\nobreakdash-)model-categories}
is a bifibration of (multi-)categories $p: \mathcal{D} \rightarrow \mathcal{S}$ 
together with the collection of a closed model structure on the fiber
\[ (\mathcal{D}_S, \Cof_S, \Fib_S, \mathcal{W}_S) \]
for any object $S$ in $\mathcal{S}$
 such that the following two properties hold:
\begin{enumerate}
\item For any $n\ge 1$ and for every multimorphism 
\[ \xymatrix{
S_1 \ar@{-}[rd] \\
\ar@{}[r]|{\vdots} & f \ar[r] &  T  \\
S_n \ar@{-}[ru] \\
}\]
the push-forward $f_\bullet$ and the various pull-backs $f^{\bullet,j}$ define a Quillen
adjunction in $n$-variables
\[ \xymatrix{ \prod_i (\mathcal{D}_{S_i}, \Cof_{S_i}, \Fib_{S_i}, \mathcal{W}_{S_i}) \ar[rr]^-{f_\bullet} && (\mathcal{D}_T, \Cof_T, \Fib_T, \mathcal{W}_T) } \]
\[ \xymatrix{(\mathcal{D}_T, \Cof_T, \Fib_T, \mathcal{W}_T) \times  \prod_{i \not= j} (\mathcal{D}_{S_i}, \Cof_{S_i}, \Fib_{S_i}, \mathcal{W}_{S_i}) \ar[rr]^-{f^{\bullet,j}} && (\mathcal{D}_{S_j}, \Cof_{S_j}, \Fib_{S_j}, \mathcal{W}_{S_j}) } \]

\item  For any 0-ary morphism $f$  in $\mathcal{S}$, let $f_\bullet()$ be the corresponding unit object (i.e.\@ the object representing the 0-ary morphism functor $\Hom_f(;-)$) and consider the cofibrant replacement $Qf_\bullet() \rightarrow f_\bullet()$. Then the natural morphism
\[ F_\bullet(X_1, \dots, X_{i-1}, Qf_\bullet(), X_i,  \dots, X_n) \rightarrow F_\bullet(X_1, \dots, X_{i-1}, f_\bullet(), X_i, \dots, X_n)  \cong (F \circ_i f)_\bullet(X_1, \dots, X_n)    \]
is a weak equivalence if all $X_i$ are cofibrant. Here $F$ is any morphism which is composable with $f$. 
\end{enumerate}
\end{DEF}

\begin{BEM}
If $\mathcal{S}=\{\cdot\}$, the above notion coincides with the notion of monoidal model-category in the sense of \cite[Definition 4.2.6]{Hov99}.
In this case it is enough to claim property 1. for $n=1, 2$. 
\end{BEM}

\begin{SATZ}\label{SATZEXISTENCEFMULTIDER}
Under the conditions of Definition~\ref{PARQUILLENN} the morphism of pre-derivators 
\[ \widetilde{p}: \DD \rightarrow \SSS \]
(defined in \ref{DEFFIBDERMODEL}) is a left fibered multiderivator (satisfying also FDer0 right) with domain $\Dir$ and a right fibered multiderivator (satisfying also FDer0 left)  with domain $\Inv$.
Furthermore for all $S \in \SSS(\cdot)$ its fiber $\DD_S$ (c.f.\@ \ref{DEFFDERFIBERS}) is just the pre-derivator associated with the pair $(\mathcal{D}_S, \mathcal{W}_S)$. 
\end{SATZ}

There are techniques by Cisinski \cite{Cis03} which allow to extend a derivator to more general diagram categories. We will explain in a forthcoming article
how these can be applied to fibered (multi\nobreakdash-)derivators.  

The proof of the theorem will occupy the rest of this section.
First we have:
\begin{PROP}\label{PARBIFIBI}
Let $\mathcal{D} \rightarrow \mathcal{S}$ be a bifibration of multicategories with complete fibers. 
For any diagram category $I$, the functors
\[ p_I: \Hom(I, \mathcal{D}) \rightarrow \Hom(I, \mathcal{S}) = \SSS(I) \]
are bifibrations of multicategories. 

Morphisms in $\Hom(I, \mathcal{D})$ are Cartesian, if and only if they are point-wise Cartesian. 
The 1-ary morphisms in $\Hom(I, \mathcal{D})$ are coCartesian, if and only if they are point-wise coCartesian. 
\end{PROP}
\begin{proof}[Proof (Sketch).]
We choose push-forward functors $f_\bullet$ and pull-back functors $f^{i,\bullet}$ for $\mathcal{D} \rightarrow \mathcal{S}$ as usual. 
Let $f \in \Hom(S_1, \dots, S_n; T)$ be a morphism in $\Hom(I, \mathcal{S})$.
We define a functor
\[ f_\bullet: \Hom(I, \mathcal{D})_{S_1} \times \cdots \times \Hom(I, \mathcal{D})_{S_n} \rightarrow \Hom(I, \mathcal{D})_T \]
by
\[ \mathcal{E}_1, \dots, \mathcal{E}_n \mapsto \{i \mapsto (f_i)_\bullet( \mathcal{E}_1(i), \dots,  \mathcal{E}_n(i)) \}.  \]
Note that a morphism $\alpha: i \rightarrow i'$ in $I$ induces a {\em well-defined} morphism 
\[ (f_i)_\bullet( \mathcal{E}_1(i), \dots,  \mathcal{E}_n(i)) \rightarrow (f_{i'})_\bullet( \mathcal{E}_1(i'), \dots,  \mathcal{E}_n(i'))\]
lying over $T(\alpha)$. 
The functor $f_\bullet$ comes equipped with a morphism in
\[   \Hom( \mathcal{E}_1, \dots, \mathcal{E}_n; f_\bullet (\mathcal{E}_1, \dots, \mathcal{E}_n)) \]
which is checked to be Cartesian in the strong form of Definition~\ref{DEFCARTCOCART}. 

For 1-ary morphisms we can perform the same construction to produce coCartesian morphisms. For $n\ge 2$ the construction is more complicated. 
Let $f \in \Hom(S_1, \dots, S_n; T)$ be a morphism with $n\ge 2$. 
To ease notation, we construct a pull-back functor w.r.t.\@ first slot. The other constructions are completely symmetric. 

For any $i_1 \in I$ consider the category (a variant of the twisted arrow category)
\[ X_{i_1}(I):=\{\ (i_2, \dots, i_n, j, \{\alpha_k\}_{k=1..n})\ |\ \alpha_k: i_k \rightarrow j\}  \]
which is covariant in $j$ and contravariant in $i_2, \dots, i_n$. For any $\beta: i_1 \rightarrow i'_1$ we have an induced functor $\widetilde{\beta}: X_{i_1}(I) \rightarrow X_{i_1'}(I)$.
Any object in $X_{i_1}(I)$ defines by pre-composition with $S_k(\alpha_k)$ for all $1 \le k \le n$ a morphism
$f_\alpha \in \Hom( S_1(i_1), \dots, S_n(i_n); T(j))$.  

We define a functor 
\[ f^{1,\bullet}: (\Hom(I, \mathcal{D})_{S_2})^{\op} \times \cdots \times  (\Hom(I, \mathcal{D})_{S_n})^{\op} \times \Hom(I, \mathcal{D})_{T} \rightarrow \Hom(I, \mathcal{D})_{S_1} \]
assigning to $\mathcal{E}_2, \dots, \mathcal{E}_n; \mathcal{F}$ the following functor $X_{i_1}(I) \rightarrow \mathcal{D}_{S_1(i_1)}$:
\[ \alpha \mapsto (f_\alpha)^{1,\bullet}( \mathcal{E}_2(i_2), \dots,   \mathcal{E}_n(i_n); \mathcal{F}(j)) \]
and then taking $\lim_{X_{i_1}(I)}$ which exists because the fibers are required to be complete. For the functoriality note that for $\beta: i_1 \rightarrow i_1'$ we have a natural morphism
\[ \lim_{X_{i_1}(I)} \cdots  \rightarrow  \lim_{X_{i_1'}(I)} \cdots \]
induced by $\widetilde{\beta}$.

We define a morphism
\[ \Xi \in \Hom_f(f^{1,\bullet}(\mathcal{E}_2, \dots, \mathcal{E}_n; \mathcal{F}), \mathcal{E}_2, \dots, \mathcal{E}_n; \mathcal{F}) \]
and we will show that it is coCartesian w.r.t.\@ the first slot in a weak sense. 
At some object $i \in I$, the morphism $\Xi$ is given by composing the projections from $\lim_{X_{i=i_1}(I)} f_\alpha^{1,\bullet}( \mathcal{E}_2(i_2), \dots,   \mathcal{E}_n(i_n); \mathcal{F}(j))$ to 
$f_i^{1, \bullet}( \mathcal{E}_2(i), \dots,   \mathcal{E}_n(i); \mathcal{F}(i) )$ (note that $f_i = f_\alpha$ for $\alpha=\{\id_i\}_k$) and then composing with the coCartesian morphism (in $\mathcal{D}$) in 
\[ \Hom(f_{i}^{1,\bullet}(\mathcal{E}_2(i), \dots, \mathcal{E}_n(i); \mathcal{F}(i)), \mathcal{E}_2(i), \dots, \mathcal{E}_n(i); \mathcal{F}(i)). \]
One checks that the so defined $\Xi$ is functorial in $i$.  It remains to be shown that the composition with $\Xi$ induces an isomorphism
\begin{equation}\label{isococart}
 \Hom_{\id_{S_1}}(\mathcal{E}_1; f^{1, \bullet}(\mathcal{E}_2, \dots, \mathcal{E}_n; \mathcal{F})) \rightarrow \Hom_f(\mathcal{E}_1, \dots, \mathcal{E}_n; \mathcal{F}).
 \end{equation}
We will give a map in the other direction which is inverse to composition with $\Xi$. 
Let 
\[ a \in \Hom_f(\mathcal{E}_1, \dots, \mathcal{E}_n; \mathcal{F}) \] 
be a morphism. To give a morphism on the left hand side of (\ref{isococart}), for any $i_1$ and $\alpha \in X_{i_1}(I)$ we have to give a
morphism (functorial in $i_1$)
\[ \mathcal{E}_1(i_1) \rightarrow f_\alpha^{1, \bullet}(\mathcal{E}_2(i_2); \dots, \mathcal{E}_n(i_n); \mathcal{F}(j))  \]
or, which is the same, a morphism
\[ \Hom_{f_\alpha}(\mathcal{E}_1(i_1), \mathcal{E}_2(i_2), \dots, \mathcal{E}_n(i_n); \mathcal{F}(j)). \]
But we have such a morphism, namely the pre-composition of $a_j$ with the $n$-tuple $\{\mathcal{E}_k(\alpha_k)\}_k$.
(Because we know already that $\Hom(I, \mathcal{D}) \rightarrow \Hom(I, \mathcal{S})$ is an {\em op}fibration of multicategories, it suffices to 
establish that $\Xi$ is coCartesian in this weak form.)
\end{proof}

\begin{BEM}
The construction in the proof of the above Proposition will become much clearer, when we define a fibered multiderivator {\em itself} as
a six-functor-formalism, similar to the definition mentioned in the introduction. For example, for $\mathcal{S}=\{\cdot\}$ we will get an external and internal monoidal product, resp.\@ right adjoints which a clear relation. 
We have in that case
\[ \boxtimes: \Hom(I, \mathcal{D}) \times \Hom(J, \mathcal{D}) \rightarrow \Hom(I \times J, \mathcal{D}) \]
by applying $\otimes$ {\em point-wise} and 
\[ \mathcal{HOM}_{l/r}: \Hom(I, \mathcal{D}) \times \Hom(J, \mathcal{D}) \rightarrow \Hom(I^{\op} \times J, \mathcal{D}) \]
by applying $\Hom_{l/r}$ {\em point-wise}. 
The formula for the internal hom obtained in the proof of the proposition boils down to the formula
\[ \Hom_{l/r}(\mathcal{E}, \mathcal{F})(i_1) = \int_{i} \mathcal{HOM}_{l/r}(\mathcal{E}(i), \mathcal{F}(i))^{\Hom(i_1, i)} \]
where $\int_{i}$ is the categorical end. 
We refer to a subsequent article \cite{Hor15} for an explanation of this in the language of six-functor-formalisms. 
\end{BEM}

We will need later the following
\begin{LEMMA}\label{LEMMABIFIBI}
Let $f \in \Hom(S_1, \dots, S_n; T)$ be a morphism in $\Hom(I, \mathcal{S})$ for some $n\ge 2$. 
Consider the pull-back functor $f^{j, \bullet}$  constructed in the proof of Proposition~\ref{PARBIFIBI}. Let $p: I \times J \rightarrow I$ be the projection and fix objects
$\mathcal{E}_1, \mathop{\dots}\limits^{\widehat{j}}, \mathcal{E}_n, \mathcal{F}$  in $\mathcal{D}$ lying over $S_1, \mathop{\dots}\limits^{\widehat{j}}, S_n, T$. Then the natural morphism
\[ p^* f^{j, \bullet}(\mathcal{E}_1,  \mathop{\dots}\limits^{\widehat{j}},  \mathcal{E}_n; \mathcal{F}) \rightarrow (p^* f)^{j, \bullet}(p^*\mathcal{E}_1,  \mathop{\dots}\limits^{\widehat{j}},  p^*\mathcal{E}_n; p^*\mathcal{F})  \]
is an isomorphism, or, in other words, the functor $p^*: \Hom(I, \mathcal{D}) \rightarrow  \Hom(I \times J, \mathcal{S})$ maps Cartesian arrows to Cartesian arrows. 
\end{LEMMA}
\begin{proof}
Again, we assume $j=1$ to ease the notation. The statement concerning the other pull-backs is completely symmetric. 
We have by definition
\[ (f^{1, \bullet}(\mathcal{E}_2, \dots, \mathcal{E}_n; \mathcal{F}))(i') = \lim_{\alpha \in X_{i_1}(I)} f_\alpha^{1, \bullet} (\mathcal{E}_2(i_1), \dots, \mathcal{E}_n(i_n); \mathcal{F}(i'))  \]
and
\[ ((p^*f)^{1, \bullet}(\mathcal{E}_2, \dots, \mathcal{E}_n; \mathcal{F}))(i',j') = \lim_{\alpha \in X_{i_1,j_1}(I \times J)} (p^*f)_\alpha^{1, \bullet} ((p^*\mathcal{E}_2)(i_1,j_1), \dots, (p^*\mathcal{E}_n)(i_n,j_n); (p^*\mathcal{F})(i',j'))  \]
The natural map in question is induced by the functor $\widetilde{p}: X_{i_1,j_1}(I \times J) \rightarrow X_{i_1}(I)$ which forgets all data involving the $J$ direction. Now there is also a functor $\widetilde{s}: X_{i_1}(I) \rightarrow X_{i_1,j_1}(I \times J)$ which is constant on the $J$-component with value $\{\id_{j_1}\}_{k=1..n}$. We have $\widetilde{p} \circ \widetilde{s} = \id$ and a chain of natural transformations $\widetilde{s}\circ\widetilde{p} \Leftarrow \cdots \Rightarrow \id$ involving only
data in the $J$-direction. However, all the
natural transformations are mapped to identities by the functor 
\[ \alpha \mapsto \lim_{\alpha \in X_{i_1,j_1}(I \times J)} (p^*f)_\alpha^{1, \bullet} ((p^*\mathcal{E}_2)(i_1,j_1), \dots, (p^*\mathcal{E}_n)(i_n,j_n); (p^*\mathcal{F})(i',j')) \] 
because everything is constant along the $J$-direction. This shows that the natural morphism in the statement is an isomorphism. 
\end{proof}

If $I$ is directed or inverse we want to show that also 
$p_I$ is a bifibration of multi-model-categories in the sense of Definition~\ref{PARQUILLENN}.

Afterwards we will apply the following variant and generalization to multicategories of the results in \cite[Expos\'e XVII, \S 2.4]{SGAIV3}.
\begin{PROP}\label{PROPDELIGNE}
Let $p: \mathcal{D} \rightarrow \mathcal{S}$ be a bifibration of (multi\nobreakdash-)model-categories in the sense of \ref{PARQUILLENN}.
Let $\mathcal{W}$ be the union of the $\mathcal{W}_S$ over all objects $S \in \mathcal{S}$. 
Then the fibers of $\widetilde{p}: \mathcal{D}[\mathcal{W}^{-1}] \rightarrow \mathcal{S}$ (as ordinary categories) are the homotopy categories $\mathcal{D}_S[\mathcal{W}_S^{-1}]$
and $\widetilde{p}$ is again a bifibration of multicategories such that the push-forward $F_\bullet$ along any $F \in \Hom_{\mathcal{S}}(X_1, \dots, X_n; Y)$
(for $n\ge 1$) is the left derived functor of the corresponding push-forward w.r.t. $p$. Similarly the pull-back w.r.t.\@ some slot ist the right derived functor of the corresponding pull-back w.r.t. $p$. 
\end{PROP} 

\begin{PAR}
The above proposition and its proof have several well-known consequences which we mention, despite being all elementary, because the proof below gives a
uniform treatment of all the cases.

\begin{enumerate}
\item The homotopy category of a model category is locally small and can be described as the category of cofibrant/fibrant objects modulo homotopy of morphisms. 
{\em Apply the proof of the proposition to the (trivial) bifibration of ordinary categories $\mathcal{D} \rightarrow \{\cdot\}$}.
\item Quillen adjunctions lead to an adjunction of the derived functors on the homotopy categories. {\em Apply the proposition to a bifibration of ordinary categories
  $\mathcal{D} \rightarrow \Delta_1$}.
\item The homotopy category of a closed monoidal model category is a closed monoidal category. 
{\em Apply the proposition to a  bifibration of multicategories $\mathcal{D} \rightarrow \{\cdot\}$}.
\item Quillen adjunctions in $n$ variables lead to an adjunction in $n$ variables on the homotopy categories. {\em Apply the proposition to a bifibration of multicategories
  $\mathcal{D} \rightarrow \Delta_{1,n}$, where the multicategory $\Delta_{1,n}$ consists of $n+1$ objects and one $n$-ary morphism connecting them.}
\end{enumerate}

Before proving Proposition~\ref{PROPDELIGNE}, we define homotopy relations on $\Hom_F(\mathcal{E}_1, \dots, \mathcal{E}_n; \mathcal{F})$ where $F \in \Hom(X_1, \dots, X_n; Y)$ is a multimorphism in $\mathcal{S}$.
\end{PAR}

\begin{DEF}
\begin{enumerate}
\item 
Two morphisms $f$ and $g$ in $\Hom_F(\mathcal{E}_1, \dots, \mathcal{E}_n; \mathcal{F})$ are called {\bf right homotopic} if there is a path object of $\mathcal{F}$
\[\xymatrix{ \mathcal{F} \ar[r] &  \mathcal{F}' \ar@<+2pt>[r]^{\pr_1} \ar@<-2pt>[r]_{\pr_2} & \mathcal{F} } \] 
and a morphism $\Hom(\mathcal{E}_1, \dots, \mathcal{E}_n; \mathcal{F}')$ over the same multimorphism $F$ such that the compositions with $\pr_1$ and $\pr_2$ are $f$ and $g$, respectively.

\item 
For $n\ge 1$, two morphisms $f$ and $g$ in $\Hom_F(\mathcal{E}_1, \dots, \mathcal{E}_n; \mathcal{F})$  are called {\bf $i$-left homotopic} if there is a cylinder object $\mathcal{E}_i'$ of $\mathcal{E}_i$
\[\xymatrix{ \mathcal{E}_i \ar@<+2pt>[r]^{\iota_{1}} \ar@<-2pt>[r]_{\iota_{2}} &  \mathcal{E}'_i  \ar[r] & \mathcal{E}_i } \] 
and a morphism $\Hom(\mathcal{E}_1, \dots, \mathcal{E}_i', \dots, \mathcal{E}_n; \mathcal{F})$ over $F$ such that the compositions with $\iota_{1}$ and $\iota_{2}$ are $f$ and $g$, respectively.
\end{enumerate}

\end{DEF}

\begin{LEMMA}\label{LEMMAHC0}
\begin{enumerate}
\item 
The condition `right homotopic' is preserved under pre-composition, while the condition `$i$-left homotopic' is preserved under post-composition. 
\item Let $n \ge 1$. 
If $f, g \in \Hom(\mathcal{E}_1, \dots, \mathcal{E}_n; \mathcal{F})$ are $i$-left homotopic and all $\mathcal{E}_i$ are cofibrant then $f$ and $g$ are  right homotopic.
If $f, g \in \Hom(\mathcal{E}_1, \dots, \mathcal{E}_n; \mathcal{F})$ are right homotopic, $\mathcal{F}$ is fibrant, and all $\mathcal{E}_j$ for $j\not=i$ are cofibrant then $f$ and $g$ are $i$-left homotopic.

\item Let $n \ge 1$. 
In $\Hom(\mathcal{E}_1, \dots, \mathcal{E}_n; \mathcal{F})$ right homotopy is an equivalence relation if all $\mathcal{E}_i$ are cofibrant.
In $\Hom(\mathcal{E}_1, \dots, \mathcal{E}_n; \mathcal{F})$ $i$-left homotopy is an equivalence relation if $\mathcal{F}$ is fibrant, and all $\mathcal{E}_j$, $j\not=i$ are cofibrant
\end{enumerate}
\end{LEMMA}

In particular on the category $\mathcal{D}^{\Cof,\Fib}$ of fibrant/cofibrant objects, $i$-left homotopy=right homotopy is an equivalence relation, 
which is compatible with composition. 

\begin{proof}
1. is obvious.

2. If all $\mathcal{E}_i$ are cofibrant then also $F_\bullet(\mathcal{E}_1, \dots, \mathcal{E}_n)$ is cofibrant and $f$ and $g$ correspond uniquely to morphisms $f', g': F_\bullet(\mathcal{E}_1, \dots, \mathcal{E}_n) \rightarrow \mathcal{F}$. Since $f$ and $g$ are $i$-left homotopic, there is a cylinder object 
\[ \xymatrix{ \mathcal{E}_i \ar@<2pt>[r] \ar@<-2pt>[r] &  \mathcal{E}_i' \ar[r] & \mathcal{E}_i } \]
realizing the $i$-left homotopy. Since $\mathcal{E}_i$ is cofibrant so is $\mathcal{E}_i'$.
Hence also 
\[ \xymatrix{ F_\bullet(\mathcal{E}_1, \dots, \mathcal{E}_n) \ar@<2pt>[r] \ar@<-2pt>[r] & F_\bullet(\mathcal{E}_1, \dots, \mathcal{E}_i', \dots, \mathcal{E}_n) \ar[r] & F_\bullet(\mathcal{E}_1, \dots, \mathcal{E}_n) } \]
is a cylinder object because all $\mathcal{E}_j$ are cofibrant, and hence also $f'$ and $g'$ are left homotopic. These are therefore also right homotopic and hence so are $f$ and $g$.
Dually we obtain the second statement. 

3. follows from \cite[Proposition 1.2.5, (iii)]{Hov99}.
\end{proof}

\begin{LEMMA}\label{LEMMAHC1}
Two $i$-left homotopic morphisms become equal in $\mathcal{D}^{\Cof}[(\mathcal{W}^{\Cof})^{-1}]$. 
\end{LEMMA}
\begin{proof}
This follows from the fact that a cylinder object 
\[ \xymatrix{ \mathcal{E}_i \ar@<2pt>[r]^{\iota_1} \ar@<-2pt>[r]_{\iota_2} &  \mathcal{E}_i' \ar[r]^p & \mathcal{E}_i } \]
automatically lies in $\mathcal{D}^{\Cof}$ if $\mathcal{E}_i$ does, and 
the two maps $\iota_1$ and $\iota_2$ become equal because $p$ becomes invertible. 
\end{proof}

We have to distinguish the easier case, in which all objects $F_\bullet()$ for 0-ary morphisms $F$ are cofibrant. 
Otherwise we define
a category $\widetilde{\mathcal{D}^{\Cof}[(\mathcal{W}^{\Cof})^{-1}]}$ in which we set 
$\Hom_{F}(;\mathcal{F}):=\Hom_{\mathcal{D}_S[\mathcal{W}_S^{-1}]}(QF_\bullet(); \mathcal{F})$ for all $\mathcal{F}$, where $F$ is a 0-ary morphism with domain $S$. 
Composition is given as follows: 
For a morphism $f \in \Hom_G(\mathcal{E}_1, \dots, \mathcal{E}_n; \mathcal{F})$ with cofibrant $\mathcal{E}_i$ and $\mathcal{F}$ and $\xi : QF_\bullet() \rightarrow \mathcal{E}_i$, we define the composition 
$\xi \circ f$ as the following composition

\[ \xymatrix{
\mathcal{E}_1 \ar@{-}[rd] \\
\ar@{}[r] \vdots^{\widehat{i}} & \cocart \ar[r] & (F\circ G)_\bullet (\mathcal{E}_2, \overset{\widehat{i}}{\dots}, \mathcal{E}_n) \ar[r]^-\sim &  G_\bullet (\mathcal{E}_1, \dots, F_\bullet(), \dots, \mathcal{E}_n) & \ar[l] \\
\mathcal{E}_n \ar@{-}[ru] \\
} \]
\[  \xymatrix{
& \ar[l] G_\bullet (\mathcal{E}_1, \dots, QF_\bullet(), \dots, \mathcal{E}_n) \ar[r] &G_\bullet (\mathcal{E}_1, \dots,  \mathcal{E}_n) \ar[r] &  \mathcal{F}.   \\
} \]

One checks that the so-defined composition is associative and independent of the choice of the push-forwards.

\begin{LEMMA}\label{LEMMAHC2}
If the object $F_\bullet()$ is cofibrant for every 0-ary morphism  $F$ then the natural functor
\[ \mathcal{D}^{\Cof}[(\mathcal{W}^{\Cof})^{-1}] \rightarrow \mathcal{D}[\mathcal{W}^{-1}] \]
is an equivalence of categories.
Otherwise it is, if we replace $\mathcal{D}^{\Cof}[(\mathcal{W}^{\Cof})^{-1}]$ by $\widetilde{\mathcal{D}^{\Cof}[(\mathcal{W}^{\Cof})^{-1}]}$.
\end{LEMMA}
\begin{proof}
The inclusion $\mathcal{D}^{\Cof} \rightarrow \mathcal{D}$ induces a functor $\Xi: \mathcal{D}^{\Cof}[(\mathcal{W}^{\Cof})^{-1}] \rightarrow \mathcal{D}[\mathcal{W}^{-1}]$. If the objects $F_\bullet()$ are not cofibrant then $\Xi$ may be modified to a functor 
\[\widetilde{\mathcal{D}^{\Cof}[(\mathcal{W}^{\Cof})^{-1}]} \rightarrow \mathcal{D}[\mathcal{W}^{-1}] \]
as follows: a 0-ary morphism $QF_\bullet() \rightarrow \mathcal{F}$ is mapped to the composition 
\[ \xymatrix{ \ar@{o->}[rr]^-{\cocart} &&  F_\bullet() & \ar[l] QF_\bullet() \ar[r] & \mathcal{F}} \] in $\mathcal{D}[\mathcal{W}^{-1}]$. 

We now specify a functor $\Phi$ in the other direction. $\Phi$ maps an object $X$ to a cofibrant replacement $QX$. 
For $n\ge 1$, a morphism $f \in \Hom(\mathcal{E}_1, \dots, \mathcal{E}_n; \mathcal{F})$ over $F$
is mapped to the following morphism. Composing with the morphisms $Q\mathcal{E}_i \rightarrow \mathcal{E}_i$, 
we get a morphism $f' \in \Hom(Q\mathcal{E}_1, \dots, Q\mathcal{E}_n; \mathcal{F})$ or equivalently a morphism
$X_i \rightarrow F^{\bullet, i}(Q\mathcal{E}_1, \overset{\widehat{i}}{\dots}, Q\mathcal{E}_n; \mathcal{F})$.
Now choose a lift (dotted arrow in the diagram)
\[ \xymatrix{
& F^{\bullet, i}(Q\mathcal{E}_1, \overset{\widehat{i}}{\dots}, Q\mathcal{E}_n; Q\mathcal{F}) \ar[d]   \\
 Q\mathcal{E}_i \ar[r] \ar@{.>}[ru] &  F^{\bullet, i}(Q\mathcal{E}_1, \overset{\widehat{i}}{\dots}, Q\mathcal{E}_n; \mathcal{F})  \\
} \]
which exists because the vertical map is again a trivial fibration (because all the $Q\mathcal{E}_i$ are cofibrant).
The resulting map in $\Hom(Q\mathcal{E}_1, \dots, Q\mathcal{E}_n; P\mathcal{F})$ is actually well-defined in 
$\mathcal{D}^{\Cof}[(\mathcal{W}^{\Cof})^{-1}]$. Indeed, two different lifts are left homotopic because $Q\mathcal{E}_i$ is cofibrant \cite[Proposition 1.2.5. (iv)]{Hov99}, and
therefore also the two morphisms in $\Hom(Q\mathcal{E}_1, \dots, Q\mathcal{E}_n; Q\mathcal{F})$ become equal in $\mathcal{D}^{\Cof}[(\mathcal{W}^{\Cof})^{-1}]$ by Lemma~\ref{LEMMAHC1}. 
From this it follows that $\Phi$ is indeed a functor on $n$-ary morphisms for $n\ge 1$.

For $n=0$, a morphism $f \in \Hom(; \mathcal{F})$ over $F$ 
corresponds to a morphism $F_\bullet() \rightarrow \mathcal{F}$. If $F_\bullet()$ is cofibrant, this morphism lifts (again uniquely up to right homotopy) to a morphism 
$F_\bullet() \rightarrow Q\mathcal{F}$, i.e.\@ to a morphism in $\Hom_F(; Q\mathcal{F})$.  
If $F_\bullet()$ is not cofibrant then the composition lifts to a morphism: $Q F_\bullet() \rightarrow P\mathcal{F}$ which is defined to be the image of $\Phi$. 
Furthermore $\Phi$ is inverse to $\Xi$ up to isomorphism. 
\end{proof}
\begin{LEMMA}\label{LEMMAHC3}
Right homotopic morphisms become equal in $\mathcal{D}^{\Cof, \Fib}[(\mathcal{W}^{\Cof, \Fib})^{-1}]$. 
\end{LEMMA}
\begin{proof}
The assertion follows from the fact that there exists a path object 
\[ \xymatrix{ \mathcal{F} & \ar@<2pt>[l]^{\pr_2} \ar@<-2pt>[l]_{\pr_1}  \mathcal{F}' & \ar[l]_i \mathcal{F} } \]
where $\mathcal{F}'$ is cofibrant and fibrant which realizes the right homotopy \cite[Proposition 1.2.6.]{Hov99}. This uses that all sources are cofibrant and the domain is fibrant.
The two morphisms $\pr_1$ and $\pr_2$ become equal because $i$ becomes invertible. 
\end{proof}

\begin{LEMMA}\label{LEMMAHC4}
The functor $\mathcal{D}^{\Fib,\Cof}[(\mathcal{W}^{\Fib,\Cof})^{-1}] \rightarrow \mathcal{D}^{\Cof}[(\mathcal{W}^{\Cof})^{-1}]$ and the functor\\
$\widetilde{\mathcal{D}^{\Fib,\Cof}[(\mathcal{W}^{\Fib,\Cof})^{-1}]} \rightarrow \widetilde{\mathcal{D}^{\Cof}[(\mathcal{W}^{\Cof})^{-1}]}$, respectively, are equivalences of multicategories.
\end{LEMMA}
\begin{proof}
The proof is analogous to that of Lemma~\ref{LEMMAHC2} but with some minor changes which require, in particular, the chosen order of restriction to cofibrant and fibrant objects.
We specify again a functor $\Phi$ in the other direction. On objects, $\Phi$ maps $\mathcal{E}$ to a fibrant replacement $R\mathcal{E}$. Note that $R\mathcal{E}$ is still cofibrant.  
A morphism $f \in \Hom(\mathcal{E}_1, \dots, \mathcal{E}_n; \mathcal{F})$ over $F$ corresponds to a morphism
$F_\bullet(\mathcal{E}_1, \dots, \mathcal{E}_n) \rightarrow \mathcal{F}$.
Now choose a lift (dotted arrow in the diagram)
\[ \xymatrix{
F_\bullet(\mathcal{E}_1, \dots, \mathcal{E}_n)   \ar[d] \ar[r]  &  \mathcal{F} \ar[r] & R\mathcal{F}  \\
F_\bullet(R\mathcal{E}_1, \mathcal{E}_2, \dots, \mathcal{E}_n) \ar[d] & \\
\vdots \ar[d] \\
F_\bullet(R\mathcal{E}_1, R\mathcal{E}_2, \dots, R\mathcal{E}_n) \ar@{.>}[rruuu] &   
} \]
It exists because the vertical maps are again trivial cofibrations (because all the $\mathcal{E}_i$ and $R\mathcal{E}_i$ are cofibrant). 
The lift is well-defined in $\mathcal{D}^{\Cof, \Fib}[(\mathcal{W}^{\Cof, \Fib})^{-1}]$, because two lifts in the triangle above become right homotopic
(because $R\mathcal{F}$ is fibrant by \cite[Proposition 1.2.5. (iv)]{Hov99}).
Therefore also the corresponding morphisms in $\Hom(R\mathcal{E}_1, \dots, R\mathcal{E}_n; R\mathcal{F})$
become equal in $\mathcal{D}^{\Cof, \Fib}[(\mathcal{W}^{\Cof, \Fib})^{-1}]$ by the previous lemma. 
It follows that $\Phi$ is indeed a functor which is inverse to the inclusion up to isomorphism. 
\end{proof}

\begin{LEMMA}
If the objects $F_\bullet()$ for all 0-ary morphisms in $\mathcal{S}$ are cofibrant then 
the natural functor
\[ \mathcal{D}^{\Fib,\Cof}[(\mathcal{W}^{\Fib,\Cof})^{-1}] \rightarrow \mathcal{D}^{\Fib,\Cof}/\sim \]
is an isomorphism of categories. Otherwise it is, if we modify the 0-ary morphisms as before. 
\end{LEMMA}
\begin{proof}
The natural functor $\mathcal{D}^{\Fib,\Cof} \rightarrow \mathcal{D}^{\Fib,\Cof}/\sim$ takes weak equivalences to isomorphisms \cite[Proposition 1.2.8]{Hov99}
and has the universal property of $\mathcal{D}^{\Fib,\Cof}[(\mathcal{W}^{\Fib,\Cof})^{-1}]$ by the same argument as in \cite[Proposition 1.2.9]{Hov99}.
\end{proof}

\begin{proof}[Proof of Proposition \ref{PROPDELIGNE}]
The previous lemmas showed that $\mathcal{D}[\mathcal{W}^{-1}]$ is equivalent to $\mathcal{D}^{\Fib,\Cof}/\sim$ if all objects of the form $F_\bullet()$ are cofibrant,
or if we replace the second multicategory by $\widetilde{\mathcal{D}^{\Fib,\Cof}/\sim}$, where we set $\Hom_{F,\widetilde{\mathcal{D}^{\Fib,\Cof}/\sim}}(; X):=\Hom_{\mathcal{D}_S[\mathcal{W}_{S}^{-1}]}(F_\bullet(), \mathcal{F})$ for all 0-ary morphism $F$ in $\mathcal{S}$ with domain $S$ and for every $\mathcal{F} \in \mathcal{D}_S$.

It remains to show that the functor
\[ p\,/\sim\,:\ \mathcal{D}^{\Fib,\Cof}/\sim\ \rightarrow \mathcal{S} \]
is bifibered  if all $F_\bullet()$ are cofibrant or otherwise bifibered for 
$n\ge 1$ (i.e.\@ (co)Cartesian morphisms exist for $n\ge1$). (The modification $\widetilde{\mathcal{D}^{\Fib,\Cof}/\sim}$ has been constructed in such a way that
it has coCartesian morphisms for $n=0$.)

We show that $p\,/\sim$ is opfibered, the other case being similar. 
Let $F$ be a multimorphism in $\mathcal{S}$ with codomain $S$. The set $\Hom_F(\mathcal{E}_1, \dots, \mathcal{E}_n; \mathcal{F})$ modulo right homotopy  is in bijection with the set
$\Hom_{\mathcal{D}_Y}(F_\bullet(\mathcal{E}_1, \dots, \mathcal{E}_n), \mathcal{F})$ modulo right homotopy. Since $\mathcal{F}$ is fibrant, the latter set is the same as
$\Hom_{\mathcal{D}_S}(R(F_\bullet(\mathcal{E}_1, \dots, \mathcal{E}_n)), \mathcal{F})$ modulo right homotopy. 
Hence morphisms in $\Hom_F(\mathcal{E}_1, \dots, \mathcal{E}_n; \mathcal{F})$ uniquely decompose as the composition
\[ \xymatrix{
\mathcal{E}_1 \ar@{-}[rd] \\
\ar@{}[r] \vdots & \cocart \ar[r] & F_\bullet(\mathcal{E}_1, \dots, \mathcal{E}_n) \ar[r] & R(F_\bullet(\mathcal{E}_1, \dots, \mathcal{E}_n))  \\
\mathcal{E}_n \ar@{-}[ru] \\
}\]
followed by a morphism in $\Hom_{\mathcal{D}_S}(R(F_\bullet(\mathcal{E}_1, \dots, \mathcal{E}_n)), \mathcal{F})$ modulo right homotopy.
More generally, by the same argument, a morphism in some
$\Hom_{GF}(\mathcal{F}_1, \dots, \mathcal{E}_1, \dots, \mathcal{E}_n, \dots, \mathcal{F}_m; \mathcal{G})$, where 
$G$ is another multimorphism in $\mathcal{S}$, modulo right homotopy
factorizes uniquely into the above composition followed by a morphism in 
\[ \Hom_{G}(\mathcal{F}_1, \dots, R(F_\bullet(\mathcal{E}_1, \dots, \mathcal{E}_n)), \dots, \mathcal{F}_m; \mathcal{G}) \] modulo right homotopy.

It remains to see that the push-forward in $\mathcal{D}[\mathcal{W}^{-1}]$ corresponds to the left derived functor of $F_\bullet$. 
For any objects $\mathcal{E}_1, \dots, \mathcal{E}_n$ the composition 
\[ \xymatrix{
RQ\mathcal{E}_1 \ar@{-}[rd] \\
\ar@{}[r] \vdots  & \cocart \ar[r] & F_\bullet(RQ\mathcal{E}_1, \dots, RQ\mathcal{E}_n) \ar[r] & R(F_\bullet(RQ\mathcal{E}_1, \dots, RQ\mathcal{E}_n))  \\
RQ\mathcal{E}_n \ar@{-}[ru] \\
}\]
is a coCartesian 
morphism lying over $F$, with domains isomorphic to the $\mathcal{E}_i$. 

However, the object $R(F_\bullet(RQ\mathcal{E}_1, \dots, RQ\mathcal{E}_n))$ is isomorphic to the value of the left derived functor of $F_\bullet$ at $\mathcal{E}_1, \dots, \mathcal{E}_n$. 
\end{proof}

\begin{PAR}\label{PARREEDY}
We now focus on the left case. If $I$ is a directed diagram, we proceed to construct a model structure on the fibers of the bifibration of multicategories (cf.\@ \ref{PARBIFIBI}):
\[ \Hom(I, \mathcal{D}) \rightarrow \Hom(I, \mathcal{S}) = \SSS(I). \]
This model structure is an analogue of the classical Reedy model structure and it has the property
that pull-backs w.r.t.\@ diagrams and the corresponding relative left Kan extension functors form a Quillen adjunction.

Let $I \in \Dir$ and let $F: I \rightarrow \mathcal{S}$ be a functor. We will define a model-category structure
\[ (\mathcal{D}_F, \Cof_F, \Fib_F, \mathcal{W}_F) \]
where $\mathcal{D}_F$ is the fiber of $\Hom(I, \mathcal{D})$ over $F$ and
where $\mathcal{W}_F$ is the class of morphisms which are element-wise in the corresponding $\mathcal{W}_{F(i)}$.

For any $G \in \mathcal{D}_F$, and for any $i \in I$, we define a {\bf latching object}
\[ L_i G := \colim_{I_i} \{F(\alpha)_\bullet G(j)\}_{\alpha:j \rightarrow i}, \] 
Here $I_i$ is the full subcategory of $I \times_{/I} i$ consisting of all objects except $\id_i$. 
We have a canonical morphism 
\[ L_i G \rightarrow G(i)  \]
in $\mathcal{D}_{F(i)}$.
We define $\Fib_F$ to be the class of morphisms which are element-wise in the corresponding $\Fib_{F(i)}$.
We define $\Cof_F$ to be the class of morphisms $G \rightarrow H$ such that for any $i \in I$ the induced morphism 
$\delta$ in the diagram
\[ \xymatrix{
L_i G \ar[d] \ar[r] & L_i H \ar[d] \\
 G(i) \ar[r] &  \text{push-out} \ar@{.>}[r]^-\delta & H(i)
} \]
belongs to $\Cof_{F(i)}$. We call a morphism $G \rightarrow H$ in $\Cof_F$ temporarily an {\bf acyclic cofibration} if $\delta$ is, in addition, a weak equivalence.
The proof that this yields a model-category structure is completely analogous to the classical case \cite[\S 5.1]{Hov99} (if $\mathcal{S}$ is trivial). We need a couple of lemmas:
\end{PAR}

\begin{LEMMA} \label{LEMMAPREPMODELCAT1fc}
The class of cofibrations (resp. acyclic cofibrations) in $\mathcal{D}_F$ consists precisely of the morphisms which
have the left lifting property w.r.t.\@ trivial fibrations (resp. fibrations).
These are stable under retracts.
\end{LEMMA}
\begin{proof}This is shown as in the classical case:
we first prove that acyclic cofibrations have the lifting property w.r.t.\@ fibrations. 
Consider a diagram
\[ \xymatrix{ G_1 \ar[r] \ar[d]^\alpha & H_1 \ar[d]^\beta \\ G_2 \ar[r] & H_2 \\  } \]
where $\alpha$ is an acyclic cofibration and $\beta$ is a fibration.
We proceed by induction on $n$ and assume that for all $i \in I$ with $\nu(i) < n$ a map
$G_2(i) \rightarrow H_1(i)$ has been constructed such that it is a lift in the above diagram, evaluated at $i$.
For each $i$ of degree $n$ consider the following diagram (where the morphism $L_i G_2 \rightarrow L_i H_1 \rightarrow H_1(i)$ is formed using the already constructed lifts):
\[ \xymatrix{ G_1(i) \coprod_{L_i G_1} L_i G_2 \ar[r] \ar[d]^{\alpha'(i)} & H_1(i) \ar[d]^{\beta(i)} \\ G_1(i) \ar[r] & H_2(i) \\  } \]
Here $\alpha'(i)$ is a trivial $\Cof_{F(i)}$-cofibration by definition, and $\beta(i)$ is a $\Fib_{F(i)}$-fibration by definition. Hence a lift exists.
In the same way the statement for cofibrations and for trivial fibrations is shown.
Closure under retracts is left as an exercise for the reader.
The assertion that the class of acyclic cofibrations (resp.\@ cofibrations) is {\em precisely} the class of morphisms that have the left lifting property w.r.t.\@ fibrations (resp.\@ trivial fibrations) follows from the retract argument as
for model categories.
\end{proof}

\begin{LEMMA}\label{LEMMAPREPMODELCAT2fc}
There exists a functorial factorization of morphisms in $\mathcal{D}_F$
into a fibration followed by an acyclic cofibration and into a trivial fibration followed by a cofibration.
\end{LEMMA}
\begin{proof}
We show this again by induction on $n$. We do the first case, the other being similar.
Let $G \rightarrow K$ a morphism in $\mathcal{D}_F$. We have the following diagram:
\[ \xymatrix{ L_i G \ar[r] \ar[d] & L_i H \ar[rr] \ar[d] &&  L_i K  \ar[d] \\ G(i) \ar[r] & G(i) \coprod_{L_i G} L_i H \ar@{.>} [r]& H(i) \ar@{.>}[r] & K(i) \\  } \]
Here the top row is constructed using the already defined factorizations. The object $H(i)$ and the dotted maps are constructed as the factorization in the model category
$\mathcal{D}_{F(i)}$ into a trivial $\Cof_{F(i)}$-cofibration followed by $\Fib_{F(i)}$-fibration. 
\end{proof}

\begin{LEMMA}\label{LEMMAPREPMODELCAT3fc}
The classes of cofibrations, acyclic cofibrations, fibrations and weak equivalences are stable under composition.
\end{LEMMA}
\begin{proof}This follows from the characterization by a lifting property (resp.\@ by definition for the case of the weak equivalences).   
\end{proof}

\begin{LEMMA}\label{LEMMAPREPMODELCAT4fc}
Acyclic cofibrations are precisely the trivial cofibrations.
\end{LEMMA}
\begin{proof}
We begin by showing that an acyclic cofibration is a weak equivalence.
It suffices to show that in the diagram
\[ \xymatrix{
L_i G \ar[d] \ar[r] & L_i H \ar[d] \\
 G(i) \ar[r] &   H(i)
} \]
the top horizontal morphism is a trivial cofibration. Then the lower horizontal morphism is a composition of two trivial cofibrations and hence is a weak equivalence.
The top morphism is indeed a trivial cofibration because
the morphism of $I_i$-diagrams (cf.\@ \ref{PARREEDY})
\[ \{F(\alpha)_\bullet G(j) \}_{\alpha: j \rightarrow i} \rightarrow \{F(\alpha)_\bullet G(j) \}_{\alpha: j \rightarrow i} \]
is a trivial cofibration in the classical sense (i.e.\@ over the constant diagram over $I_i$ with value $F(i)$) because of Lemmas\ref{LEMMAPREPMODELCAT5fc} and \ref{LEMMAPREPMODELCAT6fc}.

In the other direction, let $f$ be a trivial cofibration and factor it as $f=p g$, where $g$ is an acyclic cofibration and $p$ is a fibration. It follows that $p$ is a weak-equivalence.
Now construct a lift in the diagram
\[ \xymatrix{
F \ar[r]^g \ar[d]_f & H  \ar[d]^p \\
G \ar@{=}[r] & G
}\] 
This shows that $f$ is a retract of $g$, and hence is an acyclic cofibration, too.
\end{proof}

\begin{LEMMA}\label{LEMMAPREPMODELCAT5fc}
For each (1-ary) morphism of diagrams $f \in \Hom_{\mathcal{S}}(X_1;  Y)$ there is an associated push-forward and an associated pull-back, 
defined by taking the point-wise push-forward $f_\bullet$, and point-wise pull-back $f^{\bullet}$ (cf.\@ \ref{PARBIFIBI}), respectively.
The push-forward $f_\bullet$ respects the classes of cofibrations and acyclic cofibrations.
The pull-back $f^{\bullet}$ respects the classes of fibrations and trivial fibrations.
\end{LEMMA}
\begin{proof}
It suffices (by the lifting property) to show that $f^\bullet$ respects fibrations and trivial fibrations. This is clear because they are defined point-wise.
\end{proof}

A posteriori this will say that the pair of functors $f^\bullet, f_\bullet$ form a Quillen adjunction between the corresponding model categories (cf.\@ \ref{LEMMABIFIBMULTIMODELCAT}).

\begin{LEMMA}\label{LEMMAPREPMODELCAT6fc}
Let $i \in I$ be an object, let $\iota: I_i \rightarrow I$ be the corresponding latching category with its natural functor to $I$,
and let $F_i:= \iota^*F: I_i \rightarrow \mathcal{S}$ be the restriction of $F$ to $I_i$. 
The pull-back $\iota^*: \mathcal{D}_F \rightarrow \mathcal{D}_{F_i}$ respects cofibrations  and acyclic cofibrations.
\end{LEMMA}
\begin{proof}
It is easy to see that the pull-back induces an isomorphism of the corresponding latching objects as in the classical case.
\end{proof}

\begin{KOR}
The structure constructed in \ref{PARREEDY} defines a model category. 
\end{KOR}
\begin{proof}
This follows from the previous Lemmas. 
\end{proof}

\begin{PROP}\label{PROPFDER3}
For any morphism of directed diagrams $\alpha: I \rightarrow J$, and for any functor $F: J \rightarrow \mathcal{S}$, 
the functor
\[ \alpha^*: \mathcal{D}_F \rightarrow \mathcal{D}_{\alpha^*F} \]
has a left adjoint $\alpha_!^F$. The pair $\alpha^*, \alpha_!^F$ define a Quillen adjunction.
\end{PROP}
\begin{proof}
That the two functors define a Quillen adjunction is clear once we have shown that $\alpha_!$ exists because $\alpha^*$ preserves fibrations and weak equivalences.
Let $G$ be an object of $\mathcal{D}_F$. We define
\[ (\alpha_! G)(j)  := \colim_{I \times_{/J} j} \SSS(\mu)_\bullet \iota_j^* G. \]
For each morphism $\mu: j \rightarrow j'$ we get a functor
\[ \widetilde{\mu}: I \times_{/J} j \rightarrow I \times_{/J} j'  \]
and hence an induced morphism
\[ F(\mu)_\bullet \SSS(\mu)_\bullet  \iota_j^* G \rightarrow \widetilde{\mu}^* \SSS(\mu')_\bullet  \iota_{j'}^*. \]
Sincd $F(\mu)_\bullet$ commutes with colimits we get a morphism 
\[ F(\mu)_\bullet \colim_{I \times_{/J} j}  \SSS(\mu)_\bullet  \iota_j^* G \rightarrow \colim_{I \times_{/J} j'} \SSS(\mu')_\bullet \iota_{j'}^* \]
which we define to be $(\alpha_! G)(\mu)$.
We now proceed to show that the functor we have constructed is indeed adjoint to $\alpha^*$.
A morphism $\mu: G \rightarrow \alpha^*H$ is given by a collection of maps
$a(i): G(i) \rightarrow H(\alpha(i))$ for all objects $i \in I$, subject to the condition that the diagram
\[ \xymatrix{
F(\alpha(\lambda))_\bullet G(i) \ar[rr]^{F(\alpha(\lambda))_\bullet a(i)} \ar[d]^{\overline{G(\lambda)}} & & F(\alpha(\lambda))_\bullet H(\alpha(i)) \ar[d]^{\overline{H(\alpha(\lambda))}}  \\
G(i')  \ar[rr]^{a(i')}  & & H(\alpha(i'))  
} \]
commutes for each morphism $\lambda: i \rightarrow i'$ in $I$. 
For each $j \in J$ and morphism $\mu: \alpha(i) \rightarrow j$ we get a morphism
\[ \overline{H(\mu)} \circ (F(\mu)_\bullet  a(i)): F(\mu)_\bullet G(i) \rightarrow H(j) \] 
and therefore for fixed $j$ a morphism
\[ \colim_{I \times_{/J} j} \SSS(\mu)_\bullet  \iota_j^* G \rightarrow H(j).  \]
One checks that this yields a morphism 
$\alpha_! G \rightarrow H.$
On the other hand, let $b:  \alpha_! G \rightarrow H$ be a morphism
given by 
\[ b(j): \colim_{I \times_{/J} j} \SSS(\mu)_\bullet  \iota_j^* G \rightarrow H(j)  \]
or equivalently for all $\mu: \alpha(i) \rightarrow j$ by morphisms
\[  F(\mu)_\bullet G(i) \rightarrow H(j).  \]
In particular, if $\mu$ is the identity of $\alpha(i)$, we get morphisms
\[  G(i) \rightarrow H(\alpha(i))  \]
which constitute a morphism of diagrams $G \rightarrow \alpha^*H$.
One checks that these associations are inverse to each other.
\end{proof}

\begin{LEMMA}
Let $\alpha: I \rightarrow J$ be a morphism of directed diagrams and let $j$ be an object of $J$.
The functor $\iota_j^*: \mathcal{D}_I \rightarrow \mathcal{D}_{I \times_{/J} j}$ respects cofibrations and trivial cofibrations.
\end{LEMMA}
\begin{proof} This follows easily from the fact that $\iota_j$ induces a canonical identification
\[ I_i = ({I \times_{/J} j})_\mu \]
for any $\mu= (i, \alpha(i) \rightarrow j)$. For this implies that we have a canonical isomorphism $L_i G \cong L_\mu \iota^*_j G$. 
\end{proof}

\begin{LEMMA}\label{LEMMABIFIBMULTIMODELCAT}
The bifibration of multicategories, defined in \ref{PARBIFIBI}
\[ \Hom(I, \mathcal{D}) \rightarrow \Hom(I, \mathcal{S}) = \SSS(I) \]
equipped with the model-category structures constructed in \ref{PARREEDY} is a bifibration of multi-model-categories in the sense of \ref{PARQUILLENN}.
\end{LEMMA}
\begin{proof}
First for each multi-morphism of diagrams $f \in \Hom_{\mathcal{S}}(X_1, \dots, X_n;  Y)$ we have to see that 
the push-forward and the various pull-backs form a Quillen adjunction in $n$ variables.
The case $n=1$ has been treated above. We only work out the case $n=2$, the proof for higher $n$ being similar.
It suffices to check the following: for any cofibration
$\mathcal{E}_1 \rightarrow \mathcal{E}_1'$ and for any fibration $\mathcal{F} \rightarrow \mathcal{F}'$ the dotted induced morphism in the following diagram
\[  \xymatrix{
 f^{\bullet,2}(\mathcal{E}_1'; \mathcal{F}) \ar@{.>}[r] & \text{pull-back} \ar[r] \ar[d] &  f^{\bullet,2}(\mathcal{E}_1'; \mathcal{F}') \ar[d] \\
& f^{\bullet,2}(\mathcal{E}_1; \mathcal{F}) \ar[r] &  f^{\bullet,2}(\mathcal{E}_1; \mathcal{F}') 
 }  \]
is a fibration. Since fibrations are defined point-wise and fibered products are computed point-wise, we have only to see 
that the assertion holds point-wise. Now $\mathcal{F} \rightarrow \mathcal{F}'$ is a point-wise fibration and $\mathcal{E}_1 \rightarrow \mathcal{E}_1'$ is a Reedy cofibration, so 
by the reasoning in the proof of Lemma~\ref{LEMMAPREPMODELCAT4fc} it is in particular
a point-wise cofibration. Hence the assertion holds because of the assumption that $\mathcal{D} \rightarrow \mathcal{S}$ 
is a bifibration of multi-model-categories (\ref{PARQUILLENN}).
The requested property for the 0-ary push-forward is easier and is left to the reader.
\end{proof}

\begin{PROP}\label{SATZFDER0}
The functor $\DD(I) \rightarrow \SSS(I)$ defined in \ref{DEFFIBDERMODEL} is a bifibration of multicategories whose fibers are equivalent to $\mathcal{D}_F[\mathcal{W}_F^{-1}]$. The pull-back and push-forward functors are given by the left derived functors of $f_\bullet$, and by the right derived functors of $f^{\bullet,j}$, respectively.
\end{PROP}

\begin{proof}
We have seen in~\ref{LEMMABIFIBMULTIMODELCAT} that the fibers of $\Hom(I, \mathcal{D}) \rightarrow \SSS(I)$ are a bifibration of multi-model-categories 
in the sense of~\ref{PARQUILLENN}. 
Therefore by Proposition~\ref{PROPDELIGNE} we get that $\DD(I) \rightarrow \SSS(I)$ are bifibered multicategories with the requested properties. 
\end{proof}

\begin{proof}[Proof of Theorem~\ref{SATZEXISTENCEFMULTIDER}]
(Der1) and (Der2) for $\DD$ and $\SSS$  are obvious.

(FDer0 left) and the first part of (FDer0 right) follow from Theorem~\ref{SATZFDER0}.

(FDer3 left) follows from~\ref{PROPFDER3}.

(FDer4 left):
By construction of $\alpha_!$ the natural base-change
\begin{equation}\label{eqbasechangemc} \colim \SSS(\mu)_\bullet \iota_j^* G \rightarrow j^* \alpha_! G  \end{equation}
is an isomorphism for the non-derived functors.
For the derived functors the same follows because all functors in the equation respect cofibrations and trivial cofibrations and
all functors which have to be derived in (\ref{eqbasechangemc}) are left Quillen functors and hence can be derived by composing them 
with cofibrant replacement.

(FDer3 right) and (FDer4 right) are shown precisely the same way.

(FDer5 left):
Fixing a morphism $f \in \Hom(S_1, \dots, S_n; T)$ in $\mathcal{S}$ and objects $\mathcal{E}_2, \dots, \mathcal{E}_n$ over $S_2, \dots, S_n$ we have by Theorem~\ref{SATZFDER0} a push-forward functor
\begin{eqnarray*} 
\DD(I \times J)_{p^*S_1} &\rightarrow& \DD(I \times J)_{p^*T} \\
 \mathcal{E}_1 &\mapsto& (p^*f)_\bullet(\mathcal{E}_1, p^*\mathcal{E}_2, \dots, p^*\mathcal{E}_n)
\end{eqnarray*}
(we denote it with the same letter as the underived version)
which, by (FDer0 left), defines a morphism of pre-derivators 
\[ \DD_{S_1} \rightarrow \DD_{T}. \]
We first show that it preserves colimits, i.e.\@ that for $p: J \rightarrow \cdot$ we have that for all $\mathcal{E}_1 \in \mathcal{D}_{p^*S_1}(I \times J)$ the natural morphism
\[  f_\bullet(p_*\mathcal{E}_1, \mathcal{E}_2, \cdots, \mathcal{E}_n) \rightarrow  p_* (p^*f)_\bullet(\mathcal{E}_1, p^*\mathcal{E}_2, \cdots, p^*\mathcal{E}_n)  \]
(where we wrote $p$ also for the projection $p: I \times J \rightarrow I$) 
is an isomorphism. This is the same as showing that
\[  p^* f^{1,\bullet}(\mathcal{E}_1, \dots, \mathcal{E}_n) \rightarrow (p^*f)^{1,\bullet}(p^* \mathcal{E}_1, \dots, p^*\mathcal{E}_n)  \]
is an isomorphism. This follows from Lemma~\ref{LEMMABIFIBI} because it suffices to check this for the underived functors. 
Now let $\alpha: I \rightarrow J$ be a Grothendieck opfibration. To show that
\[  f_\bullet(\alpha_*\mathcal{E}_1, \mathcal{E}_2, \dots, \mathcal{E}_n) \rightarrow  \alpha_* (\alpha^* f)_\bullet(\mathcal{E}_1, \alpha^*\mathcal{E}_2, \dots, \alpha^*\mathcal{E}_n)  \]
is an isomorphism we may show this point-wise. Indeed, after applying $j^*$
we get 
\[ (j^*f)_\bullet(j^* \alpha_* \mathcal{E}_1, j^*\mathcal{E}_2, \dots, j^*\mathcal{E}_n) \rightarrow  j^* \alpha_* (\alpha^* f)_\bullet(\mathcal{E}_1, \alpha^*\mathcal{E}_2, \dots, \alpha^*\mathcal{E}_n)  \]
\[ (j^*f)_\bullet(p_* \iota_j^* \mathcal{E}_1, j^*\mathcal{E}_2, \dots, j^*\mathcal{E}_n) \rightarrow  p_*  \iota_j^*  (\alpha^* f)_\bullet (\mathcal{E}_1, \alpha^*\mathcal{E}_2, \dots, \alpha^*\mathcal{E}_n)  \]
where $\iota_j: I_j \rightarrow I$ is the inclusion of the fiber. Note that the commutative diagram 
\[ \xymatrix{
I_j \ar[r]^{\iota_j} \ar[d]_p & I \ar[d]^\alpha \\
j \ar[r] & J
} \]
is homotopy exact by Lemma~\ref{PROPHOMCART}, 2.\@ because $\alpha$ is a Grothendieck opfibration. Finally we get the morphism
\[ (j^*f)_\bullet(p_* \iota_j^* \mathcal{E}_1, j^*\mathcal{E}_2, \dots, j^*\mathcal{E}_n) \rightarrow  p_* (j^*f)_\bullet(\iota_j^* \mathcal{E}_1, p^* j^* \mathcal{E}_2, \dots, p^* j^* \mathcal{E}_n)  \]
which is an isomorphism by the above reasoning.

By Lemma~\ref{LEMMALEFTRIGHT} the full content of (FDer0 right) follows from (FDer5 left) while (FDer5 right) follows from (FDer0 left). 
\end{proof}

\newpage

\appendix
\section{Fibrations of categories}\label{APP}

\subsection{Grothendieck (op-)fibrations}\label{APPGROTH}

\begin{PAR}[right]
Let $p: \mathcal{D} \rightarrow \mathcal{S}$ be a functor, and let $f: S \rightarrow T$ be a morphism
in $\mathcal{S}$.  A morphism
$\xi: \mathcal{E}' \rightarrow \mathcal{E}$ over $f$ is called {\bf Cartesian} if the composition with $\xi$ induces an isomorphism
\[  \Hom_{g}(\mathcal{F}, \mathcal{E}') \cong \Hom_{f \circ g}(\mathcal{F}, \mathcal{E}) \]
for any morphism $g: R \rightarrow S$ in $\mathcal{S}$ and for every $\mathcal{F} \in \mathcal{D}_R$.

The functor $p$ is called a {\bf Grothendieck fibration} if for any $f: S \rightarrow T$ and for every object $\mathcal{E}$ in $\mathcal{D}_T$ (i.e.\@ such that $p(\mathcal{E})=T$) there
exists a Cartesian morphism $\mathcal{E}' \rightarrow \mathcal{E}$.
\end{PAR}

\begin{PAR}[left]
Let $p: \mathcal{D} \rightarrow \mathcal{S}$ be a functor, and let $f: S \rightarrow T$ be a morphism
in $\mathcal{S}$.  A morphism 
$\xi: \mathcal{E} \rightarrow \mathcal{E}'$ over $f$ is called {\bf coCartesian} if  the composition with $\xi$ induces an isomorphism
\[  \Hom_{g}(\mathcal{E'}, \mathcal{F}) \cong \Hom_{g \circ f}(\mathcal{E}, \mathcal{F}) \]
for any morphism $g: T \rightarrow U$ in $\mathcal{S}$  and for every $\mathcal{F} \in \mathcal{D}_U$ .

The functor $p$ is called {\bf a Grothendieck opfibration} if for any $f: S \rightarrow T$ and for every object $\mathcal{E}$ in $\mathcal{D}_S$ there
exists a coCartesian morphism $\mathcal{E} \rightarrow \mathcal{E}'$.
\end{PAR}

\begin{PAR}
The functor $p$ is a Grothendieck opfibration if and only if $p^{\op}: \mathcal{D}^{\op} \rightarrow \mathcal{S}^{\op}$ is a Grothendieck fibration.
We say that $p$ is a {\bf bifibration} if is a fibration and an opfibration at the same time.
If $p: \mathcal{D} \rightarrow \mathcal{S}$ is a Grothendieck fibration we may choose an associated pseudo-functor, i.e.
to each $S \in \mathcal{S}$ we associate the category $\mathcal{D}_S$, and to each $f: S \rightarrow T$ we associate
a push-forward functor
\[ f_\bullet: \mathcal{D}_S \rightarrow \mathcal{D}_T \]
such that for each $\mathcal{E}$ in $\mathcal{D}_S$ there is a Cartesian morphism $\mathcal{E} \rightarrow f_\bullet \mathcal{E}$.
The same holds similarly for an opfibration with the pull-back $f^\bullet$ instead of the push-forward.
If the functor $p$ is a bifibration, $f_\bullet$ is left adjoint to $f^\bullet$. 
Situations where this is the opposite can be modeled by considering bifibrations $\mathcal{D} \rightarrow \mathcal{S}^{\op}$.
\end{PAR}

\subsection{Fibered multicategories and the six functors}\label{APPMULTI}

\begin{PAR}We give a definition of a (op-)fibered multicategory. This is a straightforward generalization of the notion of (op-)fibered category given in the section~\ref{APPGROTH}. It is very useful to encode the formalism of the Grothendieck six functors.
Details about (op-)fibered multicategories can be found, for instance, in \cite{Her00, Her04}.

The reader should keep in mind that a multicategory
abstracts the properties of multilinear maps, and indeed every monoidal category gives rise to a multicategory setting
\begin{equation} \label{eqmonmulti} 
\Hom(A_1, \dots, A_n; B) := \Hom((A_1 \otimes (A_2 \otimes ( \cdots) ) ), B).  
\end{equation}
\end{PAR}

\begin{DEF}
A {\bf multicategory} $\mathcal{D}$ consists of 
\begin{itemize}
\item a class of objects $\Ob(\mathcal{D})$;
\item for every $n \in \Z_{\ge 0}$, and objects $X_1, \dots, X_n, Y$ a class
\[ \Hom(X_1, \dots, X_n; Y);  \]
\item a composition law, i.e. for objects $X_1, \dots, X_n$, $Y_1, \dots, Y_m$, $Z$ and for each integer $1\le i \le m$ a map:
\[ \Hom(X_1, \dots, X_n; Y_i) \times \Hom(Y_1, \dots, Y_m; Z) \rightarrow \Hom(Y_1, \dots, Y_{i-1}, X_1, \dots, X_n, Y_{i+1}, \dots, Y_m; Z); \]
\item for each object $X \in \Ob(\mathcal{D})$ an identity $\id_X \in \Hom(X; X)$; 
\end{itemize}
satisfying associativity and identity laws.
A symmetric (braided) multicategory is given by an action of the symmetric (braid) groups, i.e.\@ isomorphisms
\[  \alpha: \Hom(X_1, \dots, X_n; Y)  \rightarrow \Hom(X_{\alpha(1)}, \dots, X_{\alpha(n)}; Y) \]
for $\alpha \in S_n$ (resp.\@ $\alpha \in B_n$) forming an action which is compatible with composition in the obvious way (substitution of strings in the braid group).
\end{DEF}

In some references the composition is defined in a seemingly more general way; in the presence of identities these descriptions are, however, equivalent.
We denote a multimorphism in $f \in \Hom(X_1, \dots, X_n; Y)$ also by
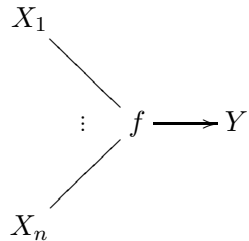
\[ \xymatrix{
X_1 \ar@{-}[rd] \\
\ar@{}[r]|{\vdots} & f \ar[r] &  Y  \\
X_n \ar@{-}[ru] \\
}\]
for $n\ge 1$, or by
\[ \xymatrix{
\ar@{o->}[rr]^f &&  Y  
}\]
for $n=0$.

We will also need the definition of a strict 2-multicategory which is a multicategory enriched in (usual) categories: 

\begin{DEF}
A (strict) {\bf 2-multicategory} $\mathcal{D}$ consists of 
\begin{itemize}
\item a class of objects $\Ob(\mathcal{D})$;
\item for every $n \in \Z_{\ge 0}$, and objects $X_1, \dots, X_n, Y$ a category
\[ \Hom(X_1, \dots, X_n; Y);  \]
\item a composition, i.e.\@ for objects $X_1, \dots, X_n$, $Y_1, \dots, Y_m$, $Z$ and for each integer $1\le i \le m$ a functor:
\[ \Hom(X_1, \dots, X_n; Y_i) \times \Hom(Y_1, \dots, Y_m; Z) \rightarrow \Hom(Y_1, \dots, Y_{i-1}, X_1, \dots, X_n, Y_{i+1}, \dots, Y_m; Z); \]
\item for each object  $X \in \Ob(\mathcal{D})$ an identity object $\id_X$ in the category $\Hom(X; X)$;
\end{itemize}
satisfying strict associativity and identity laws.
A symmetric (braided) 2-multicategory is given by an action of the symmetric (braid) groups, i.e.\@ isomorphisms of categories
\[  \alpha: \Hom(X_1, \dots, X_n; Y) \rightarrow \Hom(X_{\alpha(1)}, \dots, X_{\alpha(n)}; Y) \]
for $\alpha \in S_n$ (resp.\@ $\alpha \in B_n$) forming an action which is strictly compatible with composition in the obvious way (substitution of strings in the braid group).
\end{DEF}

The 1-composition of 2-morphisms is (as for usual 2-categories) determined by the following whiskering operations: Let $f, g \in \Hom(X_1, \dots, X_n; Y_i)$ be 1-morphisms
and let $h \in \Hom(Y_1, \dots, Y_m; Z)$ be a 1-morphism and let $\mu: f \Rightarrow g$ be a 2-morphism in $\Mor(\Hom(X_1, \dots, X_n; Y_i))$. Then we define
\[ h  \ast \mu := \id_h \cdot \mu \]
where the right hand side is the image of the morphism $\id_h \times \mu$ under the composition functor.
Similarly we define $\mu \ast h$ for $\mu: f \Rightarrow g$ with $f, g \in \Hom(Y_1, \dots, Y_m; Z)$ and $h \in \Hom(X_1, \dots, X_n; Y_i)$.

\begin{PAR}
We leave it to the reader to state the obvious definition of a functor between multicategories.
Similarly there is a definition of a {\bf opmulticategory}, in which we have classes
\[ \Hom(X; Y_1, \dots, Y_n)  \]
and similar data. For a multicategory $\mathcal{D}$ we get a natural opmulticategory $\mathcal{D}^{\mathrm{\mathrm{\op}}}$ by reversing the arrows.

The trivial category $\{\cdot\}$ is considered as a multicategory setting all $\Hom(\cdot, \dots, \cdot\,;\,\cdot)$ to the 1-element set. It is the final object in the ``category'' of multicategories.
\end{PAR}

To clarify the precise relation between multicategories and monoidal categories we have to define Cartesian and coCartesian morphisms. 
It turns out that we can actually give a definition which is a common generalization of coCartesian morphisms in opfibered categories and the morphisms expressing the existence of a tensor product:

\begin{DEF}\label{DEFCARTCOCART}
Consider a functor of
multicategories $p: \mathcal{D} \rightarrow \mathcal{S}$. We call a morphism
\[ \xi \in \Hom(X_1, \dots, X_n; Y) \] in $\mathcal{D}$ {\bf coCartesian} w.r.t.\@ $p$, 
if for all $Y_1,\dots,Y_m,Z$ with $Y_i=Y$, and for all 
\[ f \in \Hom(p(Y_1), \dots, p(Y_m); p(Z)) \]
the map
\begin{eqnarray*} 
\Hom_{f}(Y_1, \dots, Y_m; Z) &\rightarrow& \Hom_{f \circ p(\xi)}(Y_1, \dots, Y_{i-1}, X_1, \dots, X_n, Y_{i+1}, \dots, Y_m; Z) \\
 \alpha &\mapsto& \alpha \circ \xi 
\end{eqnarray*} 
is bijective. 
We call a morphism
\[ \xi \in \Hom(X_1, \dots, X_n; Y)\] in $\mathcal{D}$  {\bf Cartesian} w.r.t.\@ $p$ at the $i$-th slot, if for all $Y_1, \dots, Y_m$, and for all $f \in \Hom(p(Y_1), \dots, p(Y_m); p(X_i))$ the map
\[ \Hom_{f}(Y_1, \dots, Y_m; X_i) \rightarrow \Hom_{ p(\xi) \circ f}(X_1, \dots, X_{i-1}, Y_1, \dots, Y_m, X_{i+1}, \dots, X_n; Z). \]
\[ \alpha \mapsto \xi \circ \alpha \]
is bijective.

The functor $p: \mathcal{D} \rightarrow \mathcal{S}$ is called an {\bf opfibered multicategory} if for every
$g \in \Hom(S_1, \dots, S_n; T)$ in $\mathcal{S}$, and for every collection of objects $X_i$ with $p(X_i) = S_i$ there is some object $Y$ over $T$ and some coCartesian morphism
$\xi \in \Hom(X_1, \dots, X_n; Y)$ such that $p(\xi)=g$.

The functor $p: \mathcal{D} \rightarrow \mathcal{S}$ is called a {\bf fibered multicategory} if for every $1 \le j \le n$, for each
$g \in \Hom(S_1, \dots, S_n; T)$ in $\mathcal{S}$, for every collection of objects $X_i$ for $i \not= j$ with $p(X_i) = S_i$, and for every $Y$ over $T$, there is some object $X_j$ and some Cartesian morphism w.r.t.\@ the $j$-th slot $\xi \in \Hom(X_1, \dots, X_n; Y)$ such that $p(\xi)=g$.

The functor $p: \mathcal{D} \rightarrow \mathcal{S}$ is called a {\bf bifibered multicategory} if it is both fibered and opfibered.

A {\bf morphism of (op)fibered multicategories} is a commutative diagram of functors
\[ \xymatrix{  \mathcal{D}_1  \ar[r]^F \ar[d]  & \mathcal{D}_2 \ar[d]  \\
 \mathcal{S}_1  \ar[r]^G & \mathcal{S}_2  } \]
 such that $F$ maps (co-)Cartesian morphisms to (co-)Cartesian morphisms.
\end{DEF}

It turns out that the composition of Cartesian morphisms is Cartesian (and similarly for coCartesian morphisms if they are composed w.r.t.\@ the right slot)\footnote{As with fibered categories there are weaker notions of Cartesian which still uniquely determine a Cartesian morphism (up to
isomorphism) from given objects over a given multimorphism, however, do not imply stability under composition. Similarly for coCartesian morphisms.}.

\begin{LEMMA}\label{LEMMAPROPMULTI}
\begin{enumerate}
\item An opfibered multicategory $p: \mathcal{D} \rightarrow \{\cdot\}$ is a monoidal category defining $X \otimes Y$ to be the target of a coCartesian arrow from the pair $X, Y$
over the unique map in $\Hom(\cdot, \cdot; \cdot)$ of the final multicategory $\{\cdot\}$.

Conversely any monoidal category gives rise to an opfibered multicategory $p: \mathcal{D} \rightarrow \{\cdot\}$ via (\ref{eqmonmulti}). A multicategory $\mathcal{D}$ is a {\em closed} category if and only if it is fibered over $\{\cdot\}$.
In particular, the fibers of an (op)fibered multicategory $p: \mathcal{D} \rightarrow \mathcal{S}$ are always closed/monoidal in the following sense:
given any functor of multicategories\footnote{This specifies also morphisms
in $\Hom(\underbrace{X, \dots, X}_n; X)$, for all $n$, compatible with composition.} $x: \{\cdot\} \rightarrow \mathcal{S}$, the category $\mathcal{D}_x$ of objects over $x$ is closed/monoidal.

\item Given (op)fibered multicategories $p: \mathcal{C} \rightarrow \mathcal{D}$ and $q: \mathcal{D} \rightarrow \mathcal{E}$ also the composition $q\circ p$ is an (op)fibered multicategory. In particular, if we have an opfibered multicategory $p: \mathcal{C} \rightarrow \mathcal{S}$ and if $\mathcal{S} \rightarrow \{\cdot\}$ is opfibered (i.e.\@ $\mathcal{S}$ is monoidal) then also $\mathcal{C} \rightarrow \{\cdot\}$ is opfibered (i.e.\@ $\mathcal{C}$ is monoidal).
The same holds dually. A morphism $\alpha$ is (co)Cartesian for $q\circ p$ if and only if $\alpha$ is (co)Cartesian for $p$ and $p(\alpha)$ is (co)Cartesian for $q$.
\end{enumerate}
\end{LEMMA}

Similarly, the unit $1$ is just the target of a coCartesian morphism in $\Hom(;1)$ which exists by definition
(the existence is also required for the empty set of objects).

The second part of the lemma encapsulates the distinction between internal and external tensor product in a four (or six) functor context, see \ref{CONCRETE6FU}.

\begin{PAR}
Let $\mathcal{D}, \mathcal{S}$ be (usual) multicategories. 
More generally any opfibered multicategory $\mathcal{D} \rightarrow \mathcal{S}$ gives rise to a pseudo-functor of 2-multi-categories

\[ \mathcal{S} \rightarrow \mathcal{MCAT}^{2-\op} \]
where $\mathcal{MCAT}$ is the 2-multicategory of categories, whose objects are categories and the morphism categories are defined to be:
\[ \Hom_{\mathcal{MCAT}}(\mathcal{C}_1, \dots, \mathcal{C}_n; \mathcal{D}) := \mathrm{Fun}(\mathcal{C}_1 \times \dots \times \mathcal{C}_n, \mathcal{D}) \]

Here by a pseudo-functor $\Psi: \mathcal{S} \rightarrow \mathcal{T}$, where $\mathcal{T}$ is a 2-multicategory, we understand the obvious generalization of the usual concept of a pseudo-functor. This means that 
for each $f \in \Hom_{\mathcal{S}}(S_1, \dots, S_n; T)$ we are given a functor $\Psi(f) \in \Hom(\Psi(S_1), \dots, \Psi(S_n); T)$ and for each composition 
$g \cdot f$ a natural isomorphism
\begin{equation}\label{eqpseudo}
 \Psi_{f,g}:   \Psi(g) \Psi(f) \Rightarrow  \Psi(g \cdot f) 
 \end{equation}
satisfying the usual relation for composable morphisms $f, g$ and $h$:
\[ (\Psi(h) \ast \Psi_{f,g}) \Psi_{gf,h} = (\Psi_{g,h} \ast \Psi(f)) \Psi_{f, hg}.   \]

This definition generalizes readily to the case in which also $\mathcal{S}$ is a 2-multicategory, the only modification being that, on morphisms, we are given {\em functors} 
\[ \Hom_\mathcal{S}(S_1, \dots, S_n; T) \rightarrow  \Hom_{\mathcal{T}}(\Psi(S_1), \dots, \Psi(S_n); \Psi(T)) \] 
and the 2-morphisms (\ref{eqpseudo}) have to be functorial in $f$ and $g$. 
\end{PAR}

\begin{PAR}
Translated back to the language of fibrations we arrive at the following definition: see~\ref{DEF2CART}.

First note that the definition of coCartesian morphism (cf.\@ \ref{DEFCARTCOCART}) may be stated in the following way: A morphism
\[ \xi \in \Hom(X_1, \dots, X_n; Y) \] in $\mathcal{D}$ coCartesian w.r.t.\@ $p$, 
if for all $Y_1,\dots,Y_m, Z$ with $Y_i=Y$
the diagram of sets
{\footnotesize
\[ \xymatrix{
\Hom(Y_1, \dots, Y_m; Z) \ar[r]^-{\circ \xi} \ar[d] & \Hom(Y_1, \dots, Y_{i-1}, X_1, \dots, X_n, Y_{i+1}, \dots, Y_m; Z) \ar[d] \\
\Hom(p(Y_1), \dots, p(Y_m); p(Z)) \ar[r]^-{ \circ p(\xi)} & \Hom(p(Y_1), \dots, p(Y_{i-1}), p(X_1), \dots, p(X_n), p(Y_{i+1}), \dots, p(Y_m); p(Z))
} \]}is Cartesian.
\end{PAR}

\begin{DEF}\label{DEF2CART}
Let $p: \mathcal{D} \rightarrow \mathcal{S}$ be a strict functor of 2-multicategories. 
A 1-morphism
\[ \xi \in \Hom_f(\mathcal{E}_1, \dots, \mathcal{E}_n; \mathcal{F}) \] 
in $\mathcal{D}$ over $f \in \Hom(S_1, \dots, S_n; T)$
is called {\bf coCartesian} w.r.t.\@ $p$, 
 if for all $\mathcal{F}_1,\dots,\mathcal{F}_m, \mathcal{G}$ with $\mathcal{F}_i=\mathcal{F}$
 the diagram of categories
\[ \xymatrix{
\Hom(\mathcal{F}_1, \dots, \mathcal{F}_m; \mathcal{G}) \ar[r]^-{\circ \xi} \ar[d] & \Hom(\mathcal{F}_1, \dots, \mathcal{F}_{i-1}, \mathcal{E}_1, \dots, \mathcal{E}_n, \mathcal{F}_{i+1}, \dots, \mathcal{F}_m; \mathcal{G}) \ar[d] \\
\Hom(T_1, \dots, T_m; U) \ar[r]^-{ \circ p(\xi)} & \Hom(T_1, \dots, T_{i-1}, S_1, \dots, S_n, T_{i+1}, \dots, T_m; U)
} \]
is Cartesian (where we set $T_k:=p(\mathcal{F}_k)$ and $U:=p(\mathcal{G})$). 

The strict functor $p$ is called a {\bf 2-opfibered 1-opfibered multicategory (with 1-categorical fibers)} if for all $f \in \Hom(S_1, \dots, S_n; T)$ and objects $\mathcal{E}_1, \dots, \mathcal{E}_n$ with $p(\mathcal{E}_i)=S_i$ there is a coCartesian 1-morphism with domains $\mathcal{E}_1, \dots, \mathcal{E}_n$. Furthermore the functors 
\[ \Hom(\mathcal{E}_1, \dots, \mathcal{E}_n; \mathcal{F}) \rightarrow \Hom(p(\mathcal{E}_1), \dots, p(\mathcal{E}_n); p(\mathcal{F})) \]
 have to be Grothendieck opfibrations (with {\em discrete fibers}) and composition has to be a morphism of Grothendieck opfibrations.

The functor $p: \mathcal{D} \rightarrow \mathcal{S}$ is called a {\bf 2-fibered 1-opfibered multicategory (with 1-categorical fibers)} if for every $1 \le j \le n$ and for each
$g \in \Hom(S_1, \dots, S_n; T)$ in $\mathcal{S}$, and for each collection of objects $\mathcal{E}_i$ for $i \not= j$ with $p(\mathcal{E}_i) = S_i$, and for each $\mathcal{F}$ over $T$, there is some object $\mathcal{E}_j$ and some Cartesian 1-morphism w.r.t.\@ the $j$-th slot $\xi \in \Hom(\mathcal{E}_1, \dots, \mathcal{E}_n; \mathcal{F})$ with $p(\xi)=g$.
Furthermore the functors 
\[ \Hom(\mathcal{E}_1, \dots, \mathcal{E}_n; \mathcal{F}) \rightarrow \Hom(p(\mathcal{E}_1), \dots, p(\mathcal{E}_n); p(\mathcal{F})) \]
have to be Grothendieck opfibrations (with {\em discrete fibers}) and composition has to be a morphism of Grothendieck opfibrations.
\end{DEF}

There are several other, partly more general, definitions of an (op-)fibration with 2-categorical fibers which we will not need in this section. We will discuss them in a subsequent article~\cite{Hor15}. 

Note that for (op-)fibrations {\em with 1-categorical fibers} the composition is automatically a morphism of Grothendieck opfibrations. 

\begin{PAR}
An opfibration $p: \mathcal{D} \rightarrow \mathcal{S}$ of 2-multicategories {\em with 1-categorical fibers} is in particular (forgetting 2-morphisms) a usual opfibration.
The additional datum, which makes it into a 2-opfibration is the following: 
For each 2-morphism $\mu: f \Rightarrow g$ in $\mathcal{S}$ a map of sets (the 2-push-forward):
\[ p^*(\mu): \Hom_f(X_1, \dots, X_n; Y) \rightarrow \Hom_g(X_1, \dots, X_n; Y) \]
such that
\begin{eqnarray*}
 p^*(\id_f)(\beta) &=& \beta \\
 p^*(\mu) \circ p^*(\nu) &=& p^*(\mu \circ \nu)  
\end{eqnarray*}
(composition of 2-coCartesian morphisms are 2-coCartesian)
and
\begin{eqnarray*}
 p^*(p(\alpha) \ast \mu)(\alpha \circ \xi) &=& \alpha \circ (p^*(\mu)(\xi))  \\
 p^*(\mu \ast p(\alpha))(\xi \circ \alpha) &=& (p^*(\mu)(\xi)) \circ \alpha 
\end{eqnarray*}
(1-composition maps coCartesian morphisms to coCartesian morphisms).

The 2-morphisms between $\alpha$ and $\beta$ in $\mathcal{D}$ lying over $f$, resp.\@ $g$ in $\mathcal{S}$ can be reconstructed from the datum $p^*$ as 
\[ \Hom(\alpha, \beta) = \{\mu\in \Hom(f, g)\ |\ p^*(\mu)(\alpha) = \beta \}. \]
\end{PAR}

\begin{PAR}
With a pseudo-functor
\[ \Psi: \mathcal{S} \rightarrow \mathcal{MCAT}^{2-\op} \]
where $\mathcal{S}$ is any strict 2-multicategory, 
we associate the opfibration
\[  \mathcal{D}_\Psi \rightarrow \mathcal{S}. \]
The objects of $\mathcal{D}_\Psi$ are pairs
\[ (S, X \in \Psi(S)) \]
in which $S$ is an object of $\mathcal{S}$.
The 1-morphisms 
\[ \xymatrix{
(S_1, X_1) \ar@{-}[rd] \\
\ar@{}[r] \vdots & (f, \alpha) \ar[r] &  (T, Y)  \\
(S_n, X_n) \ar@{-}[ru] \\
}\]
are pairs of (multi-)morphisms
\[ f \in \Hom(S_1, \dots, S_n; T) \qquad \alpha: \Psi(f)(X_1, \dots, X_n) \rightarrow Y. \]
The 2-morphisms 
\[ (f, \alpha) \Rightarrow (f', \alpha') \]
are given by 2-morphisms $\mu: f \Rightarrow f'$ such that $\alpha \circ (\Psi(\mu)(X)) = \alpha'$.

The fiber\footnote{i.e.\@ the 2-category of those objects, morphisms, and 2-morphisms which $\Psi$ maps to $S$, $\id_S$, and $\id_{\id_S}$, respectively} of $\mathcal{D}_\Psi \rightarrow \mathcal{S}$ over $S$ is actually a 1-category, namely precisely the category $\Psi(S)$.
\end{PAR}

We have the following generalization of Lemma~\ref{LEMMAPROPMULTI}, 2.:
\begin{LEMMA}\label{LEMMAPROPMULTI2}
Given (op)fibered 2-multicategories $p: \mathcal{C} \rightarrow \mathcal{D}$ and $q: \mathcal{D} \rightarrow \mathcal{E}$ then the composition $q\circ p$ is an (op)fibered 2-multicategory, too.  A 1-morphism $\alpha$ is (co)Cartesian for $q\circ p$ if and only if $\alpha$ is (co)Cartesian for $p$ and $p(\alpha)$ is (co)Cartesian for $q$.
\end{LEMMA}

\begin{BEISPIEL}\label{EXMULTICATS}
Let $\mathcal{S}$ be a usual category. Then $\mathcal{S}$ may be turned into a symmetric
multicategory by setting 
\[ \Hom(X_1, \dots, X_n; Y) := \Hom(X_1; Y) \times \cdots \times \Hom(X_n; Y). \]
If $\mathcal{S}$ has coproducts, then $\mathcal{S}$ (with this multi-category structure) is opfibered over $\{\cdot\}$.
Let $p: \mathcal{D} \rightarrow \mathcal{S}$ be an opfibered (usual) category. Any object $X$ induces a canonical functor of multicategories $x: \{\cdot\} \rightarrow \mathcal{S}$ with image $X$, hence the fibers of an opfibered multicategory $p: \mathcal{D} \rightarrow \mathcal{S}$ are monoidal and
the datum $p$ is equivalent to giving a monoidal structure on the fibers such that the push-forwards $f_\bullet$ are monoidal functors and such that the compatibility morphisms between them are morphisms of monoidal functors. This is called a covariant monoidal pseudo-functor in \cite[(3.6.7)]{LH09}.
\end{BEISPIEL}
\begin{BEISPIEL}\label{EXMULTICATSOP}
Let $\mathcal{S}$ be a usual category.  Then $\mathcal{S}^{\mathrm{op}}$ may be turned into a symmetric 
multicategory (or equivalently $\mathcal{S}$ into a symmetric opmulticategory) by setting 
\[ \Hom(X_1, \dots, X_n; Y) := \Hom(Y; X_1) \times \cdots \times \Hom(Y; X_n). \]
If $\mathcal{S}$ has products then $\mathcal{S}^{\mathrm{op}}$ (with this multi-category structure)  is opfibered over $\{\cdot\}$.
Let $p: \mathcal{D} \rightarrow \mathcal{S}^{\mathrm{op}}$ be an opfibered (usual) category. Then an opfibered multicategory structure on $p$, w.r.t.\@ this multicategory structure on $\mathcal{S}^{\op}$, is equivalent to a monoidal structure on the fibers such that pull-backs $f^*$ (along morphisms in $\mathcal{S}$) are monoidal functors and such that the compatibility morphisms between them are morphisms of monoidal functors. This is called a contravariant monoidal pseudo-functor in \cite[(3.6.7)]{LH09}.
\end{BEISPIEL}

\begin{DEF}\label{DEFMULCOR}
The point is that the notion of (op)fibered multicategory is {\em not restricted to} the situation of Examples \ref{EXMULTICATS} and \ref{EXMULTICATSOP}. 
Let $\mathcal{S}$ be a category with fiber products and define $\mathcal{S}^{\mathrm{cor}}$, denoted the {\bf symmetric 2-multicategory of correspondences in $\mathcal{S}$} to be the symmetric 2-multicategory having the same objects as $\mathcal{S}$, and where the category of
morphisms $\Hom(S_1, \dots, S_n; T)$ is the category of objects
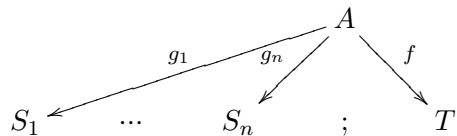
\[ \xymatrix{ 
 &&&  \ar[llld]_{g_1} A \ar[ld]_{g_n} \ar[rd]^{f} &\\
 S_1 & \cdots & S_n & ; &  T   }\]
and where the 2-morphisms $(A, f, g_1, \dots, g_n) \Rightarrow (A', f', g_1', \dots, g_n')$ are {\em iso}morphisms $A \rightarrow A'$ compatible with $f, f'$ and $g_1, g_1', \dots, g_n, g_n'$.

Composition is given by:
 
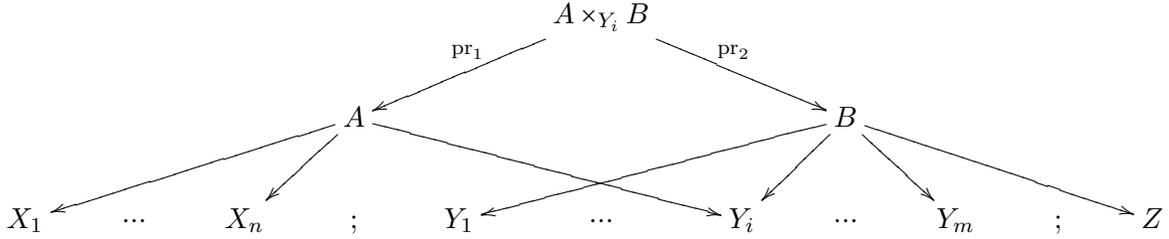
\[ \xymatrix{
&&&&& A \times_{Y_i} B  \ar[lld]_{\pr_1} \ar[rrd]^{\pr_2} \\
&&& A \ar[llld] \ar[ld] \ar[rrrd] &&&&  B \ar[llld] \ar[ld]  \ar[rd] \ar[rrrd] \\
X_1 & \cdots & X_n & ; & Y_1 & \cdots & Y_i & \cdots & Y_m & ; & Z
} \]
where strictly associative fiber products have been chosen in $\mathcal{S}$.
 
This 2-multicategory is representable
(i.e.\@ opfibered over $\{\cdot\}$), closed (i.e.\@ fibered over $\{\cdot\}$) and self-dual, with tensor product and internal hom {\em both} given by $\times$ and having as unit
the final object of $\mathcal{S}$. 
\end{DEF}

\begin{DEF}\label{DEF6FU}
Let $\mathcal{S}$ be a category with fiber products. A {\bf (symmetric) Grothendieck six-functor-formalism} on $\mathcal{S}$ is a bifibered (symmetric) 2-multicategory with 1-categorical fibers
\[ p: \mathcal{D} \rightarrow \mathcal{S}^{\mathrm{cor}}. \]
\end{DEF}
\begin{PAR}\label{CONCRETE6FU}
We have a morphism of opfibered (over $\{\cdot\}$) symmetric multicategories $\mathcal{S}^{\mathrm{op}} \rightarrow \mathcal{S}^{\mathrm{cor}}$ where $\mathcal{S}^{\mathrm{op}}$ is equipped with the symmetric multicategory structure as in \ref{EXMULTICATSOP}. However there is no reasonable morphism of opfibered multicategories $\mathcal{S} \rightarrow \mathcal{S}^{\mathrm{cor}}$. (There is no compatibility involving only `$\otimes$' and `$!$'.) From a Grothendieck six-functor-formalism we get  operations $g_*$, $g^*$ as the pull-back and the push-forward along the correspondence
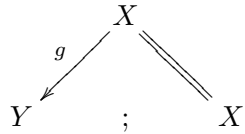
\[ \xymatrix{ 
 & \ar[ld]_g X \ar@{=}[rd] &\\
 Y & ; & X   }\]
 We get $f^!$ and $f_!$ as the pull-back and the push-forward along the correspondence
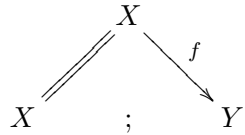
\[ \xymatrix{ 
 & \ar@{=}[ld] X \ar[rd]^f &\\
 X & ; & Y   }\]
 
We get the monoidal product $A \otimes B$ for objects $A, B$ above $X$ as the target of any Cartesian morphism 
$\otimes$ over the correspondence
\[ \xi_X = 
 \left(\vcenter{\xymatrix{
& & X \ar@{=}[lld] \ar@{=}[ld] \ar@{=}[rd] &  \\
X & X & ; & X }}\right)
\]

Alternatively, we have
\[ A \otimes B = \Delta^* (A \boxtimes B) \]
where $\Delta^*$ is the push-forward along the correspondence
\[ \left(\vcenter{\xymatrix{
& X \ar[ld]_\Delta \ar@{=}[rd]^f & \\
X \times X & ; & X }}\right)
 \]
induced by the canonical $\xi_X \in \Hom(X,X; X)$, and where $\boxtimes$ is the absolute monoidal product which exists because by Lemma~\ref{LEMMAPROPMULTI2} the composition $\mathcal{D} \rightarrow \{\cdot\}$ is opfibered, too, i.e.\@ $\mathcal{D}$ is monoidal.
\end{PAR}

\begin{PAR}
It is easy to derive from the definition of bifibered multicategory over $\mathcal{S}^{\mathrm{cor}}$ that
the absolute monoidal product $A \boxtimes B$ can be reconstructed from the fiber-wise product as $\pr_1^*A  \otimes \pr_2^*B$ on $X \times Y$, whereas the
absolute $\mathbf{HOM}(A, B)$ is given by $\mathcal{HOM}(\pr_1^*A, \pr_2^! B)$ on $X \times Y$.
In particular $DA:=\mathbf{HOM}(A,1)$ is given by $\mathcal{HOM}(A, \pi^! 1)$ for $\pi: X \rightarrow \cdot$ being the final morphism.
\end{PAR}

\begin{LEMMA}\label{LEMMA6FU}
Given a Grothendieck six-functor-formalism on $\mathcal{S}$
\[ p: \mathcal{D} \rightarrow \mathcal{S}^{\mathrm{cor}} \]
for the six operations as extracted in \ref{CONCRETE6FU} there exist naturally the following compatibility isomophisms:
\begin{center}
\begin{tabular}{rlll}
& left adjoints & right adjoints \\
\hline
$(*,*)$ & $(fg)^* \iso g^* f^*$ & $(fg)_* \iso f_* g_*$ &\\
$(!,!)$ & $(fg)_! \iso f_! g_!$ & $(fg)^! \iso g^! f^!$ &\\ 
$(!,*)$ 
& $g^* f_! \iso F_! G^*$ & $G_* F^! \iso f^! g_*$ & \\
$(\otimes,*)$ & $f^*(- \otimes -) \iso f^*- \otimes f^* -$ & $f_* \mathcal{HOM}(f^*-, -) \iso \mathcal{HOM}(-, f_*-)$  & \\
$(\otimes,!)$ & $f_!(- \otimes f^* -) \iso  (f_! -) \otimes -$ & $f_* \mathcal{HOM}(-, f^!-) \iso \mathcal{HOM}(f_! -, -)$ & \\ 
& & $f^!\mathcal{HOM}(-, -) \iso \mathcal{HOM}(f^* -, f^!-)$ & \\
$(\otimes, \otimes)$ &  $(- \otimes -) \otimes - \iso - \otimes (- \otimes -)$ &  $\mathcal{HOM}(- \otimes -, -) \iso \mathcal{HOM}(-, \mathcal{HOM}(-, -))$ & 
\end{tabular}
\end{center}

Here $f, g, F, G$ are morphisms in $\mathcal{S}$ which, in the $(!,*)$-row, are related by the {\em Cartesian} diagram
\[ \xymatrix{ \cdot \ar[r]^G \ar[d]_F  & \cdot \ar[d]^f \\ \cdot \ar[r]_g & \cdot } \]
\end{LEMMA}
\begin{BEM}
In the right column the corresponding adjoint natural transformations are listed. In each case the left hand side natural isomorphism determines the right hand side one and conversely. (In the $(\otimes,!)$-case there are 2 versions of the 
commutation between the right adjoints; in this case any of the three isomorphisms determines the other two.)
The $(!,*)$-isomorphism (between left adjoints)  is called {\bf base change}, the $(\otimes, !)$-isomorphism is called the {\bf projection formula}, and the $(*, \otimes)$-isomorphism is usually part of the definition of a {\bf monoidal functor}. The $(\otimes, \otimes)$-isomorphism is the associativity of the tensor product and part of the definition of a monoidal category. The $(*,*)$-isomorphism, and the\@ $(!,!)$-isomorphism express that the corresponding functors arrange as a pseudo-functor with values in categories.
\end{BEM}
\begin{proof}
The existence of all isomorphisms is a consequences of the fact that the composition of coCartesian morphisms is coCartesian. 
For example, the projection formula $(\otimes,!)$ is derived from the following composition in $\mathcal{S}^{\mathrm{cor}}$:
\[ \left(\vcenter{\xymatrix{
& & Y \ar@{=}[lld] \ar@{=}[ld] \ar@{=}[rd] &  \\
Y & Y & ; & Y }}\right)
\circ_1
\left(\vcenter{\xymatrix{
& X \ar@{=}[ld] \ar[rd]^f & \\
X & ; & Y }}\right)
\cong
\left(\vcenter{\xymatrix{
& & X \ar@{=}[lld] \ar[ld]^f \ar[rd]^f & \\
X & Y & ; & Y }}\right), \]
where $\circ_1$ means that we compose w.r.t.\@ the first slot.

The ``monoidality of $f^*$'' $(*,\otimes)$ is derived from the following composition in $\mathcal{S}^{\mathrm{cor}}$:
\[ 
\left(\vcenter{\xymatrix{
& X \ar[ld]^f \ar@{=}[rd] & \\
Y & ; & X }}\right)
\circ
\left(\vcenter{\xymatrix{
& & Y \ar@{=}[lld] \ar@{=}[ld] \ar@{=}[rd] &  \\
Y & Y & ; & Y }}\right)
\cong
\left(\vcenter{\xymatrix{
& & X \ar[lld]_f \ar[ld]^f \ar@{=}[rd] & \\
Y & Y & ; & X }}\right). \]
Base change $(!,*)$ is derived from:
\[
\left(\vcenter{\xymatrix{
& X \ar[ld]_g \ar@{=}[rd] & \\
A & ; & X }}\right)
\circ
\left(\vcenter{\xymatrix{
& Y \ar@{=}[ld] \ar[rd]^f & \\
Y & ; & A }}\right)
\cong
\left(\vcenter{\xymatrix{
& Y\times_A X \ar[ld]_F \ar[rd]^G & \\
Y & ; & X }}\right).
\]
\end{proof}

All compatibilities between these isomorphisms can be derived, too. Each of these compatibilities corresponds to an associativity relation in the fibered multicategory.
One can also axiomatize the properties of the morphism $f_! \rightarrow f_*$ that often accompanies a six-functor-formalism. 
Can one give a finite list of compatibility diagrams from which all the others would follow?

\begin{PAR}
The goal and motivation for this research is, as said in the introduction, to define (and to construct in reasonable contexts) a {\em derivator version} of a Grothendieck six-functor-formalism, i.e.\@ a
fibered multiderivator 
\[ \DD \rightarrow \SSS^{\mathrm{cor}} \]
where $\SSS^{\mathrm{cor}}$ is the pre-2-multiderivator associated with the 2-category $\mathcal{S}^{\mathrm{cor}}$. We will give the definition of a pre-2-multiderivator and of a fibered derivator over such in a subsequent article~\cite{Hor15}. 
\end{PAR}

\subsection{Localization of multicategories}

\begin{PROP} \label{PROPMULTILOC}
Let $\mathcal{D}$ be a (symmetric, braided) multicategory and let $\mathcal{W}$ be a subclass of 1-ary morphisms. Then there exists a  (symmetric, braided) multicategory $\mathcal{D}[\mathcal{W}^{-1}]$, which is not necessarily locally small, together with a functor $\iota: \mathcal{D} \rightarrow \mathcal{D}[\mathcal{W}^{-1}]$ of (symmetric, braided)  multicategories with the property that $\iota(w)$ is an isomorphism for all $w \in \mathcal{W}$ and which is universal w.r.t.\@ this property. 
\end{PROP}
\begin{proof}
This construction is completely analogous to the construction for usual categories. Morphisms $\Hom(X_1, \dots, X_n; Y)$ are formal compositions
of $i$-ary morphisms in $\mathcal{D}$ and formal inverses of morphisms in $\mathcal{W}$, for example:
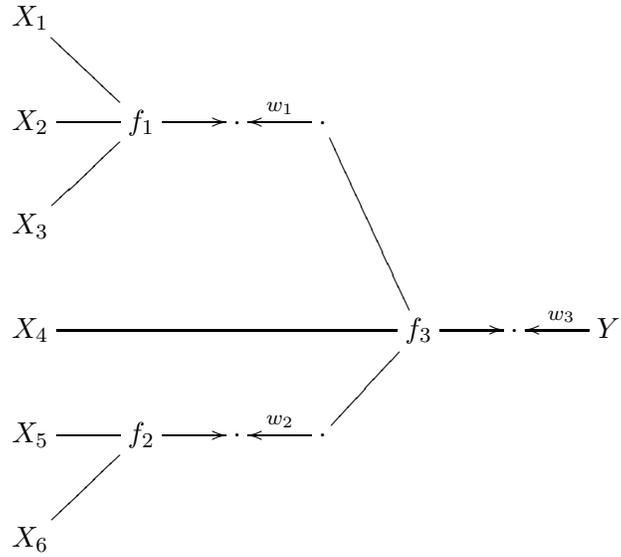
\[\xymatrix{
X_1 \ar@{-}[rd] & \\
X_2 \ar@{-}[r] & f_1 \ar[r] & \cdot & \cdot \ar[l]_{w_1} \ar@{-}[ddr] \\
X_3 \ar@{-}[ru] & \\
X_4 \ar@{-}[rrrr] & & & & f_3 \ar[r] & \cdot &  \ar[l]_{w_3} Y  \\
X_5 \ar@{-}[r] & f_2 \ar[r] & \cdot & \cdot \ar[l]_{w_2} \ar@{-}[ur] \\
X_6 \ar@{-}[ru] & 
}\]
More precisely: 
Morphisms are defined to be the class of lists of $n_i$-ary morphisms $f_i \in \Hom(X_{i,1}, \dots, X_{i,n_i}; Y_i)$, morphisms $w_i: Y_ i' \rightarrow Y_i$ in $\mathcal{W}$ and integers $k_i$ as follows
\[ (f_{1}, w_1), k_1, (f_2, w_2), k_2, \dots, k_{n-1}, (f_n, w_n) \]
such that $Y_i' = X_{i+1,k_i}$, modulo relations coming from commutative squares and forcing the $(\id, w_i)$ to become the left and right inverse of $(w_i, \id)$.
\end{proof}

\newpage
\bibliographystyle{abbrvnat}
\bibliography{6fu}

\end{document}